\documentclass[12pt]{amsart}
\usepackage{amscd,amssymb,amsthm,amsmath,amssymb,textcomp,supertabular,longtable,enumerate,rotating,mathrsfs,mathtools,hyperref,latexsym,amscd,amsbsy,mathrsfs}
\usepackage{pdflscape}
\usepackage{graphicx}
\usepackage{array}
\usepackage[matrix,arrow,curve]{xy}
\usepackage{xcolor}
\usepackage{longtable}
\usepackage{enumitem}

\usepackage{multicol}
\usepackage{tabulary}
\usepackage{colortbl}

\usepackage{geometry}
\usepackage[
    maxbibnames=999,    
    maxcitenames=1,     
    mincitenames=1,     
    style=apa,    
    url=false,
    doi=true,
]{biblatex}
\newtheorem{lemma}[equation]{Lemma}
\newtheorem{corollary}[equation]{Corollary}

\newtheorem*{calabiproblem*}{Calabi Problem}

\theoremstyle{definition}

\newtheorem{definition}[equation]{Definition}

\theoremstyle{remark}
\newtheorem{remark}[equation]{Remark}

\makeatletter\@addtoreset{equation}{section} \makeatother

\swapnumbers

\newtheoremstyle{dotless}{}{}{\rm}{}{\sc}{}{ }{}

\theoremstyle{dotless}

\newcommand{\DR}{\mathbb{R}} 
\newcommand{\DC}{\mathbb{C}}

\newcommand{\DP}{\mathbb{P}}
\newcommand{\DA}{\mathbb{A}}

\newcommand{\Pic}{\mathrm{Pic}}

\newcommand{\Aut}{\mathrm{Aut}}

\newtheorem*{theorem*}{Theorem.}
\newtheorem*{maintheorem*}{Main Theorem.}

\newtheorem*{corollary*}{Corollary.}

\author{Elena Denisova}

\title{$\delta$-invariants of Du Val del Pezzo surfaces of degree $1$}

\address{\emph{Elena Denisova}
\newline
\textnormal{School of Mathematics, The University of Edinburgh, Edinburgh EH9 3JZ, UK.}
\newline
\textnormal{\texttt{e.denisova@sms.ed.ac.uk}}}

\begin{document}

\maketitle
\begin{abstract}
In this article, we compute $\delta$-invariants of Du Val del Pezzo surfaces of degree $1$.
\end{abstract}
\section{Introduction}
\subsection{History and Results}
It is well known that a smooth Fano variety admits a~K\"ahler--Einstein metric if and only if it is $K$-polystable.
For del Pezzo surfaces, Tian and Yau proved that a smooth del Pezzo surface is $K$-polystable if and only if it is not the blow-up of $\mathbb{P}^2$ at one or two points (see~\cite{TianYau1987,Ti90}).
Later, Odaka, Spotti, and Sun determined which Du Val del Pezzo surfaces are $K$-stable in~\cite{OdakaSpottiSun16}.
Substantial progress has been made for Fano threefolds (see~\cite{Fano21, Liu23, CheltsovFujitaKishimotoPark23, LiuZhao24, GuerreiroGiovenzanaViswanathan23, Malbon24, CheltsovPark22, BelousovLoginov24, BelousovLoginov23, CheltsovFujitaKishimotoOkada23, Denisova24, Denisova23, CheltsovDenisovaFujita24}).
However, many questions remain open for Fano varieties in higher dimensions. 
In the case of threefolds, it has been observed that the problem often reduces to computing the $\delta$-invariants of (possibly singular) del Pezzo surfaces (see~\cite{Fano21, CheltsovDenisovaFujita24, CheltsovFujitaKishimotoOkada23}).
\par
In the previous parts of this series, we computed the $\delta$-invariants of Du Val del Pezzo surfaces of degree $\ge 2$.
In~\cite[Lemma~2.16]{Fano21}, it was shown that $\delta(X) = \frac{15}{7}$ when $X$ is a smooth del Pezzo surface of degree~$1$ and $|-K_X|$ contains a cuspidal curve, and $\delta(X) = \frac{12}{5}$ when $|-K_X|$ does not contain a cuspidal curve.
The proof of this theorem immediately implies the following corollary: for a Du Val del Pezzo surface $X$ of degree~$1$ and a smooth point $\mathcal{P} \in X$, we have
$
\delta_{\mathcal{P}}(X) \ge \frac{15}{7}.
$
\par
This work can be viewed as a generalization of $\alpha$-invariant computations carried out by I.~Cheltsov, D.~Kosta, J.~Park, and J.~Won in a series of papers~\cite{Cheltsov09, Cheltsov14, ParkWon10, ParkWon11}, since the $\delta$- and $\alpha$-invariants are related by the inequalities
\[
3\alpha(X) \ge \delta(X) \ge \frac{3\alpha(X)}{2}
\]
in the case of del Pezzo surfaces.
The singularity types of Du Val del Pezzo surfaces of degree~$1$ were classified in~\cite{Urabe83}.

The results on computing $\delta$-invariants of Du Val del Pezzo surfaces obtained in the earlier parts of this series, combined with the results of this article, confirm those of Odaka--Spotti--Sun~\cite{OdakaSpottiSun16} and also lead to new examples of $K$-stable singular Fano threefolds.

Let $X$ be a Du Val del Pezzo surface of degree~$1$. Then $X$ can be realized as a double cover
\[
X \xrightarrow{2:1} \DP(1,1,2),
\]
ramified along a sextic curve $R \in \DP(1,1,2)$.
In this article, we compute the $\delta$-invariants of Du Val del Pezzo surfaces of degree~$1$. 
We note that when $X$ has $\DA_7$ singularities, the $\delta$-invariant depends on whether $R$ is reducible or irreducible.
We prove that:

\begin{maintheorem*}
Let $X$ be the Du Val del Pezzo surface of degree $1$. Then the  $\delta$-invariant of $X$ is uniquely determined by the by the type of singularities on $X$  and unique elements of $|-K_X|$ containing each of singular points which is given in the following table:
{\renewcommand\arraystretch{1.2}
	\begin{longtable}{|c||c|}
		\hline
Type of singularity  & $\delta(X)$\\
\hline \hline
\shortstack{{ }\\$\DA_1$, $2\DA_1$, $3\DA_1$, $4\DA_1$, $5\DA_1$, $6\DA_1$\\
all elements of $|-K_X|$ containing singular points are nodal}  & \shortstack{{ }\\$2$\\{ }{ }} \\
\hline 
\shortstack{{ }\\$\DA_1$, $2\DA_1$, $3\DA_1$, $4\DA_1$, $5\DA_1$, $6\DA_1$\\
 some elements of $|-K_X|$ containing\\
 singular points are cuspidal } & \shortstack{{ }\\$\frac{9}{5}$\\{ }\\{ }}\\
 \hline 
 \shortstack{{ }\\
$\DA_2$, $\DA_2+\DA_1$, $\DA_2+2\DA_1$,  $\DA_2+3\DA_1$, $\DA_2+4\DA_1$,\\
$2\DA_2$, $2\DA_2+\DA_1$, $2\DA_2+2\DA_1$, $3\DA_2$, $3\DA_2+\DA_1$, $4\DA_2$\\
all elements of $|-K_X|$ containing $\DA_2$ singular points are nodal}  &  \shortstack{$\frac{12}{7}$\\{ }\\{ }} \\
\hline \
\shortstack{{ }\\$\DA_2$, $\DA_2+\DA_1$, $\DA_2+2\DA_1$,  $\DA_2+3\DA_1$, $\DA_2+4\DA_1$,\\
$2\DA_2$, $2\DA_2+\DA_1$, $2\DA_2+2\DA_1$, $3\DA_2$, $3\DA_2+\DA_1$, $4\DA_2$\\
some elements of $|-K_X|$ containing $\DA_2$ singular points are cuspidal}
& \shortstack{$\frac{3}{2}$\\{ }\\{ }}\\
\hline 
\shortstack{{ }\\$\DA_3$, $\DA_3+\DA_1$, $\DA_3+2\DA_1$, $\DA_3+3\DA_1$, $\DA_3+4\DA_1$,\\
$\DA_3+\DA_2$, $\DA_3+\DA_2+\DA_1$, $\DA_3+\DA_2+2\DA_1$,\\
$2\DA_3$, $2\DA_3+\DA_1$, $2\DA_3+2\DA_1$}& \shortstack{$\frac{3}{2}$\\{ }\\{ }}\\
\hline
$\DA_4$, $\DA_4+\DA_1$, $\DA_4+2\DA_1$, $\DA_4+\DA_2$, $\DA_4+\DA_2+\DA_1$, $\DA_4+\DA_3$, $2\DA_4$  & {$\frac{4}{3}$}\\\hline
$\DA_5$, $\DA_5+\DA_1$, $\DA_5+2\DA_1$, $\DA_5+\DA_2$, $\DA_5+\DA_2+\DA_1$  &  $\frac{6}{5}$ \\
\hline
$\DA_6$, $\DA_6+\DA_1$  &  $\frac{9}{8}$ \\
\hline
$\DA_7$  and $R$ irreducible &  $\frac{18}{17}$\\
\hline
$\DA_7$, $\DA_7+\DA_1$ and $R$ reducible & $1$\\ 
\hline
$\DA_8$, $\mathbb{D}_4$, $\mathbb{D}_4+\DA_1$, $\mathbb{D}_4+2\DA_1$, $\mathbb{D}_4+3\DA_1$, $\mathbb{D}_4+\DA_2$, $\mathbb{D}_4+\DA_3$, $2\mathbb{D}_4$ & $1$
\\\hline
$\mathbb{D}_5$, $\mathbb{D}_5+\DA_1$, $\mathbb{D}_5+2\DA_1$, $\mathbb{D}_5+\DA_2$, $\mathbb{D}_5+\DA_3$ &  $\frac{6}{7}$\\ 
\hline
$\mathbb{D}_6$, $\mathbb{D}_6+\DA_1$, $\mathbb{D}_6+2\DA_1$ &  $\frac{3}{4}$ \\ 
\hline
$\mathbb{D}_7$ &  $\frac{2}{3}$ \\ 
\hline
$\mathbb{D}_8$, $\mathbb{E}_6$, $\mathbb{E}_6+\DA_1$, $\mathbb{E}_6+\DA_2$  &  $\frac{3}{5}$
\\\hline
$\mathbb{E}_7$, $\mathbb{E}_7+\DA_1$  & $\frac{3}{7}$ \\
\hline
$\mathbb{E}_8$  &  $\frac{3}{11}$ \\\hline
	\end{longtable}}
  \end{maintheorem*}
 \noindent {\bf Acknowledgments:} I am grateful to my supervisor Professor Ivan Cheltsov for the introduction to the topic and continuous support.  
 \subsection{Applications.}  Let  $X$ be a del Pezzo surface of degree $1$ with at  most Du Val singularities. Let $S$ be a weak resolution of $X$. We will call an image on $X$ of a $(-1)$-curve in $S$  {\bf a line} as was done in \cite{CheltsovProkhorov21}. The immediate corollaries from Main Theorem are:
\begin{corollary}
Let $X$ be a Du Val del Pezzo surface of degree $1$ with $\DA_n$ or $\mathbb{D}_4$ singularities then $X$ is $K$-semi-stable.
\end{corollary}
\begin{proof}
    For such $X$ have $\delta(X)\ge 1$. Thus, $X$ is $K$-semi-stable by \cite[Theorem 1.59]{Fano21}.
\end{proof}

\begin{corollary}[\cite{OdakaSpottiSun16}]
Let $X$ be a Du Val del Pezzo surface of degree $1$ with at most $\DA_6$ singularities or a Du Val del Pezzo surface of degree $1$ with $\DA_7$ singularity and irreducible ramification divisor $R$ then $X$ is $K$-stable. Moreover, $\Aut(X)$ is finite.
\end{corollary}

\begin{proof}
    For such $X$ have $\delta(X)> 1$. Thus, $X$ is $K$-stable. By \cite[Corollary 1.3]{BlumXu19} $\Aut(X)$ is finite for $K$-stable $X$.
\end{proof}

\noindent There are also some applications in the case of threefolds. Smooth Fano threefolds over  $\DC$ were classified in \cite{Is77,Is78,MoMu81,MoMu03} into $105$ families. The detailed description of these families can be found in \cite{Fano21} where the  problem  to find all K-polystable smooth Fano threefolds in each family was posed. The output of this paper, give some alternative proofs for this problem as well as some proofs in case of singular Fano threefolds. We know (\cite{Fujita19,Li17}) that the Fano threefold $\mathbf{X}$ is $K$-stable if and only if for every prime divisor $\mathbf{E}$ over $\mathbf{X}$ we have
$$
\beta(\mathbf{E})=A_{\mathbf{X}}(\mathbf{E})-S_{\mathbf{X}}(\mathbf{E})>0
$$
 where $A_{\mathbf{X}}(\mathbf{E})$ is the~log discrepancy of the~divisor $\mathbf{E}$ and
$S_{\mathbf{X}}\big(\mathbf{E}\big)=\frac{1}{(-K_{\mathbf{X}})^3}\int\limits_0^{\infty}\mathrm{vol}\big(-K_{\mathbf{X}}-u\mathbf{E}\big)du.$
To show this, we fix a prime divisor $\mathbf{E}$ over~$\mathbf{X}$.
Then we set $Z=C_{\mathbf{X}}(\mathbf{E})$. Let $Q$ be a general point in $Z$. 
Following \cite{AbbanZhuang,Fano21} denote
$$
\delta_Q\big(X,W^{X}_{\bullet,\bullet}\big)=\inf_{\substack{F/X\\Q\in C_{X}(F)}}\frac{A_X(F)}{S\big(W^{X}_{\bullet,\bullet};F\big)}\text{ and }\delta_Q\big( \mathbf{X}\big)=\inf_{\substack{\mathbf{F}/\mathbf{X}\\Q\in C_{\mathbf{X}}(\mathbf{F})}}\frac{A_{\mathbf{X}}(\mathbf{F})}{S_{\mathbf{X}}(\mathbf{F})}
$$
where the first infimum is taken by all prime divisors $F$ over the surface $X$ whose center on $X$ contains $Q$ and the second infimum is taken by all prime divisors $\mathbf{F}$ over the threefold $\mathbf{X}$ whose center on $\mathbf{X}$ contains $Q$.

\subsubsection{Family 1.11 (Del Pezzo Threefold of degree 1)} 
 Let $\mathbf{V}$ be a  Fano threefold with canonical Gorenstein singularities such that $-K_{\mathbf{V}} \sim 2H$
for some $H \in \Pic(\mathbf{V})$ with $H^3=1$. Then 
$\mathbf{V}$  is a sextic hypersurface in $\DP(1, 1, 1, 2, 3)$ and a del Pezzo threefold of degree $1$. A general element in $|H|$  is a Du Val del Pezzo surface of degree $1$ and if $\mathbf{V}$ has isolated singularities then a general surface in $|H|$ is a smooth. 
\begin{remark}
If $\mathbf{V}$ is smooth then $\mathbf{V}$ is a smooth Fano threefold in Family  1.11. and all smooth Fano threefolds in this family can be obtained this way. Every smooth element in this family is known to be $K$-stable \cite{Fano21}.
\end{remark}
\noindent Main Theorem gives the following corollary:
\begin{corollary}
Suppose that for any point $Q$ on $\mathbf{V}$ there exists an  element $X\in|H|$ such that $Q\in X$ and $X$ has at most $\DA_2$ singularities then  $\mathbf{V}$ is $K$-stable.
\end{corollary}

\begin{proof}
Suppose $X$ is an irreducible element of $|H|$ then $S_{\mathbf{V}}(X)<1$. As explained above
we fix a prime divisor $\mathbf{E}$ over~$\mathbf{V}$.
Then we set $Z=C_{\mathbf{V}}(\mathbf{E})$ and  if $\beta(\mathbf{E})\leqslant 0$,  then
 $\delta_Q(X,W^{X}_{\bullet,\bullet})\leqslant 1$. Let $Q$ be a general point in $Z$, Let $X$ be the general element of $|H|$ that contains $Q$.  The divisor $-K_{\mathbf{V}}-uX$ is nef if and only if $u\le 2$ and the Zariski Decomposition is given by by $P(u)=-K_{\mathbf{V}}-uX\sim(2-u)X$ and $N(u)=0$ for $u\in[0,2]$.
By \cite[Corollary 1.110]{Fano21} for any divisor $F$ such that $Q\in C_{X}(F)$ over $X$ we get:
\begin{align*}
S\big(W^{X}_{\bullet,\bullet}&;F\big)=\frac{3}{(-K_{\mathbf{V}})^3}\Bigg(\int_0^\tau\big(P(u)^{2}\cdot X\big)\cdot\mathrm{ord}_{Q}\Big(N(u)\big\vert_{X}\Big)du+\int_0^\tau\int_0^\infty \mathrm{vol}\big(P(u)\big\vert_{X}-vF\big)dvdu\Bigg)=\\
&= \frac{3}{8}\int_0^\tau\int_0^\infty \mathrm{vol}\big(P(u)\big\vert_{X}-vF\big)dvdu=
\frac{3}{8}\int_0^2(2-u)^3\int_0^\infty\mathrm{vol}\big(-K_{X}-wF\big)dwdu=\\
&= \frac{3}{8}\int_0^2(2-u)^3\Big(\int_0^\infty\mathrm{vol}\big(-K_{X}-wF\big)dw\Big)du=\frac{3}{8}\int_0^2(2-u)^3S_{X}(F)du= \frac{3}{2}S_{X}(F)\le \frac{3}{2} \frac{A_{X}(F)}{\delta_Q(X)}
\end{align*}
We get that $\delta_Q(\mathbf{V})\ge \frac{2}{3}\delta_Q(X)$.  For $X$ with at most $\DA_2$-singularities we have $\delta_Q(X)\ge \frac{3}{2}$.   If $Q$ is a singular point and there exists an element $X$ of $|H|$ with  $\delta_Q(X)= \frac{3}{2}$ then  $\frac{A_{\mathbf{X}}(\mathbf{E})}{S_{\mathbf{X}}(\mathbf{E})}>\min\Big\{\frac{1}{S_{\mathbf{X}}(X)},\delta_Q\big(X,W^{X}_{\bullet,\bullet}\big)\Big\}$ from \cite[Corollary 1.108.]{Fano21} and otherwise we choose $X$ with   $\delta_Q(X)> \frac{3}{2}$  so $\delta_Q(\mathbf{V})>1$ if  $X$ has at most $\DA_2$-singularities and the result follows.
\end{proof}

\subsubsection{Family 2.1} 
Let $\mathbf{V}$ be a  Fano threefold with canonical Gorenstein singularities such that $-K_{\mathbf{V}} \sim 2H$
for some $H \in \Pic(\mathbf{V})$ with $H^3=1$. Then 
$\mathbf{V}$  is a sextic hypersurface in $\DP(1, 1, 1, 2, 3)$ and a del Pezzo threefold of degree $1$.  
Let $S_1$ and $S_2$ be two distinct surfaces in the linear system $|H|$, and let $\mathcal{C} = S_1 \cap S_2$. Suppose that the curve $\mathcal{C}$ is smooth. Then $\mathcal{C}$ is an elliptic curve by the adjunction formula. Let $\pi : \mathbf{X}\to \mathbf{V}$ be the blow up of the curve $\mathcal{C}$, and let $E$ be the $\pi$-exceptional surface. We have the following commutative diagram:
$$\xymatrix{
&\mathbf{X}\ar[dl]_{\pi} \ar[dr]^{\phi}&\\
\mathbf{V}\ar@{-->}[rr]& &\DP^1
}$$
Where $\mathbf{V} \dashrightarrow \DP^1$
is the rational map given by the pencil that is generated by $S_1$ and $S_2$,
and $\phi$ is a fibration into del Pezzo surfaces of degree $1$. 
\begin{remark}
If $R$ is smooth then $\mathbf{X}$ is a smooth Fano threefold in Family  2.1. and all smooth Fano threefolds in this family can be obtained this way. Every smooth Fano threefold in this family is known to be $K$-stable  \cite{CheltsovDenisovaFujita24}.
\end{remark}
\noindent Main Theorem gives the following corollary:
\begin{corollary}
    If  every fiber $X$ of $\phi$   at most $\mathbb{D}_4$ singularities, then $\mathbf{X}$ is $K$-stable.
\end{corollary}
\begin{proof}
If $X$ is an irreducible fiber of $p_1$ then we have $S_{\mathbf{X}}(X)<1$. We now fix a prime divisor $\mathbf{E}$ over~$\mathbf{X}$.
Then we set $Z=C_{\mathbf{X}}(\mathbf{E})$. Let $Q$ be the point on $Z$. let $X$ be the fiber of $\phi$ that passes through $Q$.  Then $-K_{\mathbf{X}}-uX$ is nef if and only if $u\le 2$ and the Zariski Decomposition is given by 
$$P(u)=
\begin{cases}
    -K_{\mathbf{X}}-uX\sim (2-u)X+E\text{ if }u\in[0,1],\\
    -K_{\mathbf{X}}-uX-(u-1)E\sim(2-u)\pi^*(H)\text{ if }u\in[1,2],
\end{cases}
\text{ and }
N(u)=
\begin{cases}
    0\text{ if }u\in[0,1],\\
    (u-1)E\text{ if }u\in[1,2],
\end{cases}
$$
We apply Abban-Zhuang method to prove that $Q\not \in E\cong \mathcal{C}\times \DP^1$.
By \cite[Corollary 1.110]{Fano21} for any divisor $F$ such that $Q\in C_{X}(F)$ over $X$   we get:

{\allowdisplaybreaks\begin{align*}
S\big(W^{X}_{\bullet,\bullet}&;F\big)=\frac{3}{(-K_{\mathbf{X}})^3}\Bigg(\int_0^\tau\big(P(u)^{2}\cdot X\big)\cdot\mathrm{ord}_{Q}\Big(N(u)\big\vert_{X}\Big)du+\int_0^\tau\int_0^\infty \mathrm{vol}\big(P(u)\big\vert_{X}-vF\big)dvdu\Bigg)=\\
&= \frac{3}{4}\int_0^\tau\int_0^\infty \mathrm{vol}\big(P(u)\big\vert_{X}-vF\big)dvdu=\\
&= \frac{3}{4}\Bigg(\int_0^1\int_0^\infty \mathrm{vol}\big(-K_{X}-vF\big)dvdu+\int_1^2\int_0^\infty \mathrm{vol}\big(-K_{X}-(u-1)E|_X-vF\big)dvdu\Bigg)=\\
&= \frac{3}{4}\Bigg(\int_0^\infty \mathrm{vol}\big(-K_{X}-vF\big)dv+\int_1^2 (2-u)^3 \int_0^\infty \mathrm{vol}\big(-K_X-(u-1)E|_X-vF\big)dv\Bigg)=\\
&= \frac{3}{4}\Bigg(\int_0^\infty \mathrm{vol}\big(-K_{X}-vF\big)dv+
\int_1^2 (2-u)^3\int_0^\infty \mathrm{vol}\big(-K_{X}-wF\big)dw du\Bigg)=\\
&= \frac{3}{4}\Bigg(\int_0^\infty \mathrm{vol}\big(-K_{X}-vF\big)dv + \int_1^2 (2-u)^3\int_0^\infty \mathrm{vol}\big(-K_{X}-wF\big)dw du\Bigg)=\\
&=\frac{3}{4}\Bigg(S_{X}(F) +  \frac{1}{4}\cdot S_{X}(F)\Bigg)=\frac{15}{16}S_{X}(F)\le \frac{15}{16}\cdot \frac{A_{X}(F)}{\delta_Q(X)}
\end{align*}}
We see that $\delta_Q(\mathbf{X})\ge \frac{16}{15}\delta_Q(X)$.  
Thus, by Main Theorem if every fiber of $p_1$ has at most $\mathbb{D}_4$ singularities the result follows.
\end{proof}
\section{Proof of Main Theorem via  Kento Fujita’s formulas}
Let $X$ be a Du Val del Pezzo surface, and let $S$ be a  minimal resolution of $X$. Let  $f\colon\widetilde{X}\to X$ be a birational morphism,
let $E$ be a prime divisor in $\widetilde{X}$.
We say that $E$ is a prime divisor \emph{over} $X$.
If~$E$~is~\mbox{$f$-exceptional}, we say that $E$ is an~exceptional invariant prime divisor \emph{over}~$X$.
We will denote the~subvariety $f(E)$ by $C_X(E)$. 
Let \index{$S_X(E)$}
$$
S_X(E)=\frac{1}{(-K_X)^2}\int_{0}^{\tau}\mathrm{vol}(f^*(-K_X)-vE)dv\text{ and }A_X (E) = 1 + \mathrm{ord}_E(K_{\widetilde{X}} - f^*(K_X)),
$$
where $\tau=\tau(E)$ is the~pseudo-effective threshold of $E$ with respect to $-K_X$.
Let $Q$ be a point in $X$. We can define a local $\delta$-invariant and a global $\delta$-invariant now
$$
\delta_Q(X)=\inf_{\substack{E/X\\ Q\in C_X(E)}}\frac{A_X(E)}{S_X(E)}\text{ and }\delta(X)=\inf_{Q\in X}\delta_Q(X)
$$
where the~infimum runs over all prime divisors $E$ over the surface $X$ such that $Q\in C_X(E)$. Similarly, for the  surface $S$ and a point $P$ on $S$ we define:
$$
\delta_P(S)=\inf_{\substack{F/S\\ P\in C_S(F)}}\frac{A_S(F)}{S_S(F)}
\text{
and }\delta(S)=\inf_{P\in S}\delta_P(S)$$
where $S_S(F)$ and $A_S(F)$ are defined as $S_X(E)$ and $A_X(E)$ above. Note that it is clear that
$$\delta(X)=\delta(S)\text{ and }\delta_Q(X)=\inf_{P: Q=f(P)}\delta_P(S)$$
Several results can help us to estimate $\delta$-invariants.
Let $C$ be a smooth curve on $S$ containing $P$. 
Set
$$
\tau(C)=\mathrm{sup}\Big\{v\in\mathbb{R}_{\geqslant 0}\ \big\vert\ \text{the divisor  $-K_S-vC$ is pseudo-effective}\Big\}.
$$
For~$v\in[0,\tau]$, let $P(v)$ be the~positive part of the~Zariski decomposition of the~divisor $-K_S-vC$,
and let $N(v)$ be its negative part. 
Then we set $$
S\big(W^C_{\bullet,\bullet};P\big)=\frac{2}{K_S^2}\int_0^{\tau(C)} h(v) dv,
\text{ where }
h(v)=\big(P(v)\cdot C\big)\times\big(N(v)\cdot C\big)_P+\frac{\big(P(v)\cdot C\big)^2}{2}.
$$
It follows from {\cite[Theorem 1.7.1]{Fano21}} that:
\begin{equation}\label{estimation1}
    \delta_P(S)\geqslant\mathrm{min}\Bigg\{\frac{1}{S_S(C)},\frac{1}{S(W_{\bullet,\bullet}^C,P)}\Bigg\}.
\end{equation}
Unfortunately, using this approach we do not always get a good estimation. In this case, we can try to apply the generalization of this method. Let $\sigma\colon \widehat{S}\to S$ be a weighted blowup of the point $P$ on $S$. Suppose, in addition, that $\widehat{S}$ is a Mori Dream space Then
\begin{itemize}
\item the~$\sigma$-exceptional curve $E_P$ such that $\sigma(E_P)=P$, it is smooth and isomorphic to $\DP^1$,
\item the~log pair $(\widehat{S},E_P)$ has purely log terminal singularities.
\end{itemize}
Thus, the birational map $\sigma$ a plt blowup of a point $P$.
Write
$$
K_{E_P}+\Delta_{E_P}=\big(K_{\widehat{S}}+E_P\big)\big\vert_{E_P},
$$
where $\Delta_{E_P}$ is an~effective $\mathbb{Q}$-divisor on $E_P$ known as the~different of the~log pair $(\widehat{S},E_P)$.
Note that the~log pair $(E_P,\Delta_{E_P})$ has at most Kawamata log terminal singularities, and the~divisor $-(K_{E_P}+\Delta_{E_P})$ is $\sigma\vert_{E_P}$-ample.
\\Let $O$ be a point on $E_P$. 
Set
$$
\tau(E_P)=\mathrm{sup}\Big\{v\in\mathbb{R}_{\geqslant 0}\ \big\vert\ \text{the divisor  $\sigma^*(-K_S)-vE_P$ is pseudo-effective}\Big\}.
$$
For~$v\in[0,\tau]$, let $\widehat{P}(v)$ be the~positive part of the~Zariski decomposition of the~divisor $\sigma^*(-K_S)-vE_P$,
and let $\widehat{N}(v)$ be its negative part. 
Then we set $$
S\big(W^{E_P}_{\bullet,\bullet};O\big)=\frac{2}{K_{\widehat{S}}^2}\int_0^{\tau(E_P)} \widehat{h}(v) dv,
\text{ where }
\widehat{h}(v)=\big(\widehat{P}(v)\cdot E_P\big)\times\big(\widehat{N}(v)\cdot E_P\big)_O+\frac{\big(\widehat{P}(v)\cdot E_P\big)^2}{2}.
$$
Let
$A_{E_P,\Delta_{E_P}}(O)=1-\mathrm{ord}_{\Delta_{E_P}}(O)$.
It follows from {\cite[Theorem 1.7.9]{Fano21}} and {\cite[Corollary 1.7.12]{Fano21}} that
\begin{equation}
\label{estimation2}
\delta_P(S)\geqslant\mathrm{min}\Bigg\{\frac{A_S(E_P)}{S_S(E_P)},\inf_{O\in E_P}\frac{A_{E_P,\Delta_{E_P}}(O)}{S\big(W^{E_P}_{\bullet,\bullet};O\big)}\Bigg\},
\end{equation}
where the~infimum is taken over all points $O\in E_P$.
\\ We will apply \ref{estimation1} and \ref{estimation2} to all minimal resolutions $S$ such that $K_S^2=1$ in order to prove Main Theorem. In case  $X$ is smooth we have $S=X$. Small circles correspond to $(-1)$-curves and large circles correspond to $(-2)$-curves on dual graphs.

\section{Du Val del Pezzo Surfaces of Degree $1$}
\noindent In \cite[Lemma 2.16]{Fano21} it was proven that $\delta(X)=\frac{15}{7}$ when $X$ is a smooth del Pezzo surface of degree $1$ and $|-K_X|$ contains a cuspidal curve, and $\delta(X)=\frac{12}{5}$ when $X$ is a smooth del Pezzo surface of degree $1$ and $|-K_X|$ does not contain a cuspidal curve. 
\par
We consider a Del Pezzo surface $X$ of degree one with at worst Du Val singularities and denote its minimal resolution by $ \pi: S \to X $. The surface $X$ can be embedded as a degree six hypersurface in the weighted projective space $ \mathbb{P}(1,1,2,3) $, given by the equation
$$
w^2 = a z^3 + z^2 f_2(x, y) + z f_4(x, y) + f_6(x, y),
$$
where $ f_2, f_4, f_6 $ are homogeneous polynomials in $X$ and $ y $ of degrees $2$, $4$, and $6$ respectively, and $a \in \mathbb{C} $ is a constant. This defines $X$ as a double cover $\varphi: X \rightarrow \mathbb{P}(1,1,2),$
given by: $$ (x : y : z : w) \mapsto (x : y : z), $$ branched along the sextic curve
$$
R : a z^3 + z^2 f_2(x, y) + z f_4(x, y) + f_6(x, y) = 0 \subset \mathbb{P}(1,1,2).
$$
The branch curve $R$ has degree six an is in general singular. There is a natural one-to-one correspondence between the singularities of $R$ and the singular points of the surface $X$; that is, the singularities of $X$ lie precisely above the singular points of $R$.
As shown in \cite{Kosta09}, the singular points of $X$ are not contained in the base locus of the anti-canonical linear system $ |-K_X| $. In other words, they are not fixed points of this system.
\par In this section, we compute  $\delta$-invariants  of Du Val del Pezzo surfaces of degree $1$. 
\begin{maintheorem*}
 Let $X$ be a Du Val del Pezzo surface of degree $1$. Then $X$ can be realized as the double cover $X\xrightarrow{2:1}\DP(1,1,2)$, which is
ramified along a sextic curve $R\in \DP(1,1,2)$.  Then the  $\delta$-invariant of $X$ is uniquely determined by the type of singularities on $X$  and unique element $\mathcal{C}$ of $|-K_X|$ containing each of singular points which is given in the following table:
{\renewcommand\arraystretch{1.2}
	\begin{longtable}{|c||c|}
		\hline
Type of singularity  & $\delta(X)$\\
\hline \hline
\shortstack{{ }\\$\DA_1$, $2\DA_1$, $3\DA_1$, $4\DA_1$, $5\DA_1$, $6\DA_1$\\
all elements of $|-K_X|$ containing singular points are nodal}  & \shortstack{{ }\\$2$\\{ }{ }} \\
\hline 
\shortstack{{ }\\$\DA_1$, $2\DA_1$, $3\DA_1$, $4\DA_1$, $5\DA_1$, $6\DA_1$\\
 some elements of $|-K_X|$ containing\\
 singular points are cuspidal } & \shortstack{{ }\\$\frac{9}{5}$\\{ }\\{ }}\\
 \hline 
 \shortstack{{ }\\
$\DA_2$, $\DA_2+\DA_1$, $\DA_2+2\DA_1$,  $\DA_2+3\DA_1$, $\DA_2+4\DA_1$,\\
$2\DA_2$, $2\DA_2+\DA_1$, $2\DA_2+2\DA_1$, $3\DA_2$, $3\DA_2+\DA_1$, $4\DA_2$\\
all elements of $|-K_X|$ containing $\DA_2$ singular points are nodal}  &  \shortstack{$\frac{12}{7}$\\{ }\\{ }} \\
\hline \
\shortstack{{ }\\$\DA_2$, $\DA_2+\DA_1$, $\DA_2+2\DA_1$,  $\DA_2+3\DA_1$, $\DA_2+4\DA_1$,\\
$2\DA_2$, $2\DA_2+\DA_1$, $2\DA_2+2\DA_1$, $3\DA_2$, $3\DA_2+\DA_1$, $4\DA_2$\\
some elements of $|-K_X|$ containing $\DA_2$ singular points are cuspidal}
& \shortstack{$\frac{3}{2}$\\{ }\\{ }}\\
\hline 
\shortstack{{ }\\$\DA_3$, $\DA_3+\DA_1$, $\DA_3+2\DA_1$, $\DA_3+3\DA_1$, $\DA_3+4\DA_1$,\\
$\DA_3+\DA_2$, $\DA_3+\DA_2+\DA_1$, $\DA_3+\DA_2+2\DA_1$,\\
$2\DA_3$, $2\DA_3+\DA_1$, $2\DA_3+2\DA_1$}& \shortstack{$\frac{3}{2}$\\{ }\\{ }}\\
\hline
$\DA_4$, $\DA_4+\DA_1$, $\DA_4+2\DA_1$, $\DA_4+\DA_2$, $\DA_4+\DA_2+\DA_1$, $\DA_4+\DA_3$, $2\DA_4$  & {$\frac{4}{3}$}\\\hline
$\DA_5$, $\DA_5+\DA_1$, $\DA_5+2\DA_1$, $\DA_5+\DA_2$, $\DA_5+\DA_2+\DA_1$  &  $\frac{6}{5}$ \\
\hline
$\DA_6$, $\DA_6+\DA_1$  &  $\frac{9}{8}$ \\
\hline
$\DA_7$  and $R$ irreducible &  $\frac{18}{17}$\\
\hline
$\DA_7$, $\DA_7+\DA_1$ and $R$ reducible & $1$\\ 
\hline
$\DA_8$, $\mathbb{D}_4$, $\mathbb{D}_4+\DA_1$, $\mathbb{D}_4+2\DA_1$, $\mathbb{D}_4+3\DA_1$, $\mathbb{D}_4+\DA_2$, $\mathbb{D}_4+\DA_3$, $2\mathbb{D}_4$ & $1$
\\\hline
$\mathbb{D}_5$, $\mathbb{D}_5+\DA_1$, $\mathbb{D}_5+2\DA_1$, $\mathbb{D}_5+\DA_2$, $\mathbb{D}_5+\DA_3$ &  $\frac{6}{7}$\\ 
\hline
$\mathbb{D}_6$, $\mathbb{D}_6+\DA_1$, $\mathbb{D}_6+2\DA_1$ &  $\frac{3}{4}$ \\ 
\hline
$\mathbb{D}_7$ &  $\frac{2}{3}$ \\ 
\hline
$\mathbb{D}_8$, $\mathbb{E}_6$, $\mathbb{E}_6+\DA_1$, $\mathbb{E}_6+\DA_2$  &  $\frac{3}{5}$
\\\hline
$\mathbb{E}_7$, $\mathbb{E}_7+\DA_1$  & $\frac{3}{7}$ \\
\hline
$\mathbb{E}_8$  &  $\frac{3}{11}$ \\\hline
	\end{longtable}}
Note that when $X$ has $\DA_7$ singularity $\delta$-invariant depends on whether $R$ is reducible or irreducible.
  \end{maintheorem*}

To understand the anti-canonical system on the smooth surface $ S $, we apply the Riemann--Roch theorem together with Serre duality and the Kawamata--Viehweg vanishing theorem. For the divisor $ -K_S $, we have
$$
\chi(\mathcal{O}_S(-K_S)) = h^0(S, \mathcal{O}_S(-K_S)) - h^1(S, \mathcal{O}_S(-K_S)) + h^2(S, \mathcal{O}_S(-K_S)).
$$
Since $ -K_S $ is nef and big, the vanishing theorems imply $ h^1 = h^2 = 0 $, and therefore
$$
h^0(S, \mathcal{O}_S(-K_S)) = \chi(\mathcal{O}_S(-K_S)) = \frac{1}{2}K_S^2 + 1 = K_S^2 + 1.
$$
Thus, the anti-canonical system $ |-K_S| $ has dimension
$$
\dim |-K_S| = h^0(S, \mathcal{O}_S(-K_S)) - 1 = K_S^2 = 1,
$$
\begin{definition}
Let $ \pi : S \to X $ be a resolution of a point $ P $ on a normal surface $X$, and let $ E = \sum E_i $ denote the exceptional divisor over $ P $. Then there exists a unique effective exceptional divisor
$
\Gamma = \sum a_i E_i, \quad a_i \in \mathbb{Z}_{> 0},
$
satisfying the following properties:
\begin{enumerate}
  \item $ \Gamma > 0 $,
  \item $ \Gamma \cdot E_i \leq 0 $ for every component $ E_i $,
  \item $ \Gamma $ is minimal with respect to this property.
\end{enumerate}
The divisor $ \Gamma $ is called the {\bf fundamental cycle} of the configuration $ \{E_i\} $.
\end{definition}
In the context of Del Pezzo surfaces of degree one, \cite{Kosta09} shows the following result: let $ H \in |-K_S| $ be an anti-canonical divisor on the resolution $S$, and let $ \Gamma $ be the fundamental cycle of the exceptional divisor over a Du Val singularity. If the curve $ H $ contains a point of $ \Gamma $, then
$
H = C + \Gamma,
$
where $ C \subset S $ is the strict transform of a $(-1)$-curve $\mathcal{C}$ on $X$.
Moreover, all fundamental cycles arising from Du Val singularities on degree one Del Pezzo surfaces are explicitly described in \cite{Kosta09} based on \cite{KodairaI,KodairaII,KodairaIII}, including their configurations and intersection properties.
\par
Let $ C \subset S $ be a $ (-1) $-curve arising as the strict transform of a curve $ \mathcal{C} \subset X $. Contracting $ C $ yields a weak resolution of a Du Val Del Pezzo surface of degree two. In the previous section, we provided a complete classification of the dual graphs formed by $ (-1) $- and $ (-2) $-curves on such surfaces. Notably, all $ (-1) $-curves on $ S $ that intersect the exceptional divisors arise as strict transforms of $ (-1) $-curves on weak Del Pezzo surfaces of degree two. We  use this classification throughout the chapter. To determine the possible local dual graphs of singularities on $ X $, we proceed as follows: starting with a weak Del Pezzo surface of degree one with singularities $ \mathbf{``S"} $, we contract a $ (-1) $-curve as described above to obtain a weak Del Pezzo surface of degree two with singularities $ \mathbf{``S2"} $, which are uniquely determined by $ \mathbf{``S"}$. Since we have a full classification of dual graphs for weak Del Pezzo surfaces of degree two, we identify all occurrences of $ \mathbf{``S2"} $ and recover from them all possible configurations of $ (-1) $- and $ (-2) $-curves on the original surface with singularities $ \mathbf{``S"} $. This procedure allows for an explicit case-by-case description of all possible local dual graphs of singularities on $ X $. 
\par For each such surface, we compute the value of the $\delta$-invariant at every singular point, and we also have an estimate for the $\delta$-invariant at smooth points from the computations in \cite[Lemma~2.16]{Fano21}. Taking the minimum of these values yields a value for the global $\delta$-invariant which is determined by the singularities of the surface.

\subsection{Finding $\delta$-invariants for degree 1}
\subsubsection{$\DA_1$ singularity on Du Val Del Pezzo surfaces of degree $1$ such that $\mathcal{C}$ is nodal}
\begin{lemma}
Let $X$ be a singular del Pezzo surface of degree $1$ with an $\DA_1$ singularity at point $\mathcal{P}$.  Let $\mathcal{C}$ be a~curve in the~pencil $|-K_X|$ that contains~$\mathcal{P}$ and it has a node in $\mathcal{P}$. Then $\delta_{\mathcal{P}} (X)=2$.
\end{lemma} 
 \begin{proof}
Let $S$ be the minimal resolution of singularities.  Then $S$ is a weak del Pezzo surface of degree $1$. Suppose $C$ is a strict transform of $\mathcal{C}$ on $S$ and $E$ is the exceptional divisor. We have $-K_S\sim C+E$. Let $P$ be a point on $S$.
\begin{figure}[h!]
    \centering
 \includegraphics[width=7cm]{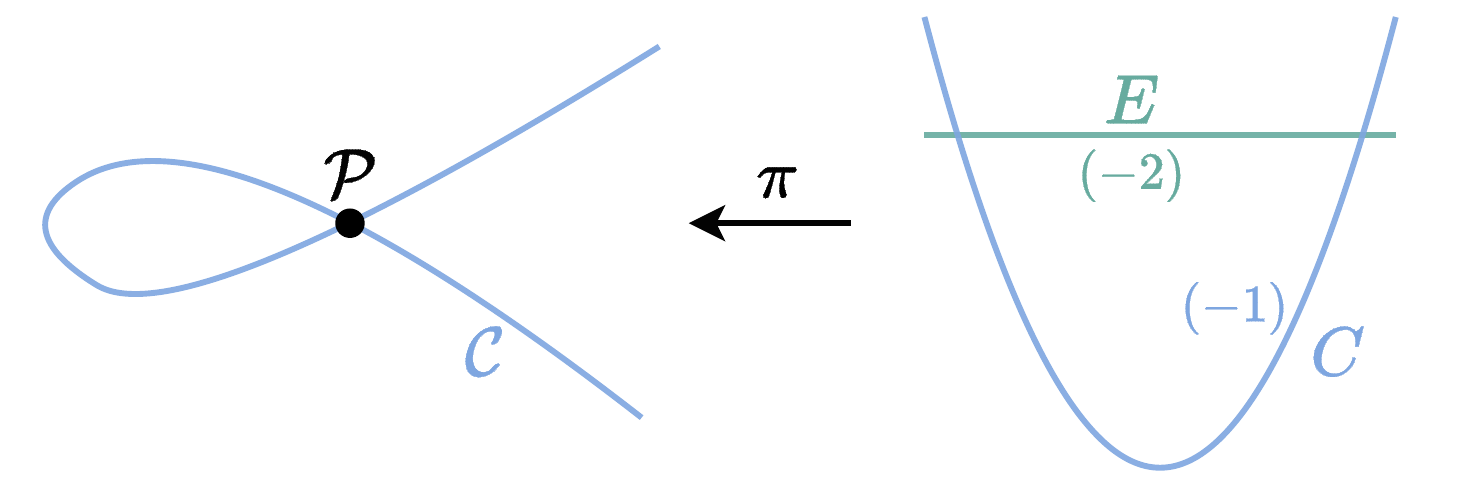}
    \caption{Picture: $(-K_S)^2=1$, $\DA_1$ singularity (nodal)}
\end{figure}
\par Suppose $P\in E$. Then $\tau(E)=1$ and the Zariski decomposition of the divisor $-K_S-vE\sim C+(1-v)E$ is given by:
$$P(v)=\begin{cases}
-K_S-vE \text{ if }v\in\big [0,\frac{1}{2}\big ],\\
-K_S-vE-(2v-1)C \text{ if } v\in\big [\frac{1}{2}, 1\big ].
\end{cases}
N(v)=
\begin{cases}
0 \text{ if }v\in\big [0,\frac{1}{2}\big ],\\
(2v-1)C \text{ if } v\in\big [\frac{1}{2}, 1\big ].
\end{cases}$$
Moreover, 
$$(P(v))^2=\begin{cases}
1-2v^2  \text{ if }v\in\big [0,\frac{1}{2}\big ],\\
2(v - 1)^2  \text{ if } v\in\big [\frac{1}{2}, 1\big ].
\end{cases}
P(v)\cdot E=\begin{cases}
2v  \text{ if }v\in\big [0,\frac{1}{2}\big ],\\
2(1-v)  \text{ if } v\in\big [\frac{1}{2}, 1\big ].
\end{cases}$$
We have
$S_{S} (E)=\frac{1}{2}$. Thus, $\delta_P(S)\le 2$ for $P\in E$. Moreover, if $P\in E$:
$$h(v)\le\begin{cases}
2v^2 \text{ if }v\in\big [0,\frac{1}{2}\big ],\\
2v(1-v)  \text{ if } v\in\big [\frac{1}{2}, 1\big ].
\end{cases}$$
Thus, $S(W_{\bullet,\bullet}^{E};P)\le\frac{1}{2}$ and We get $\delta_P(S)=2$ for $P\in E$. Which gives us $\delta_{\mathcal{P}} (X)=2$.
\end{proof}
\subsubsection{$\DA_1$ singularity on Du Val Del Pezzo surfaces of degree $1$ such that $\mathcal{C}$ is cuspidal}
 \begin{lemma}
 Let $X$ be a singular del Pezzo surface of degree $1$ with an $\DA_1$ singularity at point $\mathcal{P}$. Let $\mathcal{C}$ be a~curve in the~pencil $|-K_X|$ that contains~$\mathcal{P}$ and it has a cusp in $\mathcal{P}$. Then $\delta_{\mathcal{P}} (X)=\frac{9}{5}$.
 \end{lemma}
 \begin{proof}
 Consider the blowup $\pi_1:S_1\to X$ of $X$ at $\mathcal{P}$ with the exceptional divisor $E_1^1$ and $C^1$ is a strict transform of $\mathcal{C}$.
Let $\pi_2\colon S_2\to S_1$ be the~blow up of the~point $C^1\cap E_1^1$ with the exceptional divisor $E_2^2$ and $E_1^2$, $C^2$ are a strict transforms of $E_1^1$, $C^1$ respectively.
Let $\pi_3\colon S_3\to S_2$ be the~blow up of the~point $C^2\cap E_1^2\cap E_2^2$ with the exceptional divisor $E$ and $E_1^3$, $E_2^3$, $C^3$ are a strict transforms of $E_1^2$, $E_2^2$, $C^2$ respectively.
Then $(\pi_1\circ\pi_2\circ \pi_3)^*(-K_{X})\sim C^3+E_1^3+2E_2^3+4E$.
Let $\theta\colon S_3\to \overline{S}$ be the~contraction of the~curves $E_1^3$ and $E_2^3$, let $\overline{C}=\theta(C^3)$ and $\overline{E}=\theta(E)$.
\begin{figure}[h!]
    \centering
\includegraphics[width=14cm]{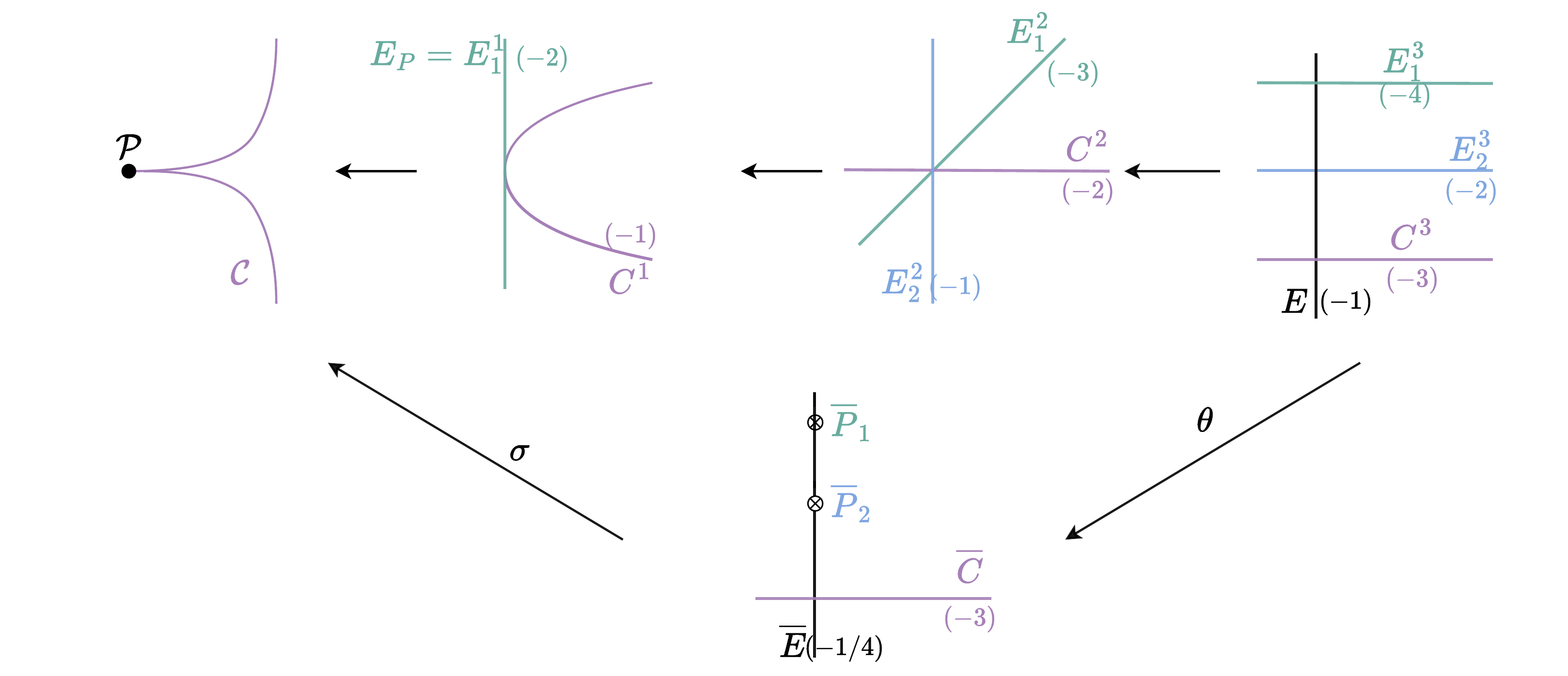}
    \caption{Picture: $(-K_S)^2=1$, $\DA_1$ singularity (cuspidal)}
\end{figure}
\par Then $\overline{P}_2=\theta(E_2^3)$ is a quotient singular point of type $\frac{1}{2} (1,1)$
and $\overline{P}_1=\theta(E_1^3)$ is a quotient singular point of type $\frac{1}{4} (1,1)$ and the intersections are given by: 
\begin{center}
\renewcommand\arraystretch{1.1}
	\begin{tabular}{|c||c|c|}
		\hline
		 &$\overline{C}$ & $\overline{E}$ \\
		\hline
		\hline
$\overline{C}$ & $-3$&\cellcolor[gray]{0.9} $1$ \\
		\hline
$\overline{E}$ &\cellcolor[gray]{0.9} $1$& $-\frac{1}{4}$  \\
		\hline
	\end{tabular}
\end{center}
Observe that $-K_{\overline{S}}$ is big. 
Then $\tau(\overline{E})=4$ and the Zariski decomposition of the divisor $\sigma^*(-K_X)-v\overline{E}\sim (4-v)\overline{E}+\overline{C}$ is given by
$$P(v)=
\begin{cases}
(4-v)\overline{E}+\overline{C}\text{ if }v\in[0,1],\\
(4-v)\overline{E}+\frac{4-v}{3}\overline{C}\text{ if }v\in[1,4].
\end{cases}
N(v)=
\begin{cases}
0\text{ if }v\in[0,1],\\
\frac{v-1}{3}\overline{C}\text{ if }v\in[1,4].
\end{cases}$$
Moreover
$$P(v)^2=
\begin{cases}
\frac{(2-v)(2+v)}{4}\text{ if }v\in[0,1],\\
\frac{(4-v)^2}{12}\text{ if }v\in[1,4].
\end{cases}
P(v)\cdot \overline{E} =
\begin{cases}
\frac{v}{4}\text{ if }v\in[0,1],\\
\frac{4-v}{12}\text{ if }v\in[1,4].
\end{cases}$$
So we have
$S_S(\overline{E})= \frac{5}{3}$ for $P\in \overline{E}$. Thus,   Thus, $\delta_P(S)\le \frac{9}{5}$.
Moreover, if $P\in  \overline{E}\backslash \overline{C}$ or $P\in  \overline{E}\cap \overline{C}$ then
$$h(v)=
\begin{cases}
\frac{v^2}{32}\text{ if }v\in[0,1],\\
\frac{(4-v)^2}{288}\text{ if }v\in[1,4].
\end{cases}
\text{ or }
h(v)=
\begin{cases}
\frac{v^2}{32}\text{ if }v\in[0,1],\\
\frac{(4 - v) (7 v - 4)}{288}\text{ if }v\in[1,4].
\end{cases}$$
So  $S(W^{\overline{E}}_{\bullet,\bullet};O)=\frac{1}{12}$
or 
$S(W^{\overline{E}}_{\bullet,\bullet};O)=\frac{1}{3}$.
On the~other hand:
$$
\delta_P(S)\geqslant\mathrm{min}\Bigg\{\frac{9}{5},\inf_{O\in\overline{E}}\frac{A_{\overline{E},\Delta_{\overline{E}}} (O)}{S\big(W^{\overline{E}}_{\bullet,\bullet};O\big)}\Bigg\},
$$
where $\Delta_{\overline{E}}=\frac{1}{2}P_1+\frac{3}{4}P_2$. So we have
$$
\frac{A_{\overline{E},\Delta_{\overline{E}}} (O)}{S(W_{\bullet,\bullet}^{\overline{E}};O)}=
\left\{\aligned
&3\ \mathrm{if}\ O=\overline{E}\cap\overline{C},\\
&3\ \mathrm{if}\ O=P_1,\\
&4\ \mathrm{if}\ O=P_2,\\
&12\ \mathrm{otherwise}.
\endaligned
\right.
$$
Thus, $\delta_{\mathcal{P}} (X)=\frac{9}{5}$.
\end{proof}
\subsubsection{$\DA_2$ singularity on Du Val Del Pezzo surfaces of degree $1$ such that $\mathcal{C}$ is nodal}
\begin{lemma} Let $X$ be a singular del Pezzo surface of degree $1$ with an $\DA_2$ singularity at point $\mathcal{P}$.  Let $\mathcal{C}$ be a~curve in the~pencil $|-K_X|$ that contains~$\mathcal{P}$ and it has a node in $\mathcal{P}$. Then $\delta_{\mathcal{P}} (X)=\frac{12}{7}$.
\end{lemma}
 \begin{proof}
Let $S$ be the minimal resolution of singularities.  Then $S$ is a weak del Pezzo surface of degree $1$. Suppose $C$ is a strict transform of $\mathcal{C}$ on $S$ and $E_1$ and $E_2$ are the exceptional divisors.  We have $-K_S\sim C+E_1+E_2$. Let $P$ be a point on $S$.
\begin{figure}[h!]
    \centering
 \includegraphics[width=7cm]{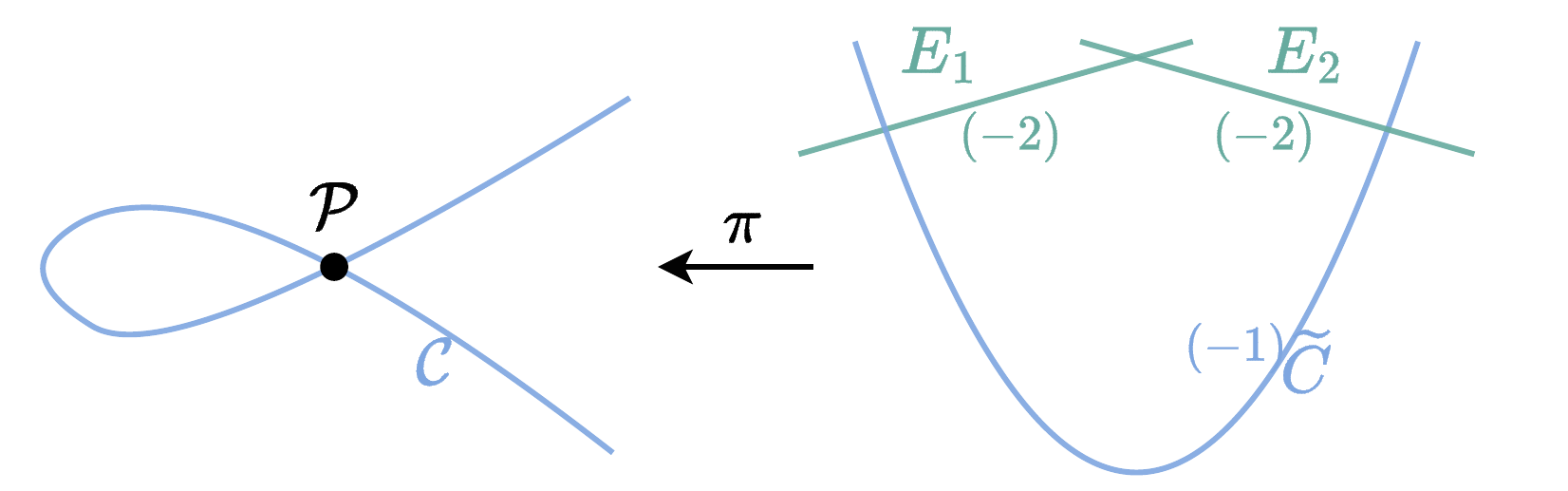}
    \caption{Picture: $(-K_S)^2=1$, $\DA_2$ singularity (nodal)}
\end{figure}
\par {\bf Step 1.} Suppose $P\in E_1\cup E_2$. Without loss of generality we can assume that $P\in E_1$ since the proof is similar in other cases. Then $\tau(E_1)=1$ and the Zariski decomposition of the divisor $-K_S-vE_1\sim C+(1-v)E_1+E_2$ is given by:
 {\allowdisplaybreaks\begin{align*}
&&P(v)=\begin{cases}
-K_S-vE_1-\frac{v}{2}E_2 \text{ if }v\in\big [0,\frac{2}{3}\big ],\\
-K_S-vE_1-(2v-1)E_2-(3v-2)C \text{ if } v\in\big [\frac{2}{3}, 1\big ].
\end{cases}\\&&
N(v)=
\begin{cases}
\frac{v}{2}E_2 \text{ if }v\in\big [0,\frac{2}{3}\big ],\\
(2v-1)E_2+(3v-2)C \text{ if }v\in\big [\frac{2}{3},1\big ].
\end{cases}
\end{align*}}
Moreover, 
$$(P(v))^2=\begin{cases}
1-\frac{3v^2}{2}  \text{ if }v\in\big [0,\frac{2}{3}\big ],\\
3(v - 1)^2  \text{ if }v\in\big [\frac{2}{3},1\big ].
\end{cases}
P(v)\cdot E=\begin{cases}
\frac{3v}{2}  \text{ if }v\in\big [0,\frac{2}{3}\big ],\\
3(1-v)  \text{ if }v\in\big [\frac{2}{3},1\big ].
\end{cases}$$
We have $S_{S} (E_1)=\frac{5}{9}$. Thus, $\delta_P(S)\le \frac{9}{5}$ for $P\in E_1\backslash E_2$. Moreover,  for such points we have 
$$h(v)\le\begin{cases}
\frac{9v^2}{8} \text{ if }v\in \big [0, \frac{2}{3}\big ],\\
\frac{3(1-v)(v + 1)}{2} \text{ if }v\in \big [\frac{2}{3},1\big ].
\end{cases}$$
Thus,
$S(W_{\bullet,\bullet}^{E_1};P)\le\frac{14}{27}<\frac{5}{9}$.
We get $\delta_P(S)=\frac{9}{5}$ for $P\in (E_1\cup E_2)\backslash (E_1\cap E_2)$. 
\par{\bf Step 2.} Suppose $P=E_1\cap E_2$. Consider the blowup $\sigma:\widetilde{S}\to S$ of $S$ at $P$ with the exceptional divisor $E_P$. Suppose $\widetilde{E}_1$,  $\widetilde{E}_2$  and $\widetilde{C}$  are strict transforma of $E_1$, $E_2$ and $C$ on $S$. 
 Then $\tau(E_P)=2$ and the  Zariski decomposition of the divisor $\sigma^*(-K_{S})-vE_P\sim \widetilde{C}+\widetilde{E}_1+\widetilde{E}_2+(2-v)E_P$ is given by:
 {\allowdisplaybreaks\begin{align*}
&&P(v)=\begin{cases}
\sigma^*(-K_{S})-vE_P-\frac{v}{3} (\widetilde{E}_1+\widetilde{E}_2) \text{ if }v\in\big[0,\frac{3}{2}\big],\\
\sigma^*(-K_{S})-vE_P-(v-1)(\widetilde{E}_1+\widetilde{E}_2)-(2v-3)\widetilde{C} \text{ if }v\in\big[\frac{3}{2},2 \big].
\end{cases}\\&&N(v)=\begin{cases}
\frac{v}{3} (\widetilde{E}_1+\widetilde{E}_2) \text{ if }v\in\big[0,\frac{3}{2}\big],\\
(v-1)(\widetilde{E}_1+\widetilde{E}_2)+(2v-3)\widetilde{C} \text{ if }v\in\big[\frac{3}{2},2 \big].
\end{cases}
\end{align*}}
Moreover,
$$(P(v))^2=
\begin{cases}
1-\frac{v^2}{3} \text{ if }v\in\big[0,\frac{3}{2}\big],\\
(2-v)^2 \text{ if }v\in\big[\frac{3}{2},2 \big].
\end{cases}
P(v)\cdot E_P=
\begin{cases}
\frac{v}{3} \text{ if }v\in\big[0,\frac{3}{2}\big],\\
2-v \text{ if }v\in\big[\frac{3}{2},2 \big].
\end{cases}$$
We have $S_{S} (E_P)=\frac{7}{6}$. Thus, $\delta_P(S)\le \frac{2}{7/6}=\frac{12}{7}$ for $P=E_1\cap E_2$. Moreover,
$$h(v)\le \begin{cases}
\frac{v^2}{6}  \text{ if }v\in\big[0,\frac{3}{2}\big],\\
\frac{(2-v)v}{2} \text{ if }v\in\big[\frac{3}{2},2 \big].
\end{cases}$$
Thus, 
$S(W_{\bullet,\bullet}^{E_P};O)\le\frac{7}{12}$.
We get $\delta_P(S)=\frac{12}{7}$ for $P=E_1\cap E_2$. Thus, $\delta_{\mathcal{P}} (X)=\frac{12}{7}$.
\end{proof}

\subsubsection{$\DA_2$ singularity on Du Val Del Pezzo surfaces of degree $1$ such that $\mathcal{C}$ is cuspidal}
 \begin{lemma} Let $X$ be a singular del Pezzo surface of degree $1$ with an $\DA_2$ singularity at point $\mathcal{P}$.  Let $\mathcal{C}$ be a~curve in the~pencil $|-K_X|$ that contains~$\mathcal{P}$ and it has a cusp in $\mathcal{P}$. Then $\delta_{\mathcal{P}} (X)=\frac{3}{2}$.
 \end{lemma}
 \begin{proof}
Let $S$ be the minimal resolution of singularities.  Then $S$ is a weak del Pezzo surface of degree $1$. 
 We have $-K_S\sim C+E_1+E_2$. Let $P$ be a point on $S$. Let also $\sigma:\widetilde{S}\to S$ be the blowup of a point $P=E_1\cap E_2\cap C$.  Let $\widetilde{C}$, $\widetilde{E}_1$ and $\widetilde{E}_2$  be strict transforms of $C$, $E_1$ and $E_2$ on $\widetilde{S}$. 
 \begin{figure}[h!]
    \centering
\includegraphics[width=12cm]{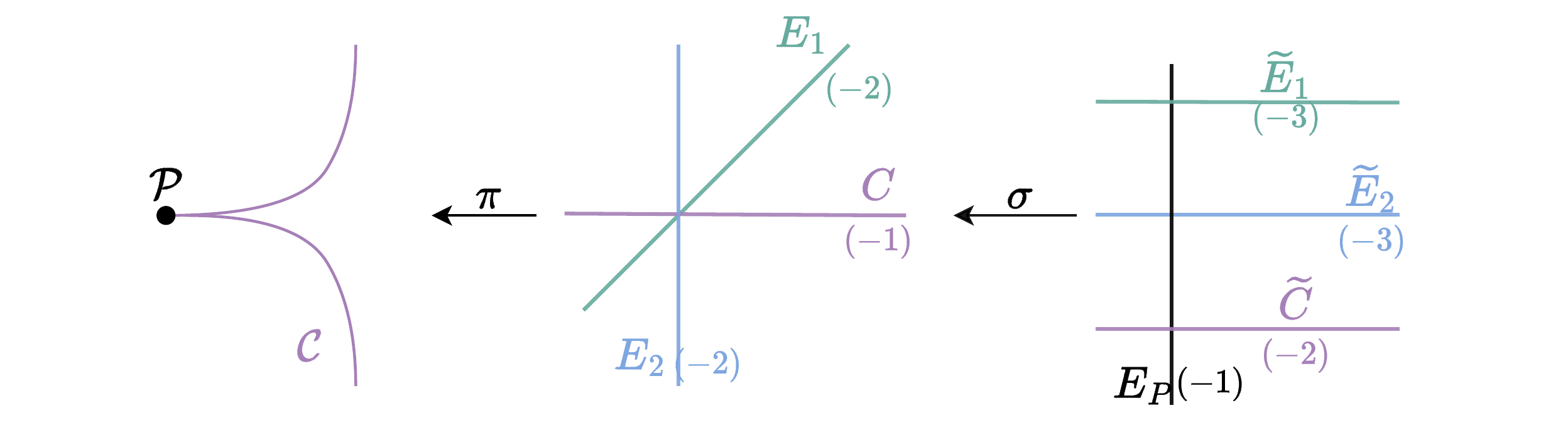}
    \caption{Picture: $(-K_S)^2=1$, $\DA_2$ singularity (cuspidal)}
\end{figure}
\par  {\bf Step 1.} Suppose $P\in E_1\cup E_2$. Without loss of generality we can assume that $P\in E_1$ since the proof is similar in other cases.  Then $\tau(E_1)=1$ and the  Zariski decomposition of the divisor $-K_S-vE_1\sim C+(1-v)E_1+E_2$ is given by:
 {\allowdisplaybreaks\begin{align*}
&&P(v)=\begin{cases}
-K_S-vE_1-\frac{v}{2}E_2 \text{ if }v\in\big [0,\frac{2}{3}\big ],\\
-K_S-vE_1-(2v-1)E_2-(3v-2)C \text{ if } v\in\big [\frac{2}{3}, 1\big ].
\end{cases}\\&&
N(v)=
\begin{cases}
\frac{v}{2}E_2 \text{ if }v\in\big [0,\frac{2}{3}\big ],\\
(2v-1)E_2+(3v-2)C \text{ if }v\in\big [\frac{2}{3},1\big ].
\end{cases}
\end{align*}}
Moreover, 
$$(P(v))^2=\begin{cases}
1-\frac{3v^2}{2}  \text{ if }v\in\big [0,\frac{2}{3}\big ],\\
3(v - 1)^2  \text{ if }v\in\big [\frac{2}{3},1\big ].
\end{cases}
P(v)\cdot E=\begin{cases}
\frac{3v}{2}  \text{ if }v\in\big [0,\frac{2}{3}\big ],\\
3(1-v)  \text{ if }v\in\big [\frac{2}{3},1\big ].
\end{cases}$$
We have $S_{S} (E_1)=\frac{5}{9}$. Thus, $\delta_P(S)\le \frac{9}{5}$ for $P\in E_1\backslash E_2$. Moreover,  for such points we have 
$$h(v)\le\begin{cases}
\frac{9v^2}{8} \text{ if }v\in \big [0, \frac{2}{3}\big ],\\
\frac{3(1-v)(v + 1)}{2} \text{ if }v\in \big [\frac{2}{3},1\big ].
\end{cases}$$
Thus,
$S(W_{\bullet,\bullet}^{E_1};P)\le \frac{14}{27}<\frac{5}{9}$.
We get $\delta_P(S)=\frac{9}{5}$ for $P\in (E_1\cup E_2)\backslash (E_1\cap E_2)$. 
\par{\bf Step 2.} Suppose $P=E_1\cap E_2$. Consider the blowup $\sigma:\widetilde{S}\to S$ of $S$ at $P$ with the exceptional divisor $E_P$. Suppose $\widetilde{E}_1$,  $\widetilde{E}_2$  and $\widetilde{C}$  are strict transforma of $E_1$, $E_2$ and $C$ on $S$. 
 Then $\tau(E_P)=3$ and the  Zariski decomposition of the divisor $\sigma^*(-K_{S})-vE_P\sim \widetilde{C}+\widetilde{E}_1+\widetilde{E}_2+(3-v)E_P$ is given by:
 {\allowdisplaybreaks\begin{align*}
&&P(v)=\begin{cases}
\sigma^*(-K_{S})-vE_P-\frac{v}{3} (\widetilde{E}_1+\widetilde{E}_2) \text{ if }v\in[0,1],\\
\sigma^*(-K_{S})-vE_P-(v-1)(\widetilde{E}_1+\widetilde{E}_2)-\frac{v-1}{2}\widetilde{C} \text{ if }v\in[1,3].
\end{cases}\\&&N(v)=\begin{cases}
\frac{v}{3} (\widetilde{E}_1+\widetilde{E}_2) \text{ if }v\in[0,1],\\
(v-1)(\widetilde{E}_1+\widetilde{E}_2)+\frac{v-1}{2}\widetilde{C} \text{ if }v\in[1,3].
\end{cases}
\end{align*}}
Moreover,
$$(P(v))^2=
\begin{cases}
1-\frac{v^2}{3} \text{ if }v\in[0,1],\\
\frac{(3-v)^2}{6} \text{ if }v\in[1,3].
\end{cases}
P(v)\cdot E_P=
\begin{cases}
\frac{v}{3} \text{ if }v\in[0,1],\\
\frac{3-v}{6} \text{ if }v\in[1,3].
\end{cases}$$
We have
$S_{S} (E_P)=\frac{4}{3}$. Thus, $\delta_P(S)\le \frac{2}{4/3}=\frac{3}{2}$ for $P=E_1\cap E_2\cap C$. Moreover, if $O\in E_P\backslash (\widetilde{E}_1\cup \widetilde{E}_2)$ if $O\in E_P\backslash \widetilde{C}$ we have:
$$h(v)\le \begin{cases}
\frac{v^2}{18}  \text{ if }v\in[0,1],\\
\frac{(3 - v)(5v - 3)}{72} \text{ if }v\in[1,3].
\end{cases}
\text{ or }
h(v)\le \begin{cases}
\frac{v^2}{6}  \text{ if }v\in[0,1],\\
\frac{(3 - v)(v +1)}{24} \text{ if }v\in[1,3].
\end{cases}$$
Thus, $S(W_{\bullet,\bullet}^{E_P};O)\le \frac{1}{3}<\frac{2}{3}$ or 
$S(W_{\bullet,\bullet}^{E_P};O)\le \frac{5}{9}<\frac{2}{3}$.
We get $\delta_P(S)=\frac{3}{2}$ for $P=E_1\cap E_2$. \\ Thus, $\delta_{\mathcal{P}} (X)=\frac{3}{2}$.
\end{proof}

\subsubsection{$\DA_3$ singularity on Du Val Del Pezzo surfaces of degree $1$}
 \begin{lemma} Let $X$ be a singular del Pezzo surface of degree $1$ with an $\DA_3$ singularity at point $\mathcal{P}$.  Let $\mathcal{C}$ be a~curve in the~pencil $|-K_X|$ that contains~$\mathcal{P}$.  Then $\delta_{\mathcal{P}} (X)=\frac{3}{2}$.
 \end{lemma}
 \begin{proof}
Let $S$ be the minimal resolution of singularities.  Then $S$ is a weak del Pezzo surface of degree $1$. Suppose $C$ is a strict transform of $\mathcal{C}$ on $S$ and $E_1$, $E_2$ and $E_3$ are the exceptional divisors with the following intersection:
 \begin{figure}[h!]
    \centering
\includegraphics[width=4cm]{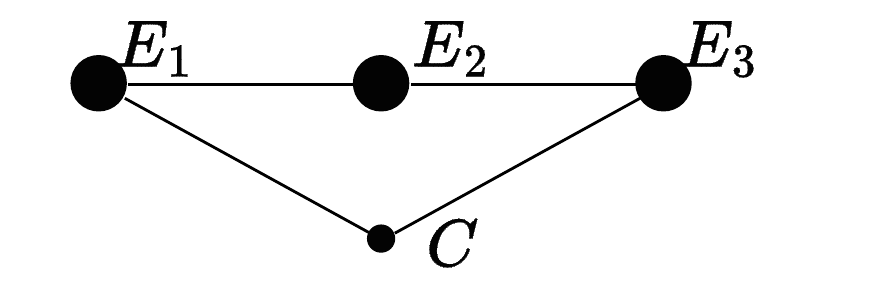}
    \caption{Dual graph: $(-K_S)^2=1$, $\DA_3$ singularity}
\end{figure}
\par We have $-K_S\sim C+E_1+E_2+E_3$. Let $P$ be a point on $S$.\\
\noindent {\bf Step 1.} Suppose $P\in E_2$.  Then $\tau(E_2)=1$ and the  Zariski decomposition of the divisor $-K_S-vE_2\sim C+E_1+(1-v)E_2+E_3$ is given by:
$$P(v)=
-K_S-vE_2-\frac{v}{2} (E_1+E_3)
\text{ and }
N(v)=
\frac{v}{2} (E_1+E_3)\text{ if }v\in[0,1].$$
Moreover, 
$$(P(v))^2=(1-v)(1+v)\text{ and }
P(v)\cdot E_2=v\text{ if }v\in[0,1].$$
We have
$S_{S} (E_2)=\frac{2}{3}$. Thus, $\delta_P(S)\le \frac{3}{2}$ for $P\in E_2$. Moreover,  for such points we have 
$h(v)\le v^2\text{ if }v\in[0,1]$. Thus,
$S(W_{\bullet,\bullet}^{E_2};P)\le\frac{2}{3}$.
We get $\delta_P(S)=\frac{3}{2}$ for $P\in E_2$. 
\par{\bf Step 2.} Suppose $P\in E_1\cup E_3$. Without loss of generality we can assume that $P\in E_1$ since the proof is similar in other cases.  Then $\tau(E_1)=1$ and the  Zariski decomposition of the divisor $-K_S-vE_1\sim  C+(1-v)E_1+E_2+E_3$ is given by:
 {\allowdisplaybreaks\begin{align*}
&&P(v)=\begin{cases}
-K_S-vE_1-\frac{v}{3} (2E_2+E_3) \text{ if }v\in\big [0,\frac{3}{4}\big ],\\
-K_S-vE_1-(2v-1)E_2-(3v-2)E_3-(4v-3)C \text{ if } v\in\big [\frac{3}{4}, 1\big ]
\end{cases}\\&&
N(v)=
\begin{cases}
\frac{v}{3} (2E_2+E_3) \text{ if }v\in\big [0,\frac{3}{4}\big ],\\
(2v-1)E_2+(3v-2)E_3+(4v-3)C \text{ if }v\in\big [\frac{3}{4},1\big ].
\end{cases}
\end{align*}}
Moreover, 
$$(P(v))^2=\begin{cases}
1-\frac{4v^2}{3}  \text{ if }v\in\big [0,\frac{3}{4}\big ],\\
4(v - 1)^2  \text{ if }v\in\big [\frac{3}{4},1\big ].
\end{cases}
P(v)\cdot E_1=\begin{cases}
\frac{4v}{3}  \text{ if }v\in\big [0,\frac{3}{4}\big ],\\
4(1-v)  \text{ if }v\in\big [\frac{3}{4},1\big ].
\end{cases}$$
We have $S_{S} (E_1)=\frac{5}{9}$. Thus, $\delta_P(S)\le \frac{9}{5}$ for $P\in E_1\backslash E_2$. Moreover,  for such points we have 
$$h(v)\le\begin{cases}
\frac{8v^2}{9} \text{ if }v\in \big [0, \frac{3}{4}\big],\\
4(1-v)(2v - 1) \text{ if }v\in \big [\frac{3}{4},1\big].
\end{cases}$$
Thus,
$S(W_{\bullet,\bullet}^{E_1};P)\le \frac{5}{12}<\frac{7}{12}$.
We get $\delta_P(S)=\frac{12}{7}$ for $P\in (E_1\cup E_3)\backslash E_2$. Thus, $\delta_{\mathcal{P}} (X)=\frac{3}{2}$.
\end{proof}

\subsubsection{$\DA_4$ singularity on Du Val Del Pezzo surfaces of degree $1$}\label{dP1-A4}
 \begin{lemma}
      Let $X$ be a singular del Pezzo surface of degree $1$ with an $\DA_4$ singularity at point $\mathcal{P}$.  Let $\mathcal{C}$ be a~curve in the~pencil $|-K_X|$ that contains~$\mathcal{P}$. Then $\delta_{\mathcal{P}} (X)=\frac{4}{3}$.
    \end{lemma}
 \begin{proof}
Let $S$ be the minimal resolution of singularities.  Then $S$ is a weak del Pezzo surface of degree $1$. Suppose $C$ is a strict transform of $\mathcal{C}$ on $S$ and $E_1$, $E_2$, $E_3$ and $E_4$ are the exceptional divisors with the intersection:
 \begin{figure}[h!]
    \centering
\includegraphics[width=6cm]{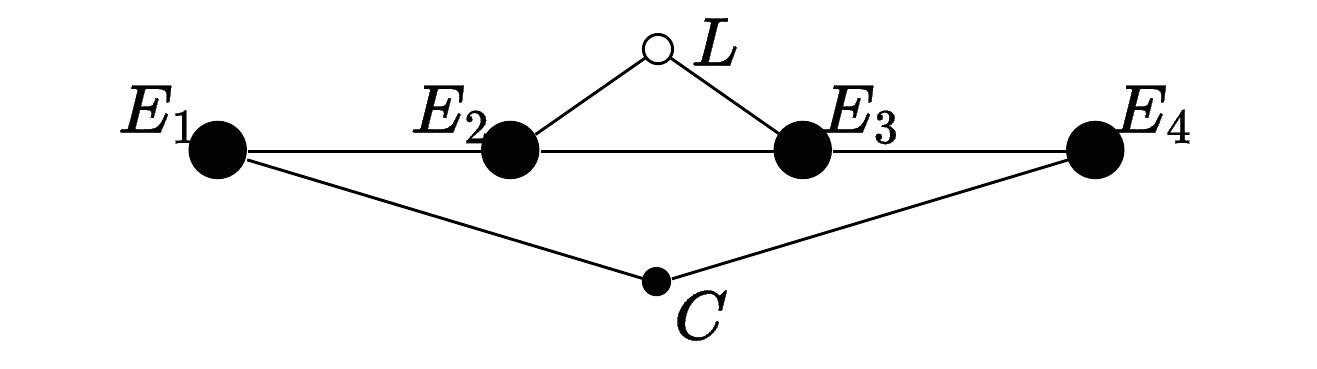}
    \caption{Dual graph: $(-K_S)^2=1$, $\DA_4$ singularity}
\end{figure}
\par We have $-K_S\sim C+E_1+E_2+E_3+E_4$. Let $P$ be a point on $S$.
Consider a linear system  $\mathcal{L}=| -2K_S-(E_1 + 2E_2 + 2E_3 + E_4)|$. Using Riemann-Roch for surfaces we get $\dim|\mathcal{L}|=1$. Thus, since the linear system $|\mathcal{L}|$ does not have base points there is a unique element $L\in |\mathcal{L}|$ such that it contains the intersection point of $E_2$ and $E_3$. Moreover we have $L\cdot E_1=L\cdot E_4=0$, $L\cdot E_2=L\cdot E_3=1$ and $L^2=0$.
\par{\bf Step 1.} Suppose $P=E_2\cap E_3$. Consider the blowup $\sigma:\widetilde{S}\to S$ of $S$ at $P$ with the exceptional divisor $E_P$. Suppose $\widetilde{E}_1$,  $\widetilde{E}_2$, $\widetilde{E}_3$, $\widetilde{E}_4$, $\widetilde{L}$  and $\widetilde{C}$  are strict transforms of $E_1$, $E_2$, $E_3$, $E_4$, $L$ and $C$ on $\widetilde{S}$. 
 Then $\tau(E_P)=\frac{5}{2}$ and the  Zariski decomposition of the divisor 
$$\sigma^*(-K_{S})-vE_P\sim \Big(\frac{5}{2}-v\Big)E_P+\frac{1}{2}\widetilde{L}+\widetilde{E}_1+\frac{3}{2}\widetilde{E}_2+\frac{3}{2}\widetilde{E}_3+\widetilde{E}_4$$ is given by:
 {\allowdisplaybreaks\begin{align*}
&&P(v)=\begin{cases}
\sigma^*(-K_{S})-vE_P-\frac{v}{5} (\widetilde{E}_1+2\widetilde{E}_2+2\widetilde{E}_3+\widetilde{E}_4) \text{ if }v\in[0,2],\\
\sigma^*(-K_{S})-vE_P-\frac{v}{5} (\widetilde{E}_1+2\widetilde{E}_2+2\widetilde{E}_3+\widetilde{E}_4)-(v-2)\widetilde{L} \text{ if }v\in\big[2,\frac{5}{2} \big].
\end{cases}\\&&N(v)=\begin{cases}
\frac{v}{5} (\widetilde{E}_1+2\widetilde{E}_2+2\widetilde{E}_3+\widetilde{E}_4) \text{ if }v\in[0,2],\\
\frac{v}{5} (\widetilde{E}_1+2\widetilde{E}_2+2\widetilde{E}_3+\widetilde{E}_4)-(v-2)\widetilde{L} \text{ if }v\in\big[2,\frac{5}{2} \big].
\end{cases}
\end{align*}}
Moreover,
$$(P(v))^2=
\begin{cases}
1-\frac{v^2}{5} \text{ if }v\in[0,2],\\
\frac{(5-2v)^2}{5} \text{ if }v\in\big[2,\frac{5}{2} \big].
\end{cases}
P(v)\cdot E_P=
\begin{cases}
\frac{v}{5} \text{ if }v\in[0,2],\\
2(1-\frac{2v}{5}) \text{ if }v\in\big[2,\frac{5}{2} \big].
\end{cases}$$
We have $S_{S} (E_P)=\frac{3}{2}$. Thus, $\delta_P(S)\le \frac{2}{3/2}=\frac{4}{3}$ for $P=E_2\cap E_3$. Moreover, if $O\in E_P\backslash (\widetilde{E}_2\cup \widetilde{E}_3)$ if $O\in E_P\backslash \widetilde{L}$ we have:
$$h(v)\le \begin{cases}
\frac{v^2}{50}  \text{ if }v\in[0,2],\\
\frac{2(5 - 2v)(3v - 5)}{25} \text{ if }v\in\big[2,\frac{5}{2} \big].
\end{cases}
\text{ or }
h(v)\le \begin{cases}
\frac{v^2}{10}  \text{ if }v\in[0,2],\\
\frac{2(5 - 2v)}{5} \text{ if }v\in\big[2,\frac{5}{2} \big].
\end{cases}$$
Thus,
$S(W_{\bullet,\bullet}^{E_P};O)\le \frac{1}{6}<\frac{3}{4}$ or 
$S(W_{\bullet,\bullet}^{E_P};O)\le \frac{11}{15}<\frac{3}{4}$.
We get $\delta_P(S)=\frac{4}{3}$ for $P=E_2\cap E_3$.
\par{\bf Step 2.} Suppose $P\in E_2\cup E_3$. Without loss of generality we can assume that $P\in E_2$ since the proof is similar in other cases. 
If  we contract the curve $C$  the resulting
surface is isomorphic to a weak del Pezzo surface of degree $2$  with at most Du Val singularities. Thus, there exist $(-1)$-curves and $(-2)$-curves  which form one of the following dual graphs:
 \begin{figure}[h!]
    \centering
\hspace*{-0.5cm}\includegraphics[width=17cm]{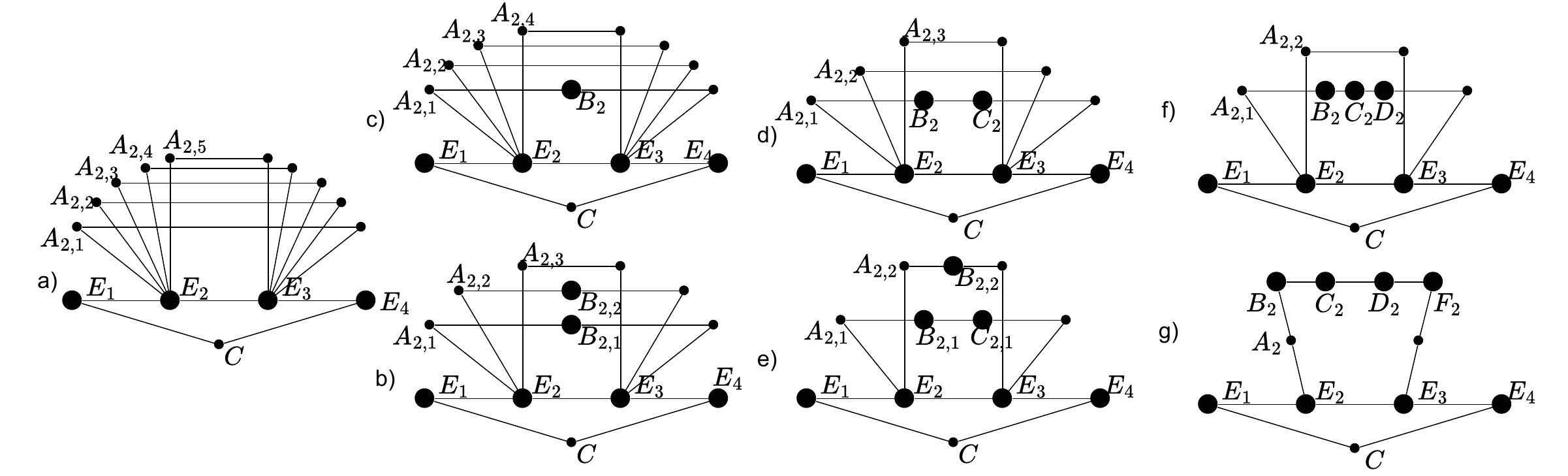}
    \caption{Dual graph: $(-K_S)^2=1$, $\DA_4$ singularity, $\delta_P(S)=\frac{15}{11}$}
\end{figure}
\par Then $\tau(E_2)=\frac{6}{5}$ and the  Zariski Decomposition of the divisor $-K_S-vE_2$ is:
   { 
 {\allowdisplaybreaks\hspace*{-1cm}\begin{align*}
\hspace*{-0.9cm}&{\text{\bf a). }} & P(v)=\begin{cases}
-K_S-vE_2-\frac{v}{6} (3E_1+4E_3+2E_4)\text{ if }v\in[0,1],\\
-K_S-vE_2-\frac{v}{6} (3E_1+4E_3+2E_4)-(v-1)(A_{2,1}+A_{2,2}+A_{2,3}+A_{2,4}+A_{2,5})\text{ if }v\in\big[1,\frac{6}{5}\big].
\end{cases} \\
\hspace*{-0.9cm}&&N(v)=\begin{cases}
\frac{v}{6} (3E_1+4E_3+2E_4)\text{ if }v\in[0,1],\\
\frac{v}{6} (3E_1+4E_3+2E_4)+(v-1)(A_{2,1}+A_{2,2}+A_{2,3}+A_{2,4}+A_{2,5})\text{ if }v\in\big[1,\frac{6}{5}\big].
\end{cases}\\
\hspace*{-0.9cm}&{\text{\bf b). }} & P(v)=\begin{cases}
-K_S-vE_2-\frac{v}{6} (3E_1+4E_3+2E_4)\text{ if }v\in[0,1],\\
-K_S-vE_2-\frac{v}{6} (3E_1+4E_3+2E_4)-(v-1)(2A_{2,1}+B_2+A_{2,2}+A_{2,3}+A_{2,4})\text{ if }v\in\big[1,\frac{6}{5}\big].
\end{cases} \\
\hspace*{-0.9cm}&&N(v)=\begin{cases}
\frac{v}{6} (3E_1+4E_3+2E_4)\text{ if }v\in[0,1],\\
\frac{v}{6} (3E_1+4E_3+2E_4)+(v-1)(2A_{2,1}+B_2+A_{2,2}+A_{2,3}+A_{2,4})\text{ if }v\in\big[1,\frac{6}{5}\big].
\end{cases}\\
\hspace*{-0.9cm}&{\text{\bf c). }} & P(v)=\begin{cases}
-K_S-vE_2-\frac{v}{6} (3E_1+4E_3+2E_4)\text{ if }v\in[0,1],\\
-K_S-vE_2-\frac{v}{6} (3E_1+4E_3+2E_4)-(v-1)(2A_{2,1}+B_{2,1}+2A_{2,2}+B_{2,2}+A_{2,3})\text{ if }v\in\big[1,\frac{6}{5}\big].
\end{cases} \\
\hspace*{-0.9cm}&&N(v)=\begin{cases}
\frac{v}{6} (3E_1+4E_3+2E_4)\text{ if }v\in[0,1],\\
\frac{v}{6} (3E_1+4E_3+2E_4)+(v-1)(2A_{2,1}+B_{2,1}+2A_{2,2}+B_{2,2}+A_{2,3})\text{ if }v\in\big[1,\frac{6}{5}\big].
\end{cases}
\\
\hspace*{-0.9cm}&{\text{\bf d). }} & P(v)=\begin{cases}
-K_S-vE_2-\frac{v}{6} (3E_1+4E_3+2E_4)\text{ if }v\in[0,1],\\
-K_S-vE_2-\frac{v}{6} (3E_1+4E_3+2E_4)-(v-1)(3A_{2,1}+2B_{2}+C_2+2A_{2,2}+A_{2,3})\text{ if }v\in\big[1,\frac{6}{5}\big].
\end{cases} \\
\hspace*{-0.9cm}&&N(v)=\begin{cases}
\frac{v}{6} (3E_1+4E_3+2E_4)\text{ if }v\in[0,1],\\
\frac{v}{6} (3E_1+4E_3+2E_4)+(v-1)(3A_{2,1}+2B_{2}+C_2+A_{2,2}+A_{2,3})\text{ if }v\in\big[1,\frac{6}{5}\big].
\end{cases}\\
\hspace*{-0.9cm}&{\text{\bf e). }} & P(v)=\begin{cases}
-K_S-vE_2-\frac{v}{6} (3E_1+4E_3+2E_4)\text{ if }v\in[0,1],\\
-K_S-vE_2-\frac{v}{6} (3E_1+4E_3+2E_4)-(v-1)(3A_{2,1}+2B_{2,1}+C_{2,1}+2A_{2,2}+B_{2,2})\text{ if }v\in\big[1,\frac{6}{5}\big].
\end{cases} \\
\hspace*{-0.9cm}&&N(v)=\begin{cases}
\frac{v}{6} (3E_1+4E_3+2E_4)\text{ if }v\in[0,1],\\
\frac{v}{6} (3E_1+4E_3+2E_4)+(v-1)(3A_{2,1}+2B_{2,1}+C_{2,1}+2A_{2,2}+B_{2,2})\text{ if }v\in\big[1,\frac{6}{5}\big],\\
\end{cases}\\
\hspace*{-0.9cm}&{\text{\bf f). }} & P(v)=\begin{cases}
-K_S-vE_2-\frac{v}{6} (3E_1+4E_3+2E_4)\text{ if }v\in[0,1],\\
-K_S-vE_2-\frac{v}{6} (3E_1+4E_3+2E_4)-(v-1)(4A_{2,1}+3B_{2}+2C_{2}+D_2+A_{2,2})\text{ if }v\in\big[1,\frac{6}{5}\big].
\end{cases} \\
\hspace*{-0.9cm}&&N(v)=\begin{cases}
\frac{v}{6} (3E_1+4E_3+2E_4)\text{ if }v\in[0,1],\\
\frac{v}{6} (3E_1+4E_3+2E_4)+(v-1)(4A_{2,1}+3B_{2}+2C_{2}+D_2+A_{2,2})\text{ if }v\in\big[1,\frac{6}{5}\big].
\end{cases}\\
\hspace*{-0.9cm}&{\text{\bf g). }} & P(v)=\begin{cases}
-K_S-vE_2-\frac{v}{6} (3E_1+4E_3+2E_4)\text{ if }v\in[0,1],\\
-K_S-vE_2-\frac{v}{6} (3E_1+4E_3+2E_4)-(v-1)(5A_{2}+4B_{2}+3C_{2}+2D_2+F_2)\text{ if }v\in\big[1,\frac{6}{5}\big].
\end{cases} \\
\hspace*{-0.9cm}&&N(v)=\begin{cases}
\frac{v}{6} (3E_1+4E_3+2E_4)\text{ if }v\in[0,1],\\
\frac{v}{6} (3E_1+4E_3+2E_4)+(v-1)(5A_{2}+4B_{2}+3C_{2}+2D_2+F_2)\text{ if }v\in\big[1,\frac{6}{5}\big].
\end{cases}
\end{align*}}}
 The Zariski Decomposition in part a). follows from 
 $$-K_S-vE_2\sim_{\DR} \Big(\frac{6}{5}-v\Big)E_2+\frac{1}{5}\Big(3E_1 + 4E_3 + 2E_4+A_{2,1}+A_{2,2}+A_{2,3}+A_{2,4}+A_{2,5} \Big).$$ 
 A similar statement holds in other parts. 
Moreover, 
$$(P(v))^2=\begin{cases}
1-\frac{5v^2}{6}  \text{ if }v\in[0,1],\\
\frac{(6-5v)^2}{6}  \text{ if }v\in\big [1, \frac{6}{5}\big ].
\end{cases}
P(v)\cdot E_2=\begin{cases}
\frac{5v}{6}  \text{ if }v\in[0,1],\\
3(1-v)  \text{ if }v\in\big [1, \frac{6}{5}\big ].
\end{cases}$$
We have $S_{S} (E_2)=\frac{11}{15}$. Thus, $\delta_P(S)\le \frac{15}{11}$ for $P\in E_2\backslash E_3$. Moreover, if $P\in E_2\cap E_1$ or if $P\in E_2\backslash E_1$ for such points we have 
$$h(v)\le\begin{cases}
\frac{55v^2}{72} \text{ if }v\in [0,1],\\
\frac{5(5v - 6)(19v - 30)}{72} \text{ if }v\in \big [1, \frac{6}{5}\big ].
\end{cases}
\text{ or }
h(v)\le\begin{cases}
\frac{25v^2}{72} \text{ if }v\in [0,1],\\
\frac{25(5v - 6)(6-7v)}{72} \text{ if }v\in \big [1, \frac{6}{5}\big ].
\end{cases}
$$
Thus,
$S(W_{\bullet,\bullet}^{E_2};P)\le \frac{29}{45}<\frac{11}{15}$ or
$S(W_{\bullet,\bullet}^{E_2};P)\le\frac{1}{3}<\frac{11}{15}$.
We get $\delta_P(S)=\frac{15}{11}$ for $P\in (E_1\cup E_2)\backslash (E_1\cap E_2)$.
\par{\bf Step 3.} Suppose $P\in E_1\cup E_4$. Without loss of generality we can assume that $P\in E_1$ since the proof is similar in other cases.  Then $\tau(E_1)=1$ and the  Zariski decomposition of the divisor $-K_S-vE_1\sim C+(1-v)E_1+E_2+E_3+E_4$ is given by:
 {\allowdisplaybreaks\begin{align*}
&&P(v)=\begin{cases}
-K_S-vE_1-\frac{v}{4} (3E_2+2E_3+E_4) \text{ if }v\in\big [0,\frac{4}{5}\big ],\\
-K_S-vE_1-(2v-1)E_2-(3v-2)E_3-(4v-3)E_4-(5v-4)C \text{ if } v\in\big [\frac{4}{5}, 1\big ].
\end{cases}\\&&
N(v)=
\begin{cases}
\frac{v}{4} (3E_2+2E_3+E_4) \text{ if }v\in\big [0,\frac{4}{5}\big ],\\
(2v-1)E_2+(3v-2)E_3+(4v-3)E_4+(5v-4)C \text{ if }v\in\big [\frac{4}{5},1\big ].
\end{cases}
\end{align*}}
Moreover, 
$$(P(v))^2=\begin{cases}
1-\frac{5v^2}{4}  \text{ if }v\in\big [0,\frac{4}{5}\big ],\\
5(v - 1)^2  \text{ if }v\in\big [\frac{4}{5},1\big ].
\end{cases}
P(v)\cdot E_1=\begin{cases}
\frac{5v}{4}  \text{ if }v\in\big [0,\frac{4}{5}\big ],\\
5(1-v)  \text{ if }v\in\big [\frac{4}{5},1\big ].
\end{cases}$$
We have
$S_{S} (E_1)=\frac{3}{5}$
Thus, $\delta_P(S)\le \frac{5}{3}$ for $P\in E_1\backslash E_2$. Moreover,  for such points we have 
$$h(v)\le\begin{cases}
\frac{25v^2}{32} \text{ if }v\in \big [0, \frac{4}{5}\big ],\\
\frac{5(1 - v )(5v - 3)}{2} \text{ if }v\in \big [\frac{4}{5},1\big ].
\end{cases}$$
Thus,
$S(W_{\bullet,\bullet}^{E_1};P)\le \frac{2}{5}<\frac{3}{5}$.
We get $\delta_P(S)=\frac{5}{3}$ for $P\in (E_1\cup E_4)\backslash (E_2\cup E_3)$. 
Thus, $\delta_{\mathcal{P}} (X)=\frac{4}{3}$.
\end{proof}
\subsubsection{$\DA_5$ singularity  on Du Val Del Pezzo surfaces of degree $1$}\label{dP1-A5}
 \begin{lemma} Let $X$ be a singular del Pezzo surface of degree $1$ with an $\DA_5$ singularity at point $\mathcal{P}$.  Let $\mathcal{C}$ be a~curve in the~pencil $|-K_X|$ that contains~$\mathcal{P}$.  Then $\delta_{\mathcal{P}} (X)=\frac{6}{5}$.
 \end{lemma}
 \begin{proof}
 Let $S$ be the minimal resolution of singularities.  Then $S$ is a weak del Pezzo surface of degree $1$. Suppose $C$ is a strict transform of $\mathcal{C}$ on $S$ and $E_1$, $E_2$, $E_3$, $E_4$ and $E_5$ are the exceptional divisors with the intersection:
 \begin{figure}[h!]
    \centering
\includegraphics[width=7cm]{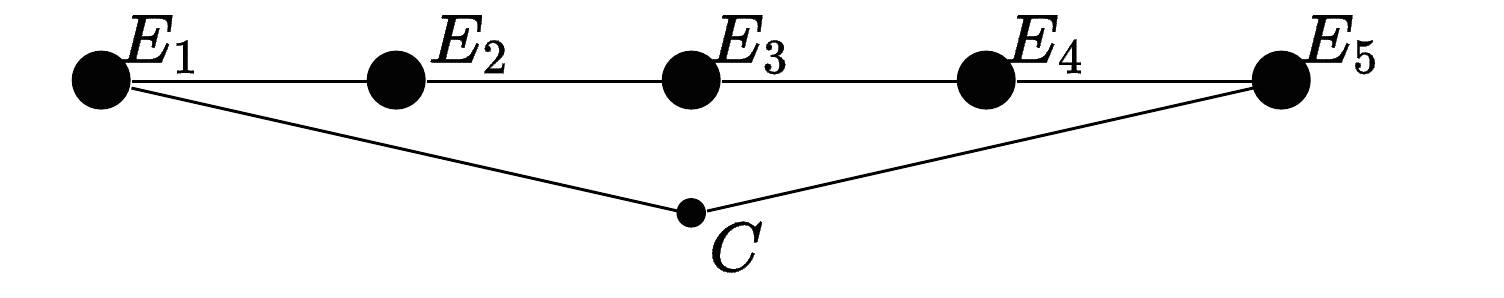}
    \caption{Dual graph: $(-K_S)^2=1$, $\DA_5$ singularity}
\end{figure}
\par We have $-K_S\sim C+E_1+E_2+E_3+E_4+E_5$. Let $P$ be a point on $S$.\\
{\bf Step 1.} Suppose $P\in  E_3$.  If  we contract the curve $C$  the resulting
surface is isomorphic to a weak del Pezzo surface of degree $2$  with at most Du Val singularities.  Thus, there exist $(-1)$-curves and $(-2)$-curves  which form one of the following dual graphs:
\begin{figure}[h!]
    \centering
\includegraphics[width=12cm]{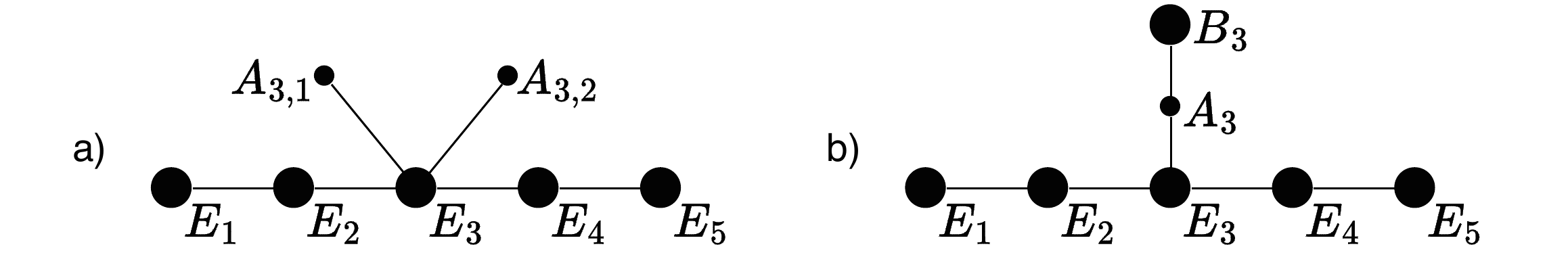}
    \caption{Dual graph: $(-K_S)^2=1$, $\DA_5$ singularity, $\delta_P(S)=\frac{6}{5}$}
\end{figure}
\par   Then $\tau(E_3)=\frac{3}{2}$ and the  Zariski Decomposition of the divisor $-K_S-vE_3$ is:
   { 
 {\allowdisplaybreaks\begin{align*}
&{\text{\bf a). }} & P(v)=\begin{cases}
-K_S-vE_3-\frac{v}{3} (E_1+2E_2+2E_4+E_5)\text{ if }v\in[0,1],\\
-K_S-vE_3-\frac{v}{3} (E_1+2E_2+2E_4+E_5)-(v-1)(A_{3,1}+A_{3,2})\text{ if }v\in\big[1,\frac{3}{2}\big].
\end{cases} \\
&&N(v)=\begin{cases}
\frac{v}{3} (E_1+2E_2+2E_4+E_5)\text{ if }v\in[0,1],\\
\frac{v}{3} (E_1+2E_2+2E_4+E_5)+(v-1)(A_{3,1}+A_{3,2})\text{ if }v\in\big[1,\frac{3}{2}\big].
\end{cases}\\
&{\text{\bf b). }} & P(v)=\begin{cases}
-K_S-vE_3-\frac{v}{3} (E_1+2E_2+2E_4+E_5)\text{ if }v\in[0,1],\\
-K_S-vE_3-\frac{v}{3} (E_1+2E_2+2E_4+E_5)-(v-1)(2A_{3}+B_3)\text{ if }v\in\big[1,\frac{3}{2}\big].
\end{cases} \\
&&N(v)=\begin{cases}
\frac{v}{3} (E_1+2E_2+2E_4+E_5)\text{ if }v\in[0,1],\\
\frac{v}{3} (E_1+2E_2+2E_4+E_5)+(v-1)(2A_{3}+B_3)\text{ if }v\in\big[1,\frac{3}{2}\big].
\end{cases}
\end{align*}}}
 The Zariski Decomposition in part a). follows from 
 $$-K_S-vE_3\sim_{\DR} \Big(\frac{3}{2}-v\Big)E_3+\frac{1}{2}\Big(E_1+2E_2+2E_4+E_5+A_{3,1}+A_{3,2}\Big).$$ 
 A similar statement holds in other parts. 
Moreover, 
$$(P(v))^2=\begin{cases}
1-\frac{2v^2}{3}  \text{ if }v\in[0,1],\\
\frac{(3-2v)^2}{3}  \text{ if }v\in\big [1, \frac{3}{2}\big ].
\end{cases}
P(v)\cdot E_3=\begin{cases}
\frac{2v}{3}  \text{ if }v\in[0,1],\\
2(1-\frac{2v}{3})  \text{ if }v\in\big [1, \frac{3}{2}\big ].
\end{cases}$$
We have
$S_{S} (E_3)=\frac{5}{6}$. Thus, $\delta_P(S)\le \frac{6}{5}$ for $P\in E_3$. Moreover, if $P\in E_3\cap (E_2\cup E_4)$ or if $P\in E_3\backslash (E_2\cup E_4)$  we have 
$$h(v)\le\begin{cases}
\frac{2v^2}{3} \text{ if }v\in [0,1],\\
\frac{2(3-2v)}{3} \text{ if }v\in\big [1, \frac{3}{2}\big ].
\end{cases}
\text{ or }
h(v)\le\begin{cases}
\frac{2v^2}{9} \text{ if }v\in [0,1],\\
\frac{2(3 -2v)(4v - 3)}{9} \text{ if }v\in\big [1, \frac{3}{2}\big ].
\end{cases}
$$
Thus,
$S(W_{\bullet,\bullet}^{E_3};P)\le \frac{7}{9}<\frac{5}{6}$
or
$S(W_{\bullet,\bullet}^{E_3};P)\le  \frac{1}{3}<\frac{5}{6}$.
We get $\delta_P(S)=\frac{6}{5}$ for $P\in E_3$.
\par{\bf Step 2.} Suppose $P\in  E_2\cup E_4$. Without loss of generality we can assume that
$P \in E_2$ since the proof is similar in other cases. There exist $(-1)$-curves and $(-2)$-curves  which form one of the following dual graphs:
\begin{figure}[h!]
    \centering
\includegraphics[width=16cm]{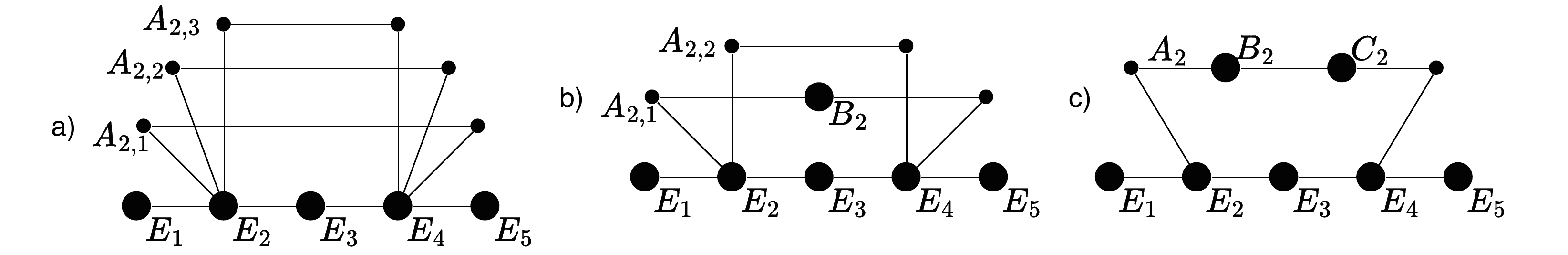}
    \caption{Dual graph: $(-K_S)^2=1$, $\DA_5$ singularity, $\delta_P(S)=\frac{9}{7}$}
\end{figure}
\par Then $\tau(E_2)=\frac{4}{3}$ and the  Zariski Decomposition of the divisor $-K_S-vE_2$ is:
   { 
 {\allowdisplaybreaks\begin{align*}
&{\text{\bf a). }} & P(v)=\begin{cases}
-K_S-vE_2-\frac{v}{4} (2E_1+3E_3+2E_4+E_5)\text{ if }v\in[0,1],\\
-K_S-vE_2-\frac{v}{4} (2E_1+3E_3+2E_4+E_5)-(v-1)(A_{2,1}+A_{2,2}+A_{2,3})\text{ if }v\in\big[1,\frac{4}{3}\big].
\end{cases} \\
&&N(v)=\begin{cases}
\frac{v}{4} (2E_1+3E_3+2E_4+E_5)\text{ if }v\in[0,1],\\
\frac{v}{4} (2E_1+3E_3+2E_4+E_5)+(v-1)(A_{2,1}+A_{2,2}+A_{2,3})\text{ if }v\in\big[1,\frac{4}{3}\big].
\end{cases}\\
&{\text{\bf b). }} & P(v)=\begin{cases}
-K_S-vE_2-\frac{v}{4} (2E_1+3E_3+2E_4+E_5)\text{ if }v\in[0,1],\\
-K_S-vE_2-\frac{v}{4} (2E_1+3E_3+2E_4+E_5)-(v-1)(2A_{2,1}+B_{2,1}+A_{2,2})\text{ if }v\in\big[1,\frac{4}{3}\big].
\end{cases} \\
&&N(v)=\begin{cases}
\frac{v}{4} (2E_1+3E_3+2E_4+E_5)\text{ if }v\in[0,1],\\
\frac{v}{4} (2E_1+3E_3+2E_4+E_5)+(v-1)(2A_{2,1}+B_{2,1}+A_{2,2})\text{ if }v\in\big[1,\frac{4}{3}\big].
\end{cases}\\
&{\text{\bf c). }} & P(v)=\begin{cases}
-K_S-vE_2-\frac{v}{4} (2E_1+3E_3+2E_4+E_5)\text{ if }v\in[0,1],\\
-K_S-vE_2-\frac{v}{4} (2E_1+3E_3+2E_4+E_5)-(v-1)(3A_{2}+B_{2}+C_2)\text{ if }v\in\big[1,\frac{4}{3}\big].
\end{cases} \\
&&N(v)=\begin{cases}
\frac{v}{4} (2E_1+3E_3+2E_4+E_5)\text{ if }v\in[0,1],\\
\frac{v}{4} (2E_1+3E_3+2E_4+E_5)+(v-1)(3A_{2}+B_{2}+C_2)\text{ if }v\in\big[1,\frac{4}{3}\big].
\end{cases}
\end{align*}}}
The Zariski Decomposition in part a). follows from 
 $$-K_S-vE_2\sim_{\DR} \Big(\frac{4}{3}-v\Big)E_2+\frac{1}{3}\Big(2E_1+3E_3+2E_4+E_5+A_{2,1}+A_{2,2}+A_{2,3}\Big).$$ 
 A similar statement holds in other parts. Moreover, 
$$(P(v))^2=\begin{cases}
1-\frac{3v^2}{4}  \text{ if }v\in[0,1],\\
\frac{(4-3v)^2}{4}  \text{ if }v\in\big [1, \frac{4}{3}\big ].
\end{cases}
P(v)\cdot E_2=\begin{cases}
\frac{3v}{4}  \text{ if }v\in[0,1],\\
3(1-\frac{3v}{4})  \text{ if }v\in\big [1, \frac{4}{3}\big ].
\end{cases}$$
We have
$S_{S} (E_2)=\frac{7}{9}$. Thus, $\delta_P(S)\le \frac{9}{7}$ for $P\in E_2$. Moreover, if $P\in E_2\cap E_1$ or if $P\in E_2\backslash (E_1\cup E_3)$  we have 
$$h(v)\le\begin{cases}
\frac{21v^2}{32} \text{ if }v\in [0,1],\\
\frac{3(3v - 4)(5v - 12)}{32} \text{ if }v\in \big [1, \frac{4}{3}\big ].
\end{cases}
\text{ or }
h(v)\le\begin{cases}
\frac{9v^2}{32} \text{ if }v\in [0,1],\\
\frac{9(3v - 4)(4 -5v )}{32} \text{ if }v\in \big [1, \frac{4}{3}\big ].
\end{cases}
$$
Thus,
$S(W_{\bullet,\bullet}^{E_2};P)\le \frac{23}{36}<\frac{7}{9}$
or
$S(W_{\bullet,\bullet}^{E_2};P)\le  \frac{1}{3}<\frac{7}{9}$.
We get $\delta_P(S)=\frac{9}{7}$ for $P\in (E_2\cup E_4)\backslash E_3$.
\par{\bf Step 3.} Suppose $P\in E_1\cup E_5$. Without loss of generality we can assume that $P\in E_1$ since the proof is similar in other cases.  Then $\tau(E_1)=1$ and the  Zariski decomposition of the divisor $-K_S-vE_1\sim  C+(1-v)E_1+E_2+E_3+E_4+E_5$ is given by:
 {\allowdisplaybreaks\begin{align*}
\hspace*{-0.5cm}&&P(v)=\begin{cases}
-K_S-vE_1-\frac{v}{5} (4E_2+3E_3+2E_4+E_5) \text{ if }v\in\big [0,\frac{5}{6}\big ],\\
-K_S-vE_1-(2v-1)E_2-(3v-2)E_3-(4v-3)E_4-(5v-4)E_5-(6v-5)C \text{ if } v\in\big [\frac{5}{6}, 1\big ].
\end{cases}\\\hspace*{-0.5cm}&&
N(v)=
\begin{cases}
\frac{v}{5} (4E_2+3E_3+2E_4+E_5) \text{ if }v\in\big [0,\frac{5}{6}\big ],\\
(2v-1)E_2+(3v-2)E_3+(4v-3)E_4+(5v-4)E_5+(6v-5)C \text{ if }v\in\big [\frac{5}{6},1\big ].
\end{cases}
\end{align*}}
Moreover, 
$$(P(v))^2=\begin{cases}
1-\frac{6v^2}{5}  \text{ if }v\in\big [0,\frac{5}{6}\big ],\\
6(v - 1)^2  \text{ if }v\in\big [\frac{5}{6},1\big ].
\end{cases}
P(v)\cdot E_1=\begin{cases}
\frac{6v}{5}  \text{ if }v\in\big [0,\frac{5}{6}\big ],\\
6(1-v)  \text{ if }v\in\big [\frac{5}{6},1\big ].
\end{cases}$$
We have
$S_{S} (E_1)=\frac{11}{18}$. Thus, $\delta_P(S)\le \frac{18}{11}$ for $P\in E_1\backslash E_2$. Moreover,  for such points we have 
$$h(v)\le\begin{cases}
\frac{18v^2}{25} \text{ if }v\in \big [0, \frac{5}{6}\big ],\\
6(1- v)(3v-2) \text{ if } v\in\big [\frac{5}{6}, 1\big ].
\end{cases}$$
Thus,
$S(W_{\bullet,\bullet}^{E_1};P)\le \frac{7}{18}<\frac{11}{18}$.
We get $\delta_P(S)=\frac{18}{11}$ for $P\in (E_1\cup E_5)\backslash (E_2\cup E_4)$. \\ Thus, $\delta_{\mathcal{P}} (X)=\frac{6}{5}$.
 \end{proof}
 \subsubsection{$\DA_6$ singularity  on Du Val Del Pezzo surfaces of degree $1$}
 \begin{lemma} Let $X$ be a singular del Pezzo surface of degree $1$ with an $\DA_6$ singularity at point $\mathcal{P}$.  Let $\mathcal{C}$ be a~curve in the~pencil $|-K_X|$ that contains~$\mathcal{P}$. Then $\delta_{\mathcal{P}} (X)=\frac{9}{8}$.
 \end{lemma}
  \begin{proof}
  Let $S$ be the minimal resolution of singularities.  Then $S$ is a weak del Pezzo surface of degree $1$. Suppose $C$ is a strict transform of $\mathcal{C}$ on $S$ and $E_1$, $E_2$, $E_3$, $E_4$, $E_6$ and $E_7$ are the exceptional divisors with the intersection:
  \begin{figure}[h!]
    \centering
\includegraphics[width=10cm]{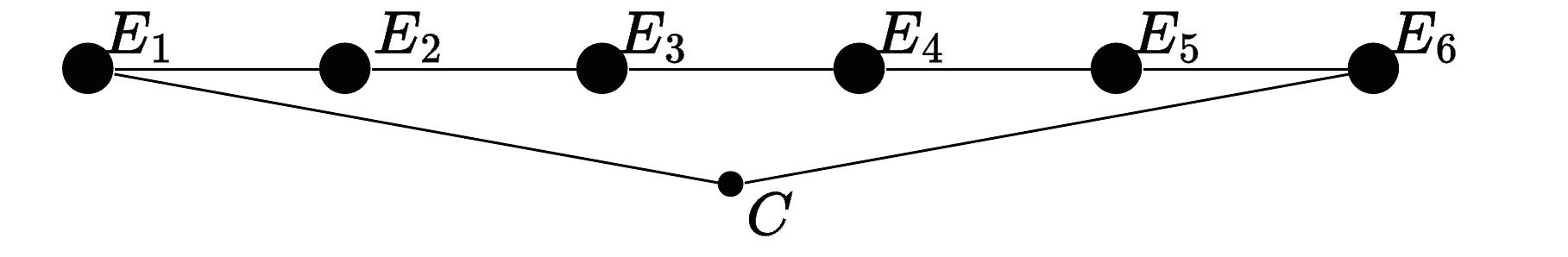}
    \caption{Dual graph: $(-K_S)^2=1$, $\DA_6$ singularity}
\end{figure}
\par  We have $-K_S\sim C+E_1+E_2+E_3+E_4+E_5+E_6$. Let $P$ be a point on $S$.
\par{\bf Step 1.} Suppose $P\in  E_3\cup E_4$. Without loss of generality we can assume that $P\in E_3$ since the
proof is similar in other cases. There exist $(-1)$-curves and $(-2)$-curves  which form the following dual graph:
\begin{figure}[h!]
    \centering
\includegraphics[width=7cm]{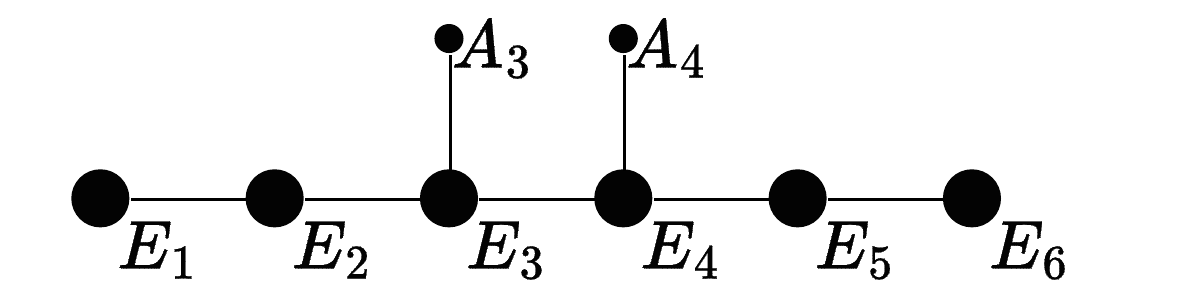}
    \caption{Dual graph: $(-K_S)^2=1$, $\DA_6$ singularity, $\delta_P(S)=\frac{9}{8}$}
\end{figure}
\par  Then $\tau(E_3)=\frac{3}{2}$ and the  Zariski Decomposition of the divisor $-K_S-vE_3$ is:
 {\allowdisplaybreaks\begin{align*}
\hspace*{-0.5cm}&&P(v)=\begin{cases}
-K_S-vE_3-\frac{v}{12} (4E_1+8E_2+9E_4+6E_5+3E_6)\text{ if }v\in[0,1],\\
-K_S-vE_3-\frac{v}{12} (4E_1+8E_2+9E_4+6E_5+3E_6)-(v-1)A_3\text{ if }v\in\big[1,\frac{4}{3}\big],\\
-K_S-vE_3-\frac{v}{3} (E_1+2E_2)-(v-1)(3E_4+2E_5+E_6+A_3)-(3v-4)A_4\text{ if }v\in\big[\frac{4}{3},\frac{3}{2}\big].
\end{cases}\\
\hspace*{-0.5cm}&&N(v)=\begin{cases}
\frac{v}{12} (4E_1+8E_2+9E_4+6E_5+3E_6)\text{ if }v\in[0,1],\\
\frac{v}{12} (4E_1+8E_2+9E_4+6E_5+3E_6)+(v-1)A_3\text{ if }v\in\big[1,\frac{4}{3}\big],\\
\frac{v}{3} (E_1+2E_2)+(v-1)(3E_4+2E_5+E_6+A_3)+(3v-4)A_4\text{ if }v\in\big[\frac{4}{3},\frac{3}{2}\big].
\end{cases}
\end{align*}}
\noindent The Zariski Decomposition in part a). follows from 
 $$-K_S-vE_3\sim_{\DR} \Big(\frac{3}{2}-v\Big)E_3+\frac{1}{2}\Big(E_1+2E_2+3E_4+2E_5+E_6+A_3+A_4\Big).$$ 
Moreover, 
$$(P(v))^2=\begin{cases}
1-\frac{7v^2}{12}  \text{ if }v\in[0,1],\\
2 - 2v + \frac{5v^2}{12}\text{ if }v\in\big[1,\frac{4}{3}\big],\\
\frac{2(3-2v)^2}{3}  \text{ if }v\in\big[\frac{4}{3},\frac{3}{2}\big].
\end{cases}
P(v)\cdot E_3=\begin{cases}
\frac{7v}{12}  \text{ if }v\in[0,1],\\
1-\frac{5v}{12}\text{ if }v\in\big[1,\frac{4}{3}\big],\\
4(1-\frac{2v}{3})  \text{ if }v\in\big[\frac{4}{3},\frac{3}{2}\big].
\end{cases}$$
We have
$S_{S} (E_3)=\frac{8}{9}$.
Thus, $\delta_P(S)\le \frac{9}{8}$ for $P\in E_3$. Moreover, if $P\in E_3\cap A_3$ or if $P\in E_3\cap E_2$ or if $P\in E_3\backslash (E_2\cup A_3)$  we have 
{  $$h(v)\le\begin{cases}
\frac{49v^2}{288}  \text{ if }v\in[0,1],\\
\frac{(12 - 5v)(19v - 12)}{288}\text{ if }v\in\big[1,\frac{4}{3}\big],\\
\frac{4(2v - 3)(v-3 )}{9}  \text{ if }v\in\big[\frac{4}{3},\frac{3}{2}\big].
\end{cases}
\text{ or }
h(v)\le\begin{cases}
\frac{161v^2}{288}  \text{ if }v\in[0,1],\\
\frac{(12 - 5v)(11v + 12)}{288}\text{ if }v\in\big[1,\frac{4}{3}\big],\\
\frac{8(2v - 3)(v-3 )}{9}  \text{ if }v\in\big[\frac{4}{3},\frac{3}{2}\big].
\end{cases}$$$$
\text{ or }
h(v)\le\begin{cases}
\frac{175v^2}{288}  \text{ if }v\in[0,1],\\
\frac{(12 - 5v)(13v + 12)}{288}\text{ if }v\in\big[1,\frac{4}{3}\big],\\
\frac{4(2v - 3)(5v-3 )}{9}  \text{ if }v\in\big[\frac{4}{3},\frac{3}{2}\big].
\end{cases}
$$}
\noindent Thus,
$S(W_{\bullet,\bullet}^{E_3};P)\le \frac{8}{27}<\frac{8}{9}$
or
$S(W_{\bullet,\bullet}^{E_3};P)\le \frac{29}{36}<\frac{8}{9}$
or
$S(W_{\bullet,\bullet}^{E_3};P)\le \frac{8}{9}$.
We get $\delta_P(S)=\frac{9}{8}$ for $P\in E_3\cup E_4$.
\par{\bf Step 2.} Suppose $P\in  E_2\cup E_5$. Without loss of generality we can assume that
$P \in E_2$ since the proof is similar in other cases. There exist $(-1)$-curves and $(-2)$-curves  which form one of the following dual graphs:
\begin{figure}[h!]
    \centering
\includegraphics[width=14cm]{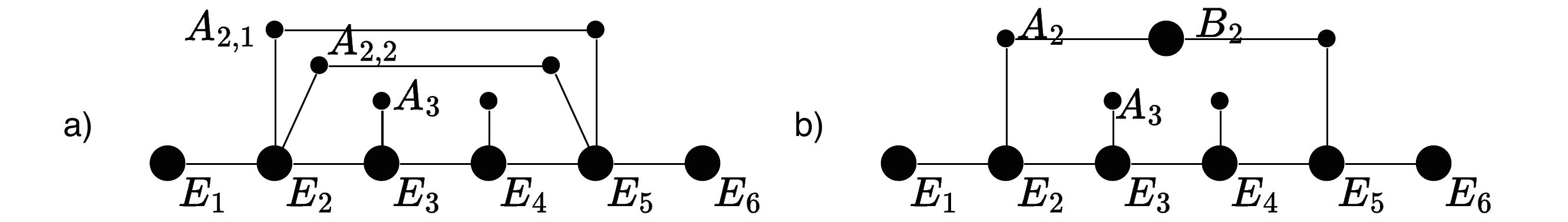}
    \caption{Dual graph: $(-K_S)^2=1$, $\DA_6$ singularity, $\delta_P(S)=\frac{36}{29}$}
\end{figure}
\par     Then $\tau(E_2)=\frac{4}{3}$ and the  Zariski Decomposition of the divisor $-K_S-vE_2$ is:
   { 
 {\allowdisplaybreaks\begin{align*}
\hspace*{-0.9cm}&{\text{\bf a). }} & P(v)=\begin{cases}
-K_S-vE_2-\frac{v}{10} (5E_1+8E_3+6E_4+4E_5+2E_6)\text{ if }v\in[0,1],\\
-K_S-vE_2-\frac{v}{10} (5E_1+8E_3+6E_4+4E_5+2E_6)-(v-1)(A_{2,1}+A_{2,2})\text{ if }v\in\big[1,\frac{5}{4}\big],\\
-K_S-vE_2-\frac{v}{2}E_1-(v-1)(4E_3+3E_4+2E_5+E_6+A_{2,1}+A_{2,2})-(4v-5)A_3\text{ if }v\in\big[\frac{5}{4},\frac{4}{3}\big].
\end{cases} \\
\hspace*{-0.9cm}&&N(v)=\begin{cases}
\frac{v}{10} (5E_1+8E_3+6E_4+4E_5+2E_6)\text{ if }v\in[0,1],\\
\frac{v}{10} (5E_1+8E_3+6E_4+4E_5+2E_6)+(v-1)(A_{2,1}+A_{2,2})\text{ if }v\in\big[1,\frac{5}{4}\big],\\
\frac{v}{2}E_1+(v-1)(4E_3+3E_4+2E_5+E_6+A_{2,1}+A_{2,2})+(4v-5)A_3\text{ if }v\in\big[\frac{5}{4},\frac{4}{3}\big].
\end{cases}\\
\hspace*{-0.9cm}&{\text{\bf b). }} & P(v)=\begin{cases}
-K_S-vE_2-\frac{v}{10} (5E_1+8E_3+6E_4+4E_5+2E_6)\text{ if }v\in[0,1],\\
-K_S-vE_2-\frac{v}{10} (5E_1+8E_3+6E_4+4E_5+2E_6)-(v-1)(2A_{2}+B_2)\text{ if }v\in\big[1,\frac{5}{4}\big],\\
-K_S-vE_2-\frac{v}{2}E_1-(v-1)(4E_3+3E_4+2E_5+E_6+2A_{2}+B_2)-(4v-5)A_3\text{ if }v\in\big[\frac{5}{4},\frac{4}{3}\big].
\end{cases} \\
\hspace*{-0.9cm}&&N(v)=\begin{cases}
\frac{v}{10} (5E_1+8E_3+6E_4+4E_5+2E_6)\text{ if }v\in[0,1],\\
\frac{v}{10} (5E_1+8E_3+6E_4+4E_5+2E_6)+(v-1)(2A_{2}+B_2)\text{ if }v\in\big[1,\frac{5}{4}\big],\\
\frac{v}{2}E_1+(v-1)(4E_3+3E_4+2E_5+E_6+2A_{2}+B_2)+(4v-5)A_3\text{ if }v\in\big[\frac{5}{4},\frac{4}{3}\big].
\end{cases}
\end{align*}}}
The Zariski Decomposition in part a). follows from 
 $$-K_S-vE_2\sim_{\DR} \Big(\frac{4}{3}-v\Big)E_2+\frac{1}{3}\Big(2E_1+4E_3+3E_4+2E_5+E_6+A_{2,1}+A_{2,2}+A_3\Big).$$ 
 A similar statement holds in other parts. Moreover, 
$$(P(v))^2=\begin{cases}
1-\frac{7v^2}{10}  \text{ if }v\in[0,1],\\
3 - 4v + \frac{13v^2}{10}\text{ if }v\in\big[1,\frac{5}{4}\big],\\
\frac{(4-3v)^2}{2}  \text{ if }v\in\big [\frac{5}{4}, \frac{4}{3}\big ].
\end{cases}
P(v)\cdot E_2=\begin{cases}
\frac{7v}{10}  \text{ if }v\in[0,1],\\
2 - \frac{13v}{10}\text{ if }v\in\big[1,\frac{5}{4}\big],\\
3(2-\frac{3v}{2})  \text{ if }v\in\big [\frac{5}{4}, \frac{4}{3}\big ].
\end{cases}$$
We have $S_{S} (E_2)=\frac{29}{36}$. Thus, $\delta_P(S)\le \frac{36}{29}$ for $P\in E_2$. Moreover, if $P\in E_2\cap E_1$ or if $P\in E_2\backslash (E_1\cup E_3)$  we have 
$$h(v)\le\begin{cases}
\frac{119v^2}{200} \text{ if }v\in [0,1],\\
\frac{(13v-20)(3v - 20)}{200}\text{ if }v\in\big[1,\frac{5}{4}\big],\\
\frac{3(3v - 4)(7v - 12)}{8} \text{ if }v\in\big [\frac{5}{4}, \frac{4}{3}\big ].
\end{cases}
\text{ or }
h(v)\le
\begin{cases}
\frac{49v^2}{200} \text{ if }v\in [0,1],\\
\frac{(13v-20)(27v - 20)}{200}\text{ if }v\in\big[1,\frac{5}{4}\big],\\
\frac{3(3v - 4)(v - 4)}{8} \text{ if }v\in\big [\frac{5}{4}, \frac{4}{3}\big ].
\end{cases}
$$
Thus
$S(W_{\bullet,\bullet}^{E_2};P)\le \frac{29}{45}<\frac{29}{36}$
or
$S(W_{\bullet,\bullet}^{E_2};P)\le \frac{23}{72}<\frac{29}{36}$.
We get $\delta_P(S)=\frac{36}{29}$ for $P\in (E_2\cup E_5)\backslash (E_3\cup E_4)$.
\par{\bf Step 3.} Suppose $P\in E_1\cup E_6$. Without loss of generality we can assume that $P\in E_1$ since the proof is similar in other cases.  Then $\tau(E_1)=1$ and the  Zariski decomposition of the divisor $-K_S-vE_1\sim C+(1-v)E_1+E_2+E_3+E_4+E_5+E_6$ is given by:
{ 
 {\small\allowdisplaybreaks\begin{align*}
\hspace*{-0.7cm}&&P(v)=\begin{cases}
-K_S-vE_1-\frac{v}{6} (5E_2+4E_3+3E_4+2E_5+E_6) \text{ if }v\in\big [0,\frac{6}{7}\big ],\\
-K_S-vE_1-(2v-1)E_2-(3v-2)E_3-(4v-3)E_4-(5v-4)E_5-(6v-5)E_6-(7v-6)C \text{ if } v\in\big [\frac{6}{7}, 1\big ].
\end{cases}\\
\hspace*{-0.7cm}&&N(v)=
\begin{cases}
\frac{v}{6} (5E_2+4E_3+3E_4+2E_5+E_6) \text{ if }v\in\big [0,\frac{6}{7}\big ],\\
(2v-1)E_2+(3v-2)E_3+(4v-3)E_4+(5v-4)E_5+(6v-5)E_6+(7v-6)C \text{ if }v\in\big [\frac{6}{7},1\big ].
\end{cases}
\end{align*}}}
Moreover, 
$$(P(v))^2=\begin{cases}
1-\frac{7v^2}{6}  \text{ if }v\in\big [0,\frac{6}{7}\big ],\\
7(v - 1)^2  \text{ if }v\in\big [\frac{6}{7},1\big ].
\end{cases}
P(v)\cdot E_1=\begin{cases}
\frac{7v}{6}  \text{ if }v\in\big [0,\frac{6}{7}\big ],\\
7(1-v)  \text{ if }v\in\big [\frac{6}{7},1\big ].
\end{cases}$$
We have $S_{S} (E_1)=\frac{13}{21}$. Thus, $\delta_P(S)\le \frac{21}{13}$ for $P\in E_1\backslash E_2$. Moreover,  for such points we have 
$$h(v)\le\begin{cases}
\frac{49v^2}{72} \text{ if }v\in \big [0,\frac{6}{7}\big ],\\
\frac{7(1-v)(7v - 5)}{2} \text{ if }v\in \big [\frac{6}{7},1\big ].
\end{cases}$$
Thus
$S(W_{\bullet,\bullet}^{E_1};P)\le \frac{8}{21}<\frac{13}{21}$.
We get $\delta_P(S)=\frac{21}{13}$ for $P\in (E_1\cup E_6)\backslash (E_2\cup E_5)$. 
\\Thus, $\delta_{\mathcal{P}} (X)=\frac{9}{8}$.
 \end{proof}
  \subsubsection{$\DA_7$ singularity (reducible ramification)   on Du Val Del Pezzo surfaces of degree $1$}
 \begin{lemma} Let $X$ be a singular del Pezzo surface of degree $1$ with an $\DA_7$ singularity at point $\mathcal{P}$. $X$ can be realized as the double cover $X\xrightarrow{2:1}\DP(1,1,2)$, which is
ramified along a sextic curve $R\in \DP(1,1,2)$. Suppose $R$ is reducible. Let $\mathcal{C}$ be a~curve in the~pencil $|-K_X|$ that contains~$\mathcal{P}$.  Then $\delta_{\mathcal{P}} (X)=1$.
\end{lemma}
  \begin{proof}
  Let $S$ be the minimal resolution of singularities.  Then $S$ is a weak del Pezzo surface of degree $1$. Suppose $C$ is a strict transform of $\mathcal{C}$ on $S$ and $E_1$, $E_2$, $E_3$, $E_4$, $E_5$, $E_6$ and $E_7$ are the exceptional divisors with the intersection:
  \begin{figure}[h!]
    \centering
\includegraphics[width=10.5cm]{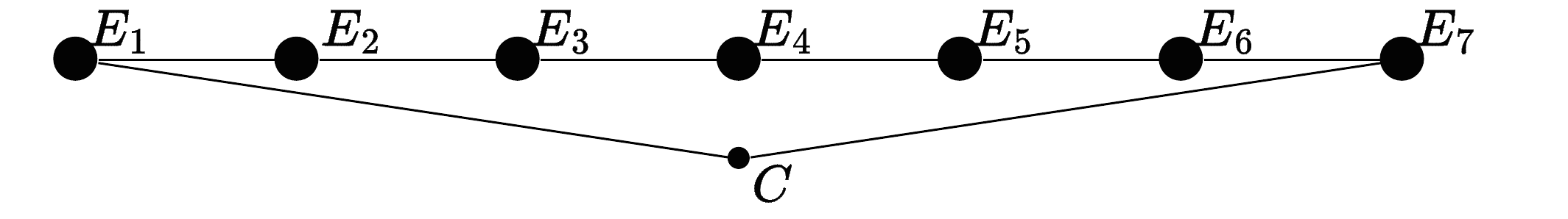}
    \caption{Dual graph: $(-K_S)^2=1$, $\DA_7$ singularity (reducible ramification divisor)}
\end{figure}
\par  We have $-K_S\sim C+E_1+E_2+E_3+E_4+E_5+E_6+E_7$. Let $P$ be a point on $S$. If the ramification divisor $R$ is reducible, then this
implies the existence of a $(-1)$-curve $A_4$ which intersects the fundamental cycle only at  $E_4$ and this intersection is transversal.\\
{\bf Step 1.} Suppose $P\in  E_4$.  There exist $(-1)$-curves and $(-2)$-curves  which form  the following dual graph:
\begin{figure}[h!]
    \centering
\includegraphics[width=8cm]{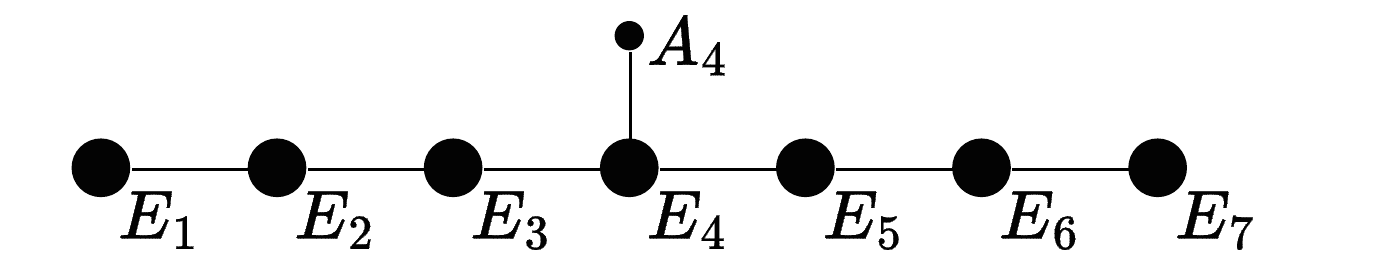}
    \caption{Dual graph: $(-K_S)^2=1$, $\DA_7$ singularity, $\delta_P(S)=1$}
\end{figure}
\par     Then $\tau(E_4)=2$ and the  Zariski Decomposition of the divisor $-K_S-vE_4$ is:
 {\allowdisplaybreaks\begin{align*}
&&P(v)=\begin{cases}
-K_S-vE_4-\frac{v}{4} (E_1+2E_2+3E_3+3E_5+2E_6+E_7)\text{ if }v\in[0,1],\\
-K_S-vE_4-\frac{v}{4} (E_1+2E_2+3E_3+3E_5+2E_6+E_7)-(v-1)A_4\text{ if }v\in[1,2].
\end{cases} \\
&&N(v)=\begin{cases}
\frac{v}{4} (E_1+2E_2+3E_3+3E_5+2E_6+E_7)\text{ if }v\in[0,1],\\
\frac{v}{4} (E_1+2E_2+3E_3+3E_5+2E_6+E_7)+(v-1)A_4\text{ if }v\in[1,2].
\end{cases}
\end{align*}}
The Zariski Decomposition  follows from 
 $$-K_S-vE_4\sim_{\DR} (2-v)E_4+\frac{1}{4}\Big(E_1+2E_2+3E_3+3E_5+2E_6+E_7+4A_4\Big).$$ 
Moreover, 
$$(P(v))^2=\begin{cases}
1-\frac{v^2}{2}  \text{ if }v\in[0,1],\\
\frac{(2-v)^2}{2}  \text{ if }v\in[1,2].
\end{cases}
P(v)\cdot E_4=\begin{cases}
\frac{v}{2}  \text{ if }v\in[0,1],\\
1-\frac{v}{2})  \text{ if }v\in[1,2].
\end{cases}$$
We have $S_{S} (E_4)=1$. Thus, $\delta_P(S)\le 1$ for $P\in E_4$. Moreover, if $P\in E_4\cap (E_3\cup E_5)$ or if $P\in E_4\backslash (E_3\cup E_5)$   we have 
$$h(v)\le\begin{cases}
\frac{v^2}{2} \text{ if }v\in [0,1],\\
\frac{(2 -v )(v+1)}{4} \text{ if }v\in [1,2].
\end{cases}
\text{ or }
h(v)\le\begin{cases}
\frac{v^2}{8} \text{ if }v\in [0,1],\\
\frac{(2 -v )(3v - 2)}{8} \text{ if }v\in [1,2].
\end{cases}
$$
Thus $S(W_{\bullet,\bullet}^{E_4};P)\le \frac{11}{12}< 1$
or $S(W_{\bullet,\bullet}^{E_4};P)\le  \frac{1}{3}< 1$.
We get $\delta_P(S)=1$ for $P\in E_4$.
\par{\bf Step 2.} Suppose $P\in  E_3\cup E_5$. Without loss of generality we can assume that $P\in E_3$ since the
proof is similar in other cases.  There exist $(-1)$-curves and $(-2)$-curves  which form  the following dual graph:
\begin{figure}[h!]
    \centering
\includegraphics[width=8cm]{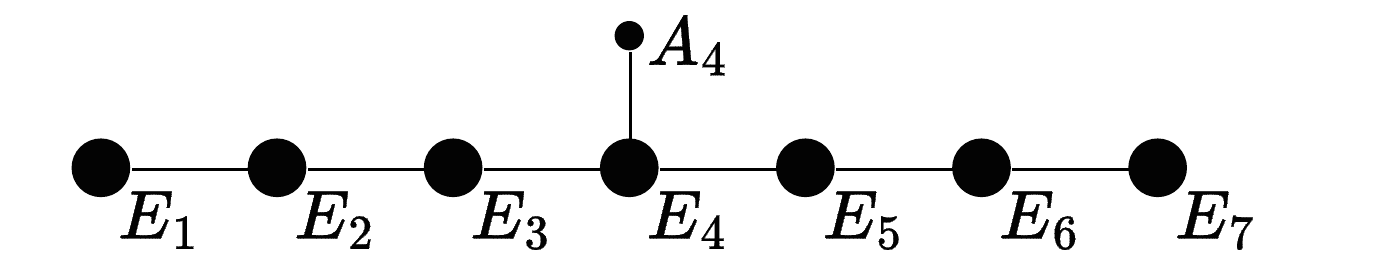}
    \caption{Dual graph: $(-K_S)^2=1$, $\DA_7$ singularity, $\delta_P(S)=\frac{12}{11}$}
\end{figure}
\par  Then $\tau(E_3)=\frac{3}{2}$ and the  Zariski Decomposition of the divisor $-K_S-vE_3$ is:
 {\allowdisplaybreaks\begin{align*}
\hspace*{-0.5cm}&&P(v)=\begin{cases}
-K_S-vE_3-\frac{v}{15} (5E_1+10E_2+12E_4+9E_5+6E_6+3E_7)\text{ if }v\in\big[0,\frac{5}{4}\big],\\
-K_S-vE_3-\frac{v}{3} (E_1+2E_2)-(v-1)(4E_4+3E_5+2E_6+E_7)-(4v-5)A_4\text{ if }v\in\big[\frac{5}{4},\frac{3}{2}\big].
\end{cases} \\
\hspace*{-0.5cm}&&N(v)=\begin{cases}
\frac{v}{15} (5E_1+10E_2+12E_4+9E_5+6E_6+3E_7)\text{ if }v\in\big[0,\frac{5}{4}\big],\\
\frac{v}{3} (E_1+2E_2)+(v-1)(4E_4+3E_5+2E_6+E_7)+(4v-5)A_4\text{ if }v\in\big[\frac{5}{4},\frac{3}{2}\big].
\end{cases}
\end{align*}} 
The Zariski Decomposition  follows from 
 $$-K_S-vE_3\sim_{\DR} \Big(\frac{3}{2}-v\Big)E_3+\frac{1}{2}\Big(E_1+2E_2+4E_4+3E_5+2E_6+E_7+2A_4\Big).$$ 
Moreover, 
$$(P(v))^2=\begin{cases}
1-\frac{8v^2}{15}  \text{ if }v\in\big[0,\frac{5}{4}\big],\\
\frac{2(3-2v)^2}{3}  \text{ if }v\in\big[\frac{5}{4},\frac{3}{2}\big].
\end{cases}
P(v)\cdot E_3=\begin{cases}
\frac{8v}{15}  \text{ if }v\in \big[0,\frac{5}{4}\big],\\
4(1-\frac{2v}{3})  \text{ if }v\in \big[\frac{5}{4},\frac{3}{2}\big].
\end{cases}$$
We have $S_{S} (E_3)=\frac{11}{12}$. Thus, $\delta_P(S)\le \frac{12}{11}$ for $P\in E_3$. Moreover, if $P\in E_3\backslash E_4$   we have 
$$h(v)\le\begin{cases}
\frac{112v^2}{225} \text{ if }v\in \big [0,\frac{5}{4}\big ],\\
\frac{8(2v - 3)(v - 3)}{9} \text{ if }v\in \big [\frac{5}{4},\frac{3}{2}\big ].
\end{cases}
$$
Thus,
$S(W_{\bullet,\bullet}^{E_3};P)\le \frac{5}{6}< \frac{11}{12}$.
We get $\delta_P(S)=\frac{12}{11}$ for $P\in (E_3\cup E_5)\backslash E_4$.
\par{\bf Step 3.} Suppose $P\in  E_2\cup E_6$. Without loss of generality we can assume that $P\in E_2$ since the proof
is similar in other cases.   There exist $(-1)$-curves and $(-2)$-curves  which form one of the following dual graphs:
\begin{figure}[h!]
    \centering
\includegraphics[width=15cm]{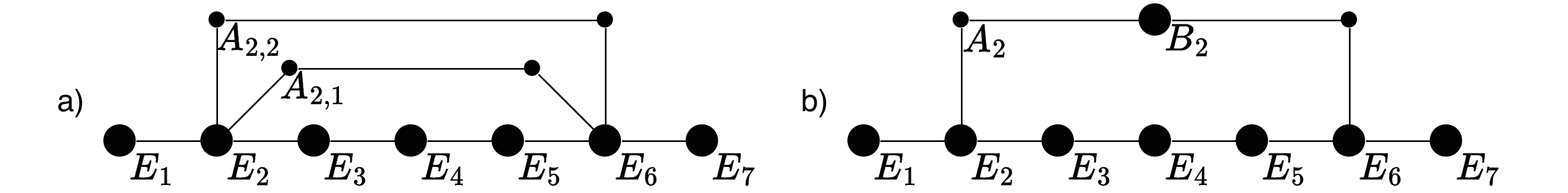}
    \caption{Dual graph: $(-K_S)^2=1$, $\DA_7$ singularity, $\delta_P(S)=\frac{6}{5}$}
\end{figure}
\par 
 Then $\tau(E_2)=\frac{3}{2}$ and the  Zariski Decomposition of the divisor $-K_S-vE_2$ is:
   { 
 {\allowdisplaybreaks\begin{align*}
\hspace*{-0.5cm}&{\text{\bf a). }} & P(v)=\begin{cases}
-K_S-vE_2-\frac{v}{6} (3E_1+5E_3+4E_4+3E_5+2E_6+E_7)\text{ if }v\in[0,1],\\
-K_S-vE_2-\frac{v}{6} (3E_1+5E_3+4E_4+3E_5+2E_6+E_7)-(v-1)(A_{2,1}+A_{2,2})\text{ if }v\in\big[1,\frac{3}{2}\big].
\end{cases} \\
\hspace*{-0.5cm}&&N(v)=\begin{cases}
\frac{v}{6} (3E_1+5E_3+4E_4+3E_5+2E_6+E_7)\text{ if }v\in[0,1],\\
\frac{v}{6} (3E_1+5E_3+4E_4+3E_5+2E_6+E_7)+(v-1)(A_{2,1}+A_{2,2})\text{ if }v\in\big[1,\frac{3}{2}\big].
\end{cases}\\
\hspace*{-0.5cm}&{\text{\bf b). }} & P(v)=\begin{cases}
-K_S-vE_2-\frac{v}{6} (3E_1+5E_3+4E_4+3E_5+2E_6+E_7)\text{ if }v\in[0,1],\\
-K_S-vE_2-\frac{v}{6} (3E_1+5E_3+4E_4+3E_5+2E_6+E_7)-(v-1)(2A_{2}+B_2)\text{ if }v\in\big[1,\frac{3}{2}\big].
\end{cases} \\
\hspace*{-0.5cm}&&N(v)=\begin{cases}
\frac{v}{6} (3E_1+5E_3+4E_4+3E_5+2E_6+E_7)\text{ if }v\in[0,1],\\
\frac{v}{6} (3E_1+5E_3+4E_4+3E_5+2E_6+E_7)+(v-1)(2A_{2}+B_2)\text{ if }v\in\big[1,\frac{3}{2}\big].
\end{cases}
\end{align*}}}
The Zariski Decomposition  follows from 
 $$-K_S-vE_2\sim_{\DR} \Big(\frac{3}{2}-v\Big)E_2+\frac{1}{4}\Big(3E_1+5E_3+4E_4+3E_5+2E_6+E_7+2A_{2,1}+2A_{2,2}\Big).$$
Moreover, 
$$(P(v))^2=\begin{cases}
1-\frac{2v^2}{3}  \text{ if }v\in[0,1],\\
\frac{(3-2v)^2}{3}  \text{ if }v\in\big [1, \frac{3}{2}\big ].
\end{cases}
P(v)\cdot E_2=\begin{cases}
\frac{2v}{3}  \text{ if }v\in[0,1],\\
2(1-\frac{2v}{3})  \text{ if }v\in\big [1, \frac{3}{2}\big ].
\end{cases}$$
We have $S_{S} (E_2)=\frac{5}{6}$. Thus, $\delta_P(S)\le \frac{6}{5}$ for $P\in E_2$. Moreover, if $P\in E_2\cap E_1$ or if $P\in E_2\backslash (E_1\cup E_3)$  we have 
$$h(v)\le\begin{cases}
\frac{5v^2}{9} \text{ if }v\in [0,1],\\
\frac{(2v - 3)(v-6)}{9} \text{ if }v\in\big [1, \frac{3}{2}\big ].
\end{cases}
\text{ or }
h(v)\le\begin{cases}
\frac{2v^2}{9} \text{ if }v\in [0,1],\\
\frac{2(3 -2v)(4v - 3)}{9} \text{ if }v\in\big [1, \frac{3}{2}\big ].
\end{cases}
$$
Thus
$S(W_{\bullet,\bullet}^{E_2};P)\le \frac{23}{36}<\frac{5}{6}$
or $S(W_{\bullet,\bullet}^{E_2};P)\le\frac{1}{3}<\frac{5}{6}$.
We get $\delta_P(S)=\frac{6}{5}$ for $P\in (E_2\cup E_6)\backslash (E_1\cup E_7)$.
\par{\bf Step 4.} Suppose $P\in E_1\cup E_7$. Without loss of generality we can assume that $P\in E_1$ since the proof is similar in other cases.  Then $\tau(E_1)=1$ and the  Zariski decomposition of the divisor $-K_S-vE_1$ is given by:
{\footnotesize
 {\allowdisplaybreaks\begin{align*}
\hspace*{-1cm}&&P(v)=\begin{cases}
-K_S-vE_1-\frac{v}{7} (6E_2+5E_3+4E_4+3E_5+2E_6+E_7) \text{ if }v\in\big [0,\frac{7}{8}\big ],\\
-K_S-vE_1-(2v-1)E_2-(3v-2)E_3-(4v-3)E_4-(5v-4)E_5-(6v-5)E_6-(7v-6)E_7-(8v-7)C \text{ if } v\in\big [\frac{7}{8}, 1\big ].
\end{cases}\\\hspace*{-1cm}&&
N(v)=
\begin{cases}
\frac{v}{7} (6E_2+5E_3+4E_4+3E_5+2E_6+E_7) \text{ if }v\in\big [0,\frac{7}{8}\big ],\\
(2v-1)E_2+(3v-2)E_3-(4v-3)E_4+(5v-4)E_5+(6v-5)E_6+(7v-6)C  \text{ if } v\in\big [\frac{7}{8}, 1\big ].
\end{cases}
\end{align*}}}
Moreover, 
$$(P(v))^2=\begin{cases}
1-\frac{8v^2}{7}  \text{ if }v\in\big [0,\frac{7}{8}\big ],\\
8(v - 1)^2   \text{ if } v\in\big [\frac{7}{8}, 1\big ].
\end{cases}
P(v)\cdot E_1=\begin{cases}
\frac{8v}{7}  \text{ if }v\in\big [0,\frac{7}{8}\big ],\\
8(1-v)   \text{ if } v\in\big [\frac{7}{8}, 1\big ].
\end{cases}$$
We have $S_{S} (E_1)=\frac{5}{8}$. Thus, $\delta_P(S)\le \frac{8}{5}$ for $P\in E_1\backslash E_2$. Moreover,  for such points we have 
$$h(v)\le\begin{cases}
\frac{32v^2}{49} \text{ if }v\in \big [0, \frac{7}{8}\big ],\\
8(1 - v )(3v - 4) \text{ if }v\in \big [\frac{7}{8},1\big ]
\end{cases}$$
Thus, $S(W_{\bullet,\bullet}^{E_1};P)\le  \frac{13}{96 }<\frac{5}{8}$. We get $\delta_P(S)=\frac{8}{5}$ for $P\in (E_1\cup E_7)\backslash (E_2\cup E_6)$.  Thus, $\delta_{\mathcal{P}} (X)=1$.
 \end{proof}
   \subsubsection{$\DA_7$ singularity (irreducible ramification)   on Du Val Del Pezzo surfaces of degree $1$}
 \begin{lemma} Let $X$ be a singular del Pezzo surface of degree $1$ with an $\DA_7$ singularity at point $\mathcal{P}$. $X$ can be realized as the double cover $X\xrightarrow{2:1}\DP(1,1,2)$, which is
ramified along a sextic curve $R\in \DP(1,1,2)$. Suppose $R$ is irreducible. Let $\mathcal{C}$ be a~curve in the~pencil $|-K_X|$ that contains~$\mathcal{P}$.  Then $\delta_{\mathcal{P}} (X)=\frac{18}{17}$.
\end{lemma}
  \begin{proof}
  Let $S$ be the minimal resolution of singularities.  Then $S$ is a weak del Pezzo surface of degree $1$. Suppose $C$ is a strict transform of $\mathcal{C}$ on $S$ and $E_1$, $E_2$, $E_3$, $E_4$, $E_5$, $E_6$ and $E_7$ are the exceptional divisors with the intersection:
 \begin{figure}[h!]
    \centering
\includegraphics[width=11cm]{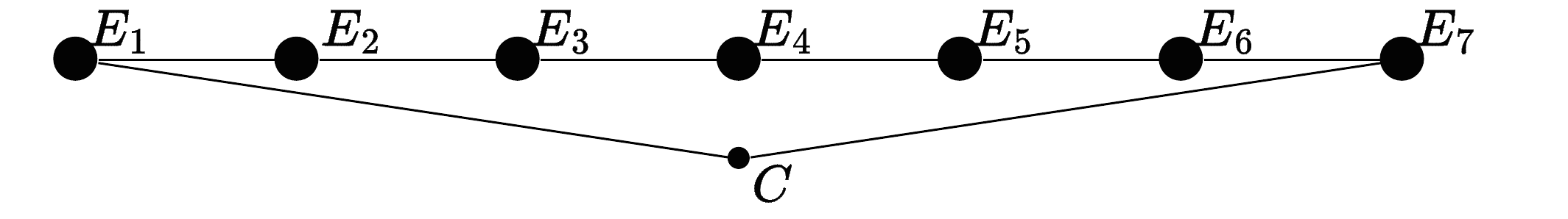}
    \caption{Dual graph: $(-K_S)^2=1$, $\DA_7$ singularity (irreducible ramification divisor)}
\end{figure}
\par We have $-K_S\sim C+E_1+E_2+E_3+E_4+E_5+E_6+E_7$. Let $P$ be a point on $S$. If the ramification divisor $R$ is reducible, then this
implies that there is no $(-1)$-curve that intersects the fundamental cycle only at  $E_4$.
\par{\bf Step 1.} Suppose $P\in  E_4$.   There exist $(-1)$-curves and $(-2)$-curves  which form the following dual graph:
\begin{figure}[h!]
    \centering
\includegraphics[width=8cm]{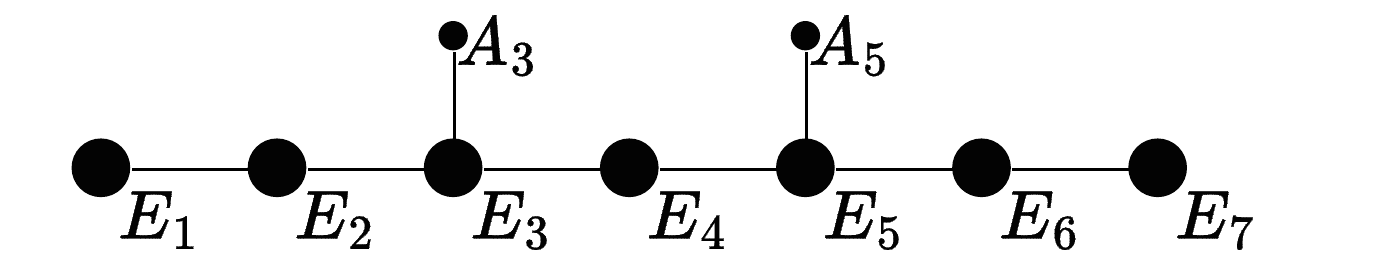}
    \caption{Dual graph: $(-K_S)^2=1$, $\DA_7$ singularity, $\delta_P(S)=\frac{18}{17}$ (1)}
\end{figure}
\par  Then $\tau(E_4)=\frac{3}{2}$ and the  Zariski Decomposition of the divisor $-K_S-vE_4$ is:
 {\allowdisplaybreaks\begin{align*}
\hspace*{-0.5cm}&&P(v)=\begin{cases}
-K_S-vE_4-\frac{v}{4} (E_1+2E_2+3E_3+3E_5+2E_6+E_7)\text{ if }v\in\big[0,\frac{4}{3}\big],\\
-K_S-vE_4-(v-1)(E_1+2E_2+3E_3+3E_5+2E_6+E_7)-(3v-4)A_3\text{ if }v\in\big[\frac{4}{3},\frac{3}{2}\big].
\end{cases} \\
\hspace*{-0.5cm}&&N(v)=\begin{cases}
\frac{v}{4} (E_1+2E_2+3E_3+3E_5+2E_6+E_7)\text{ if }v\in\big[0,\frac{4}{3}\big],\\
(v-1)(E_1+2E_2+3E_3+3E_5+2E_6+E_7)+(3v-4)A_3\text{ if }v\in\big[\frac{4}{3},\frac{3}{2}\big].
\end{cases}
\end{align*}}
The Zariski Decomposition  follows from 
 $$-K_S-vE_4\sim_{\DR} \Big(\frac{3}{2}-v\Big)E_4+\frac{1}{2}\Big(E_1+2E_2+3E_3+3E_5+2E_6+E_7+2A_3\Big).$$
Moreover, 
$$(P(v))^2=\begin{cases}
1-\frac{v^2}{2}  \text{ if }v\in\big[0,\frac{4}{3}\big],\\
(3-2v)^2 \text{ if }v\in\big[\frac{4}{3},\frac{3}{2}\big].
\end{cases}
P(v)\cdot E_4=\begin{cases}
\frac{v}{2}  \text{ if }v\in \big[0,\frac{4}{3}\big],\\
2(3-2v) \text{ if }v\in \big[\frac{4}{3},\frac{3}{2}\big].
\end{cases}$$
We have $S_{S} (E_4)=\frac{17}{18}$. Thus, $\delta_P(S)\le \frac{18}{17}$ for $P\in E_4$. Moreover, if $P\in E_4$   we have 
$$h(v)\le\begin{cases}
\frac{v^2}{2} \text{ if }v\in \big [0,\frac{4}{3}\big ],\\
2(3-2v)v \text{ if }v\in \big [\frac{4}{3},\frac{3}{2}\big ]
\end{cases}
$$
Thus $S(W_{\bullet,\bullet}^{E_2};P)\le  \frac{17}{18}$.
We get $\delta_P(S)=\frac{18}{17}$ for $P\in E_4$.\\
\par {\bf Step 2.} Suppose $P\in  E_3\cup E_5$. Without loss of generality we can assume that $P\in E_3$ since the
proof is similar in other cases. There exist $(-1)$-curves and $(-2)$-curves  which form the following dual graph:
\begin{figure}[h!]
    \centering
\includegraphics[width=7cm]{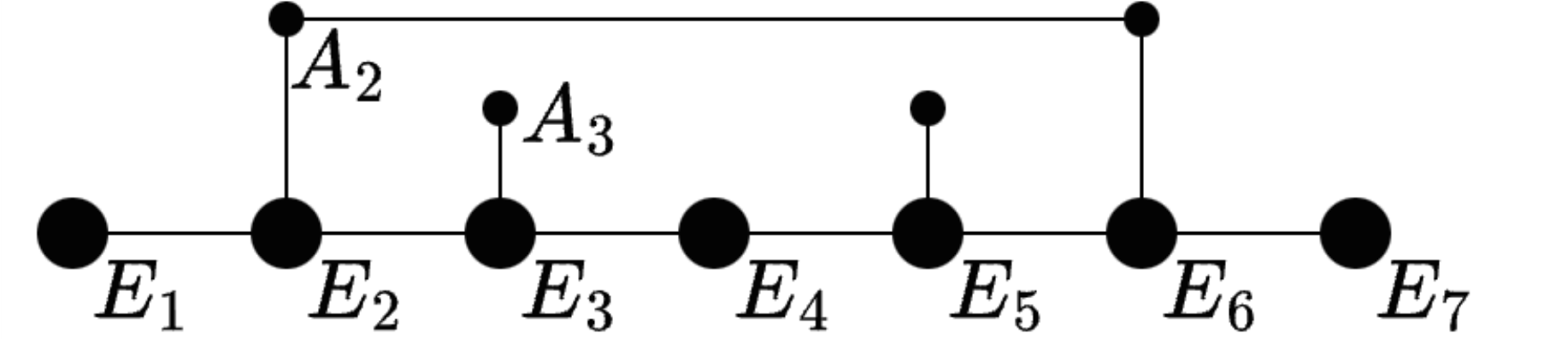}
    \caption{Dual graph: $(-K_S)^2=1$, $\DA_7$ singularity, $\delta_P(S)=\frac{18}{17}$ (2)}
\end{figure}
\par Then $\tau(E_3)=\frac{5}{3}$ and the  Zariski Decomposition of the divisor $-K_S-vE_3$ is:
{ \small  {\allowdisplaybreaks\begin{align*}
\hspace*{-0.5cm}&&P(v)=\begin{cases}
-K_S-vE_3-\frac{v}{15} (5E_1+10E_2+12E_4+9E_5+6E_6+3E_7)\text{ if }v\in[0,1],\\
-K_S-vE_3-\frac{v}{15} (5E_1+10E_2+12E_4+9E_5+6E_6+3E_7)-(v-1)A_3\text{ if }v\in\big[1,\frac{3}{2}\big],\\
-K_S-vE_3-(v-1)(E_1+2E_2+A_3)-\frac{v}{5} (4E_4+3E_5+2E_6+E_7)-(2v-3)A_2\text{ if }v\in\big[\frac{3}{2},\frac{5}{3}\big].
\end{cases}\\\hspace*{-0.5cm}&&N(v)=\begin{cases}
\frac{v}{15} (5E_1+10E_2+12E_4+9E_5+6E_6+3E_7)\text{ if }v\in[0,1],\\
\frac{v}{15} (5E_1+10E_2+12E_4+9E_5+6E_6+3E_7)+(v-1)A_3\text{ if }v\in\big[1,\frac{3}{2}\big],\\
(v-1)(E_1+2E_2+A_3)+\frac{v}{5} (4E_4+3E_5+2E_6+E_7)+(2v-3)A_2\text{ if }v\in\big[\frac{3}{2},\frac{5}{3}\big].
\end{cases}
\end{align*}}}
The Zariski Decomposition  follows from 
 $$-K_S-vE_3\sim_{\DR} \Big(\frac{5}{3}-v\Big)E_3+\frac{1}{3}\Big(2E_1+4E_2+2A_3+4E_4+3E_5+2E_6+E_7+2A_2\Big).$$
Moreover, 
$$(P(v))^2=\begin{cases}
1-\frac{8v^2}{15}  \text{ if }v\in[0,1],\\
2 - 2v + \frac{7v^2}{15}\text{ if }v\in\big[1,\frac{3}{2}\big],\\
\frac{(5-3v)^2}{5}  \text{ if }v\in\big [\frac{3}{2},\frac{5}{3}\big ].
\end{cases}
P(v)\cdot E_3=\begin{cases}
\frac{8v}{15}  \text{ if }v\in[0,1],\\
1-\frac{7v}{15}\text{ if }v\in\big[1,\frac{3}{2}\big],\\
3(1-\frac{3v}{5})  \text{ if }v\in\big [\frac{3}{2},\frac{5}{3}\big ].
\end{cases}$$
We have $S_{S} (E_3)=\frac{17}{18}$. Thus, $\delta_P(S)\le \frac{18}{17}$ for $P\in E_3$. Moreover, if $P\in E_3\cap A_3$ or if $P\in E_3\cap E_2$ we have 
{ $$h(v)\le\begin{cases}
\frac{32v^2}{225}  \text{ if }v\in[0,1],\\
\frac{(15 - 7v)(23v - 15)}{450}\text{ if }v\in\big[1,\frac{3}{2}\big],\\
\frac{3(5 - 3v)(v + 5)}{50}  \text{ if }v\in\big [ \frac{3}{2},\frac{5}{3} \big ].
\end{cases}
\text{ or }
h(v)\le\begin{cases}
\frac{112v^2}{225}  \text{ if }v\in[0,1],\\
\frac{(15 - 7v)(13v + 15)}{450}\text{ if }v\in\big[1,\frac{3}{2}\big],\\
\frac{3(5 - 3v)(11v - 5)}{50}  \text{ if }v\in\big [ \frac{3}{2},\frac{5}{3} \big ].
\end{cases}
$$}
Thus
$S(W_{\bullet,\bullet}^{E_3};P)\le \frac{14}{45}<\frac{17}{18}$
or
$S(W_{\bullet,\bullet}^{E_3};P)\le \frac{37}{45}<\frac{17}{18}$.
We get $\delta_P(S)=\frac{18}{17}$ for $P\in (E_3\cup E_5)\backslash E_4$.
\par {\bf Step 3.} Suppose $P\in  E_2\cup E_6$. Without loss of generality we can assume that
$P \in E_2$ since the proof is similar in other cases.  Then $\tau(E_2)=\frac{4}{3}$ and the   Zariski Decomposition of the divisor $-K_S-vE_2$ is:
 {\allowdisplaybreaks\begin{align*}
\hspace*{-0.5cm}&&P(v)=\begin{cases}
-K_S-vE_2-\frac{v}{6} (3E_1+5E_3+4E_4+3E_5+2E_6+E_7)\text{ if }v\in[0,1],\\
-K_S-vE_2-\frac{v}{6} (3E_1+5E_3+4E_4+3E_5+2E_6+E_7)-(v-1)A_2\text{ if }v\in\big[1,\frac{6}{5}\big],\\
-K_S-vE_2-\frac{v}{2}E_1-(v-1)(5E_3+4E_4+3E_5+2E_6+E_7+A_2)-(5v-6)A_3\text{ if }v\in\big[\frac{6}{5},\frac{4}{3}\big].
\end{cases} \\\hspace*{-0.5cm}&& N(v)=\begin{cases}
\frac{v}{6} (3E_1+5E_3+4E_4+3E_5+2E_6+E_7)\text{ if }v\in[0,1],\\
\frac{v}{6} (3E_1+5E_3+4E_4+3E_5+2E_6+E_7)+(v-1)A_2\text{ if }v\in\big[1,\frac{6}{5}\big],\\
\frac{v}{2}E_1+(v-1)(5E_3+4E_4+3E_5+2E_6+E_7+A_2)+(5v-6)A_3\text{ if }v\in\big[\frac{6}{5},\frac{4}{3}\big].
\end{cases} 
\end{align*}}
Moreover, 
$$(P(v))^2=\begin{cases}
1-\frac{2v^2}{3}  \text{ if }v\in[0,1],\\
2 - 2v + \frac{v^2}{3}\text{ if }v\in\big[1,\frac{6}{5}\big],\\
\frac{(4-3v)^2}{2}  \text{ if }v\in\big [\frac{6}{5}, \frac{4}{3}\big ].
\end{cases}
P(v)\cdot E_2=\begin{cases}
\frac{2v}{3}  \text{ if }v\in[0,1],\\
1 - \frac{v}{3}\text{ if }v\in\big[1,\frac{6}{5}\big],\\
3(2-\frac{3v}{2})  \text{ if }v\in\big [\frac{6}{5}, \frac{4}{3}\big ].
\end{cases}$$
The Zariski Decomposition  follows from 
 $$-K_S-vE_2\sim_{\DR} \Big(\frac{4}{3}-v\Big)E_2+\frac{1}{3}\Big(2E_1+5E_3+4E_4+3E_5+2E_6+E_7+A_2+2A_3\Big).$$
We have $S_{S} (E_2)=\frac{37}{45}$.
Thus, $\delta_P(S)\le \frac{45}{37}$ for $P\in E_2$. Moreover, if $P\in E_2\cap E_1$ or if $P\in E_2\backslash (E_1\cup E_3)$  we have 
$$h(v)\le\begin{cases}
\frac{5v^2}{9} \text{ if }v\in [0,1],\\
\frac{(3 - v)(2v + 3)}{18}\text{ if }v\in\big[1,\frac{6}{5}\big],\\
\frac{3(3v - 4)(7v - 12)}{8} \text{ if }v\in\big [\frac{6}{5}, \frac{4}{3}\big ].
\end{cases}
\text{ or }
h(v)\le
\begin{cases}
\frac{2v^2}{9} \text{ if }v\in [0,1],\\
\frac{(3 - v)(5v - 3)}{18}\text{ if }v\in\big[1,\frac{6}{5}\big],\\
\frac{3(3v - 4)(5v - 8)}{8} \text{ if }v\in\big [\frac{6}{5}, \frac{4}{3}\big ].
\end{cases}
$$
Thus,
$S(W_{\bullet,\bullet}^{E_2};P)\le  \frac{59}{90}<\frac{37}{45}$
or $S(W_{\bullet,\bullet}^{E_2};P)\le  \frac{13}{45}<\frac{37}{45}$.
We get $\delta_P(S)=\frac{45}{37}$ for $P\in (E_2\cup E_6)\backslash (E_3\cup E_5)$.
\par{\bf Step 4.} Suppose $P\in E_1\cup E_7$. Without loss of generality we can assume that $P\in E_1$ since the proof is similar in other cases.  Then $\tau(E_1)=1$ and the  Zariski decomposition of the divisor $-K_S-vE_1\sim C+E_1+E_2+E_3+E_4+E_5+E_6+E_7$ is given by:
{\footnotesize
 {\allowdisplaybreaks\begin{align*}
\hspace*{-1.2cm}&&P(v)=\begin{cases}
-K_S-vE_1-\frac{v}{7} (6E_2+5E_3+4E_4+3E_5+2E_6+E_7) \text{ if }v\in\big [0,\frac{7}{8}\big ],\\
-K_S-vE_1-(2v-1)E_2-(3v-2)E_3-(4v-3)E_4-(5v-4)E_5-(6v-5)E_6-(7v-6)E_7-(8v-7)C \text{ if } v\in\big [\frac{7}{8}, 1\big ].
\end{cases}\\
\hspace*{-1.2cm}&&
N(v)=
\begin{cases}
\frac{v}{7} (6E_2+5E_3+4E_4+3E_5+2E_6+E_7) \text{ if }v\in\big [0,\frac{7}{8}\big ],\\
(2v-1)E_2+(3v-2)E_3+(4v-3)E_4+(5v-4)E_5+(6v-5)E_6+(7v-6)C  \text{ if } v\in\big [\frac{7}{8}, 1\big ].
\end{cases}
\end{align*}}}
Moreover, 
$$(P(v))^2=\begin{cases}
1-\frac{8v^2}{7}  \text{ if }v\in\big [0,\frac{7}{8}\big ],\\
8(v - 1)^2   \text{ if } v\in\big [\frac{7}{8}, 1\big ].
\end{cases}
P(v)\cdot E_1=\begin{cases}
\frac{8v}{7}  \text{ if }v\in\big [0,\frac{7}{8}\big ],\\
8(1-v)   \text{ if } v\in\big [\frac{7}{8}, 1\big ].
\end{cases}$$
We have $S_{S} (E_1)=\frac{5}{8}$.
Thus, $\delta_P(S)\le \frac{8}{5}$ for $P\in E_1\backslash E_2$. Moreover,  for such points we have 
$$h(v)\le\begin{cases}
\frac{32v^2}{49} \text{ if }v\in \big [0, \frac{7}{8}\big ],\\
8(1 - v )(3v - 4) \text{ if }v\in \big [\frac{7}{8},1\big ]
\end{cases}$$
Thus,
$S(W_{\bullet,\bullet}^{E_1};P)\le \frac{13}{96}<\frac{5}{8}$.
We get $\delta_P(S)=\frac{8}{5}$ for $P\in (E_1\cup E_7)\backslash (E_2\cup E_6)$.  Thus, $\delta_{\mathcal{P}} (X)=\frac{18}{17}$.
 \end{proof}
  \subsubsection{$\DA_8$ singularity  on Du Val Del Pezzo surfaces of degree $1$}
 \begin{lemma} Let $X$ be a singular del Pezzo surface of degree $1$ with an $\DA_8$ singularity at point $\mathcal{P}$.  Let $\mathcal{C}$ be a~curve in the~pencil $|-K_X|$ that contains~$\mathcal{P}$.  Then $\delta_{\mathcal{P}} (X)=1$.
 \end{lemma}
  \begin{proof}
 Let $S$ be the minimal resolution of singularities.  Then $S$ is a weak del Pezzo surface of degree $1$. Suppose $C$ is a strict transform of $\mathcal{C}$ on $S$ and $E_1$, $E_2$, $E_3$, $E_4$, $E_5$, $E_6$, $E_7$ and $E_8$ are the exceptional divisors with the intersection:
 \begin{figure}[h!]
    \centering
\includegraphics[width=12cm]{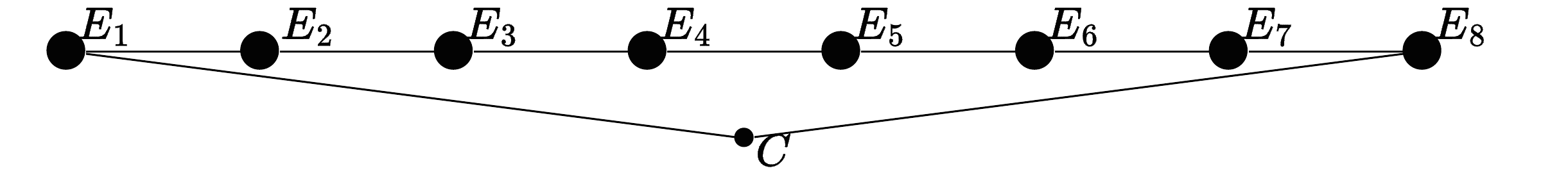}
    \caption{Dual graph: $(-K_S)^2=1$, $\DA_8$ singularity}
\end{figure}
\par 
 We have $-K_S\sim C+E_1+E_2+E_3+E_4+E_5+E_6+E_7+E_8$.  Let $P$ be a point on $S$.
\par{\bf Step 1.} Suppose $P\in  E_4\cup E_5$. Without loss of generality we can assume that $P\in E_4$ since the proof
is similar in other cases.  There exist $(-1)$-curves and $(-2)$-curves  which form the following dual graph:
\begin{figure}[h!]
    \centering
\includegraphics[width=9cm]{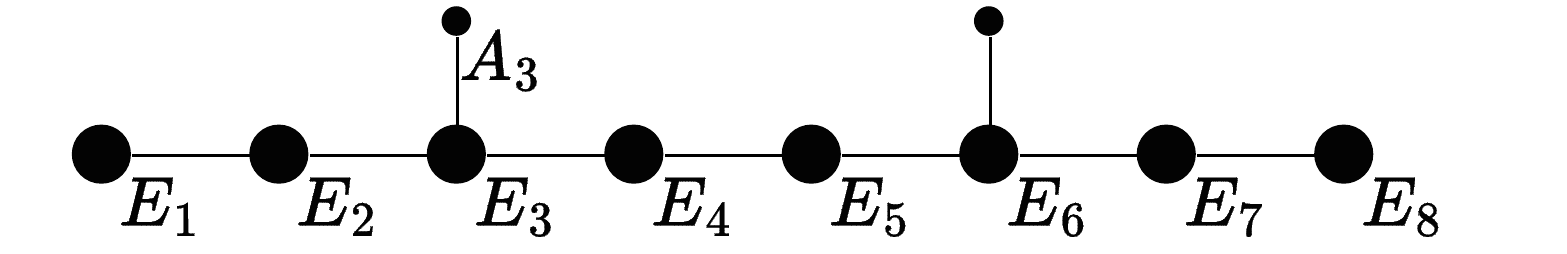}
    \caption{Dual graph: $(-K_S)^2=1$, $\DA_8$ singularity, $\delta_P(S)=1$}
\end{figure}
\par 
 Then $\tau(E_4)=\frac{5}{3}$ and the  Zariski Decomposition of the divisor $-K_S-vE_4$ is:
{ 
 {\allowdisplaybreaks\begin{align*}
\hspace*{-0.7cm}&&P(v)=\begin{cases}
-K_S-vE_4-\frac{v}{20} (5E_1+10E_2+15E_3+16E_5+12E_6+8E_7+4E_8)\text{ if }v\in\big[0,\frac{4}{3}\big],\\
-K_S-vE_4-(v-1)(E_1+2E_2+3E_3)-\frac{v}{5} (4E_5+3E_6+2E_7+E_8)-(3v-4)A_3\text{ if }v\in\big[\frac{4}{3},\frac{5}{3}\big].
\end{cases} \\
\hspace*{-0.7cm}&&N(v)=\begin{cases}
\frac{v}{20} (5E_1+10E_2+15E_3+16E_5+12E_6+8E_7+4E_8)\text{ if }v\in\big[0,\frac{4}{3}\big],\\
(v-1)(E_1+2E_2+3E_3)+\frac{v}{5} (4E_5+3E_6+2E_7+E_8)+(3v-4)A_3\text{ if }v\in\big[\frac{4}{3},\frac{5}{3}\big].
\end{cases}
\end{align*}}}
The Zariski Decomposition  follows from 
 $$-K_S-vE_4\sim_{\DR} \Big(\frac{5}{3}-v\Big)E_4+\frac{1}{3}\Big(2E_1+4E_2+6E_3+4E_5+3E_6+2E_7+E_8+3A_3\Big).$$
Moreover, 
$$(P(v))^2=\begin{cases}
1-\frac{9v^2}{20}  \text{ if }v\in\big[0,\frac{4}{3}\big],\\
\frac{(5-3v)^2}{5}\text{ if }v\in\big[\frac{4}{3},\frac{5}{3}\big].
\end{cases}
P(v)\cdot E_4=\begin{cases}
\frac{9v}{20}  \text{ if }v\in \big[0,\frac{4}{3}\big],\\
3(2-\frac{3v}{5}) \text{ if }v\in \big[\frac{4}{3},\frac{5}{3}\big].
\end{cases}$$
We have $S_{S} (E_4)=1$. Thus, $\delta_P(S)\le 1$ for $P\in E_4$. Moreover, if $P\in E_4\cap E_5$ or if $P\in E_4\backslash E_5$   we have 
$$h(v)\le\begin{cases}
\frac{369v^2}{800} \text{ if }v\in \big [0,\frac{4}{3}\big ],\\
\frac{3(3v - 5)(v - 15)}{50} \text{ if }v\in \big [\frac{4}{3},\frac{5}{3}\big ].
\end{cases}
\text{ or }
h(v)\le\begin{cases}
\frac{351v^2}{800} \text{ if }v\in \big [0,\frac{4}{3}\big ],\\
\frac{9(3v - 5)(5 - 7v)}{50} \text{ if }v\in \big [\frac{4}{3},\frac{5}{3}\big ].
\end{cases}
$$
Thus,
$S(W_{\bullet,\bullet}^{E_2};P)\le  1$.
We get $\delta_P(S)=1$ for $P\in E_4\cup E_5$.\\
{\bf Step 2.} Suppose $P\in  E_3\cup E_5$.  Without loss of generality we can assume that $P\in E_3$ since the proof
is similar in other cases.  Then $\tau(E_3)=2$ and the   Zariski Decomposition of the divisor $-K_S-vE_3$ is:
 {\allowdisplaybreaks\begin{align*}
&&P(v)=\begin{cases}
-K_S-vE_3-\frac{v}{6} (2E_1+4E_2+5E_4+4E_5+3E_6+2E_7+E_8)\text{ if }v\in[0,1],\\
-K_S-vE_3-\frac{v}{6} (2E_1+4E_2+5E_4+4E_5+3E_6+2E_7+E_8)-(v-1)A_3\text{ if }v\in[1,2].
\end{cases} \\
&&N(v)=\begin{cases}
\frac{v}{6} (2E_1+4E_2+5E_4+4E_5+3E_6+2E_7+E_8)\text{ if }v\in[0,1],\\
\frac{v}{6} (2E_1+4E_2+5E_4+4E_5+3E_6+2E_7+E_8)+(v-1)A_3\text{ if }v\in[1,2].
\end{cases}
\end{align*}}
The Zariski Decomposition  follows from 
 $$-K_S-vE_3\sim_{\DR} (2-v)E_3+\frac{1}{3}\Big(2E_1+4E_2+5E_4+4E_5+3E_6+2E_7+E_8\Big)+A_3.$$
Moreover, 
$$(P(v))^2=\begin{cases}
1-\frac{v^2}{2}  \text{ if }v\in[0,1],\\
\frac{(2-v)^2}{2}  \text{ if }v\in[1,2].
\end{cases}
P(v)\cdot E_3=\begin{cases}
\frac{v}{2}  \text{ if }v\in[0,1],\\
1-\frac{v}{2}  \text{ if }v\in[1,2].
\end{cases}$$
We have
$S_{S} (E_3)=1$. Thus, $\delta_P(S)\le 1$ for $P\in E_3$. Moreover, if $P\in E_4\cap E_2$ or if $P\in E_4\backslash (E_2\cup E_4)$   we have 
$$h(v)\le\begin{cases}
\frac{11v^2}{24} \text{ if }v\in [0,1],\\
\frac{(2 -v )(5v+6)}{24} \text{ if }v\in [1,2].
\end{cases}
\text{ or }
h(v)\le\begin{cases}
\frac{v^2}{8} \text{ if }v\in [0,1],\\
\frac{(2 -v )(3v - 2)}{8} \text{ if }v\in [1,2].
\end{cases}
$$
Thus,
$S(W_{\bullet,\bullet}^{E_3};P)\le \frac{5}{6}< 1$
or
$S(W_{\bullet,\bullet}^{E_3};P)\le  \frac{1}{3}< 1$.
We get $\delta_P(S)=1$ for $P\in (E_3\cup E_6)\backslash (E_4\cup E_5)$.
\par{\bf Step 3.} Suppose $P\in  E_2\cup E_7$.  Then $\tau(E_2)=\frac{4}{3}$ and the   Zariski Decomposition of the divisor $-K_S-vE_2$ is:
 {\small\allowdisplaybreaks\begin{align*}
&&P(v)=\begin{cases}
-K_S-vE_2-\frac{v}{2}E_1-\frac{v}{7} (6E_3+5E_4+4E_5+3E_6+2E_7+E_7)\text{ if }v\in\big[0,\frac{7}{6}\big],\\
-K_S-vE_2-\frac{v}{2}E_1-(v-1)(6E_3+5E_4+4E_5+3E_6+2E_7+E_7)-(6v-7)A_3\text{ if }v\in\big[\frac{7}{6},\frac{4}{3}\big].
\end{cases}\\&&N(v)=\begin{cases}
\frac{v}{2}E_1+\frac{v}{7} (6E_3+5E_4+4E_5+3E_6+2E_7+E_7)\text{ if }v\in\big[0,\frac{7}{6}\big],\\
\frac{v}{2}E_1+(v-1)(6E_3+5E_4+4E_5+3E_6+2E_7+E_7)+(6v-7)A_3\text{ if }v\in\big[\frac{7}{6},\frac{4}{3}\big].
\end{cases}
\end{align*}}
The Zariski Decomposition  follows from 
 $$-K_S-vE_2\sim_{\DR} \Big(\frac{4}{3}-v\Big)E_2+\frac{1}{3}\Big(2E_1+6E_3+5E_4+4E_5+3E_6+2E_7+E_7+3A_3\Big).$$
Moreover, 
$$(P(v))^2=\begin{cases}
1-\frac{9v^2}{14}  \text{ if }v\in\big [0,\frac{7}{6}\big ],\\
\frac{(4-3v)^2}{2}  \text{ if }v\in\big [\frac{7}{6}, \frac{4}{3}\big ].
\end{cases}
P(v)\cdot E_2=\begin{cases}
\frac{9v}{14}  \text{ if }v\in\big [0,\frac{7}{6}\big ],\\
3(1-\frac{3v}{2})  \text{ if }v\in\big [\frac{7}{6}, \frac{4}{3}\big ].
\end{cases}$$
We have $S_{S} (E_2)=\frac{5}{6}$.
Thus, $\delta_P(S)\le \frac{6}{5}$ for $P\in E_2$. Moreover, if $P\in E_2\backslash E_3$  we have 
$$h(v)\le\begin{cases}
\frac{207v^2}{392} \text{ if }v\in \big [0, \frac{7}{6}\big ],\\
\frac{3(3v - 4)(7v - 12)}{8} \text{ if }v\in\big[\frac{7}{6},\frac{4}{3}\big].
\end{cases}
$$
Thus
$$S(W_{\bullet,\bullet}^{E_2};P)\le 2 \Big(\int_0^{7/6} \frac{207v^2}{392} dv+\int_{7/6}^{4/3} \frac{3(3v - 4)(7v - 12)}{8} dv\Big)= \frac{1}{4}<\frac{5}{6}$$
We get $\delta_P(S)=\frac{6}{5}$ for $P\in (E_2\cup E_7)\backslash (E_3\cup E_6)$.
\par{\bf Step 4.} Suppose $P\in E_1\cup E_8$. Without loss of generality we can assume that $P\in E_1$ since the proof is similar in other cases.  Then $\tau(E_1)=1$ and the  Zariski decomposition of the divisor $-K_S-vE_1\sim C+(1-v)E_1+E_2+E_3+E_4+E_5+E_6+E_7+E_8$ is given by:
{ \footnotesize	
 {\allowdisplaybreaks\begin{align*}
\hspace*{-2cm}&&P(v)=\begin{cases}
-K_S-vE_1-\frac{v}{8} (7E_2+6E_3+5E_4+4E_5+3E_6+2E_7+E_8) \text{ if }v\in\big [0,\frac{8}{9}\big ],\\
-K_S-vE_1-(2v-1)E_2-(3v-2)E_3-(4v-3)E_4-(5v-4)E_5-(6v-5)E_6-(7v-6)E_7-(8v-7)E_8-(9v-8)C \text{ if } v\in\big [\frac{8}{9}, 1\big ].
\end{cases}\\
\hspace*{-2cm}&&
N(v)=
\begin{cases}
\frac{v}{8} (7E_2+6E_3+5E_4+4E_5+3E_6+2E_7+E_8) \text{ if }v\in\big [0,\frac{8}{9}\big ],\\
(2v-1)E_2+(3v-2)E_3+(4v-3)E_4+(5v-4)E_5+(6v-5)E_6+(7v-6)E_8+(9v-8)C\text{ if }v\in\big [\frac{8}{9},1\big ].
\end{cases}
\end{align*}}}
Moreover, 
$$(P(v))^2=\begin{cases}
1-\frac{9v^2}{8}  \text{ if }v\in\big [0,\frac{8}{9}\big ],\\
9(v - 1)^2  \text{ if }v\in\big [\frac{8}{9},1\big ].
\end{cases}
P(v)\cdot E_1=\begin{cases}
\frac{9v}{8}  \text{ if }v\in\big [0,\frac{8}{9}\big ],\\
9(1-v)  \text{ if }v\in\big [\frac{8}{9},1\big ].
\end{cases}$$
We have $S_{S} (E_1)=\frac{17}{27}$.
Thus, $\delta_P(S)\le \frac{27}{17}$ for $P\in E_1\backslash E_2$. Moreover,  for such points we have 
$$h(v)\le\begin{cases}
\frac{81v^2}{128} \text{ if }v\in \big [0, \frac{8}{9}\big ],\\
\frac{9(1 - v )(9v - 7)}{2} \text{ if }v\in \big [\frac{8}{9},1\big ].
\end{cases}$$
Thus,
$S(W_{\bullet,\bullet}^{E_1};P)\le \frac{10}{27}<\frac{17}{27}$.
We get $\delta_P(S)=\frac{27}{17}$ for $P\in (E_1\cup E_8)\backslash (E_2\cup E_7)$. \\ Thus, $\delta_{\mathcal{P}} (X)=1$.
\end{proof}
\subsubsection{$\mathbb{D}_4$ singularity on Du Val Del Pezzo surfaces of degree $1$}
\label{dP1-D4}
 \begin{lemma} Let $X$ be a singular del Pezzo surface of degree $1$ with an $\mathbb{D}_4$ singularity at point $\mathcal{P}$.  Let $\mathcal{C}$ be a~curve in the~pencil $|-K_X|$ that contains~$\mathcal{P}$. Then $\delta_{\mathcal{P}} (X)=1$.
 \end{lemma}
 \begin{proof}
Let $S$ be the minimal resolution of singularities.  Then $S$ is a weak del Pezzo surface of degree $1$. Suppose $C$ is a strict transform of $\mathcal{C}$ on $S$ and $E$, $E_1$, $E_2$ and $E_3$ are the exceptional divisors with the intersection:
\begin{figure}[h!]
    \centering
\includegraphics[width=3.2cm]{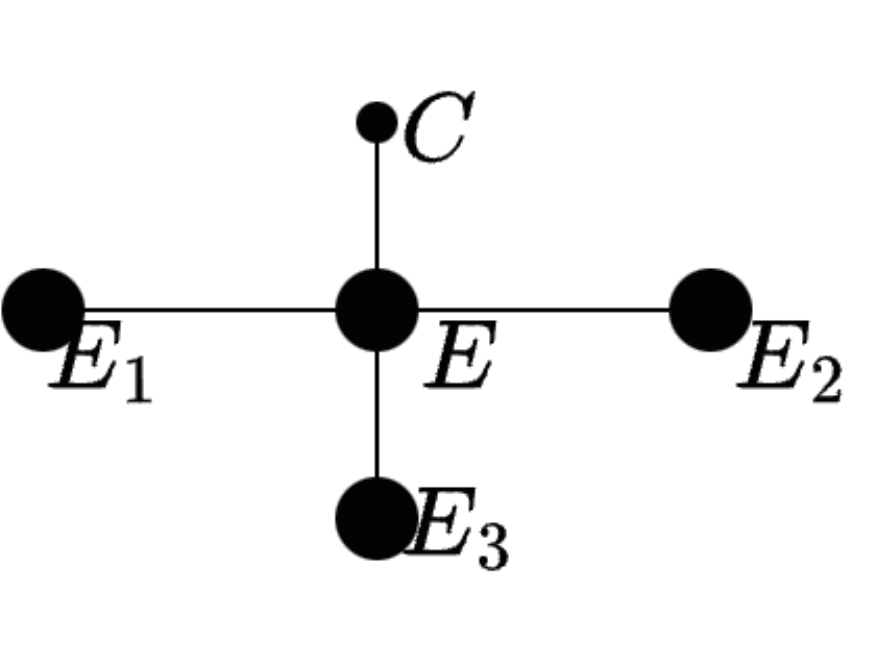}
    \caption{Dual graph: $(-K_S)^2=1$, $\mathbb{D}_4$ singularity}
\end{figure}
\par 
 We have $-K_S\sim C+2E+E_1+E_2+E_3$. Let $P$ be a point on $S$.
\par {\bf Step 1.} Suppose $P\in  E$.  Then $\tau(E)=2$ and the  Zariski decomposition of the divisor $-K_S-vE\sim (2-v)E+E_1+E_2+E_3+C$ is:
 {\allowdisplaybreaks\begin{align*}
&&P(v)=\begin{cases}
-K_S-vE-\frac{v}{2} (E_1+E_2+E_3)\text{ if }v\in[0,1],\\
-K_S-vE-\frac{v}{2} (E_1+E_2+E_3)-(v-1)C\text{ if }v\in[1,2].
\end{cases}\\&&
N(v)=\begin{cases}
\frac{v}{2} (E_1+E_2+E_3)\text{ if }v\in[0,1],\\
\frac{v}{2} (E_1+E_2+E_3)+(v-1)C\text{ if }v\in[1,2].
\end{cases}
\end{align*}} 
Moreover, 
$$(P(v))^2=\begin{cases}
1-\frac{v^2}{2}  \text{ if }v\in[0,1],\\
\frac{(2-v)^2}{2}  \text{ if }v\in[1,2].
\end{cases}
P(v)\cdot E=\begin{cases}
\frac{v}{2}  \text{ if }v\in[0,1],\\
1-\frac{v}{2}  \text{ if }v\in[1,2].
\end{cases}$$
We have $S_{S} (E)=1$
Thus, $\delta_P(S)\le 1$ for $P\in E$. Moreover, if $P\in E\cap (E_1\cup E_2\cup E_3)$ or if $P\in E\backslash (E_1\cup E_2\cup E_3)$   we have 
$$h(v)\le\begin{cases}
\frac{3v^2}{8} \text{ if }v\in [0,1],\\
\frac{(2 -v )(2+v)}{24} \text{ if }v\in [1,2].
\end{cases}
\text{ or }
h(v)\le\begin{cases}
\frac{v^2}{8} \text{ if }v\in [0,1],\\
\frac{(2 -v )(3v - 2)}{8} \text{ if }v\in [1,2].
\end{cases}
$$
Thus,
$S(W_{\bullet,\bullet}^{E};P)\le \frac{2}{3}< 1$
or
$S(W_{\bullet,\bullet}^{E};P)\le  \frac{1}{3}< 1$.
We get $\delta_P(S)=1$ for $P\in E$.
\par{\bf Step 2.} Suppose $P\in E_1\cup E_2\cup E_3$. Without loss of generality we can assume that $P\in E_1$ since the
proof is similar in other cases.  Then $\tau(E_1)=1$ and the  Zariski decomposition of the divisor $-K_S-vE_1\sim  C+2E+(1-v)E_1+E_2+E_3$ is given by:
$$P(v)=
-K_S-vE_1-\frac{v}{2} (2E+E_1+E_2)
\text{ and }
N(v)=
\frac{v}{2} (2E+E_1+E_2)\text{ if }v\in[0,1].$$
Moreover, 
$$(P(v))^2=(1-v)(1+v)\text{ and }
P(v)\cdot E_1=v\text{ if }v\in[0,1].$$
We have $S_{S} (E_1)=\frac{2}{3}$.
Thus, $\delta_P(S)\le \frac{3}{2}$ for $P\in E_1$. Moreover,  for $E_1\backslash E$ such points we have 
$h(v)\le \frac{v^2}{2}\text{ if }v\in[0,1].$
Thus,
$S(W_{\bullet,\bullet}^{E_1};P)\le \frac{1}{3}<\frac{2}{3}$.
We get $\delta_P(S)=\frac{3}{2}$ for $P\in (E_1\cup E_2\cup E_3)\backslash E$.  Thus, $\delta_{\mathcal{P}} (X)=1$.
\end{proof}
\subsubsection{$\mathbb{D}_5$ singularity on Du Val Del Pezzo surfaces of degree $1$}
\label{dP1-D5}
 \begin{lemma} Let $X$ be a singular del Pezzo surface of degree $1$ with an $\mathbb{D}_5$ singularity at point $\mathcal{P}$.  Let $\mathcal{C}$ be a~curve in the~pencil $|-K_X|$ that contains~$\mathcal{P}$. Then $\delta_{\mathcal{P}} (X)=\frac{6}{7}$.
 \end{lemma}
 \begin{proof}
Let $S$ be the minimal resolution of singularities.  Then $S$ is a weak del Pezzo surface of degree $1$. Suppose $C$ is a strict transform of $\mathcal{C}$ on $S$ and $E$, $E_1$, $E_2$, $E_3$ and $E_4$ are the exceptional divisors with the intersection:
\begin{figure}[h!]
    \centering
\includegraphics[width=4cm]{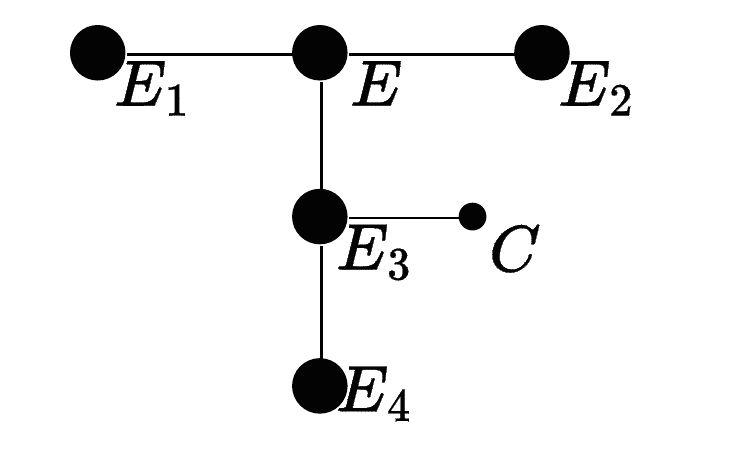}
    \caption{Dual graph: $(-K_S)^2=1$, $\mathbb{D}_5$ singularity}
\end{figure}
\par
We have $-K_S\sim C+E_1+E_2+2E+2E_3+E_4$. Let $P$ be a point on $S$.
\par{\bf Step 1.} Suppose $P\in  E$.  Then $\tau(E)=2$ and the  Zariski decomposition of the divisor $-K_S-vE\sim (2-v)E+E_1+E_2+2E_3+E_4+C$ is:
 {\allowdisplaybreaks\begin{align*}
&&P(v)=\begin{cases}
-K_S-vE-\frac{v}{6} (3E_1+3E_2+4E_3+2E_4)\text{ if }v\in\big[0,\frac{3}{2}\big],\\
-K_S-vE-\frac{v}{2} (E_1+E_2)-(v-1)(2E_3+E_4)-(2v-3)C\text{ if }v\in\big[\frac{3}{2},2\big].
\end{cases} \\&&N(v)=\begin{cases}
\frac{v}{6} (3E_1+3E_2+4E_3+2E_4)\text{ if }v\in\big[0,\frac{3}{2}\big],\\
\frac{v}{2} (E_1+E_2)+(v-1)(2E_3+E_4)+(2v-3)C\text{ if }v\in\big[\frac{3}{2},2\big].
\end{cases}
\end{align*}}
Moreover, 
$$(P(v))^2=\begin{cases}
1-\frac{v^2}{3}  \text{ if }v\in\big [0,\frac{3}{2}\big ],\\
(2-v)^2  \text{ if }v\in\big[\frac{3}{2},2\big].
\end{cases}
P(v)\cdot E=\begin{cases}
\frac{v}{3}  \text{ if }v\in\big [0,\frac{3}{2}\big ],\\
2-v  \text{ if }v\in\big[\frac{3}{2},2\big].
\end{cases}$$
We have
$S_{S} (E)=\frac{7}{6}$.
Thus, $\delta_P(S)\le \frac{6}{7}$ for $P\in E$. Moreover, if $P\in E\cap (E_1\cup E_2)$ or if $P\in E\backslash (E_1\cup E_2)$   we have 
$$h(v)\le\begin{cases}
\frac{2v^2}{9} \text{ if }v\in \big [0, \frac{3}{2}\big ],\\
2-v \text{ if }v\in \big [\frac{3}{2}, 2\big ].
\end{cases}
\text{ or }
h(v)\le\begin{cases}
\frac{5v^2}{18} \text{ if }v\in \big [0, \frac{3}{2}\big ],\\
\frac{(2 -v )(3v - 2)}{2} \text{ if }v\in \big [\frac{3}{2}, 2\big ].
\end{cases}
$$
Thus,
$S(W_{\bullet,\bullet}^{E};P)\le\frac{3}{4}< \frac{7}{6}$
or
$S(W_{\bullet,\bullet}^{E};P)\le   1<\frac{7}{6}$.
We get $\delta_P(S)=\frac{6}{7}$ for $P\in E$.
\par{\bf Step 2.} Suppose $P\in  E_1\cup E_2$. Without loss of generality we can assume that $P\in E_1$ since the proof
is similar in other cases.  There exist $(-1)$-curves and $(-2)$-curves  which form one of the following dual graphs:
\begin{figure}[h!]
    \centering
\hspace*{-1cm}\includegraphics[width=18cm]{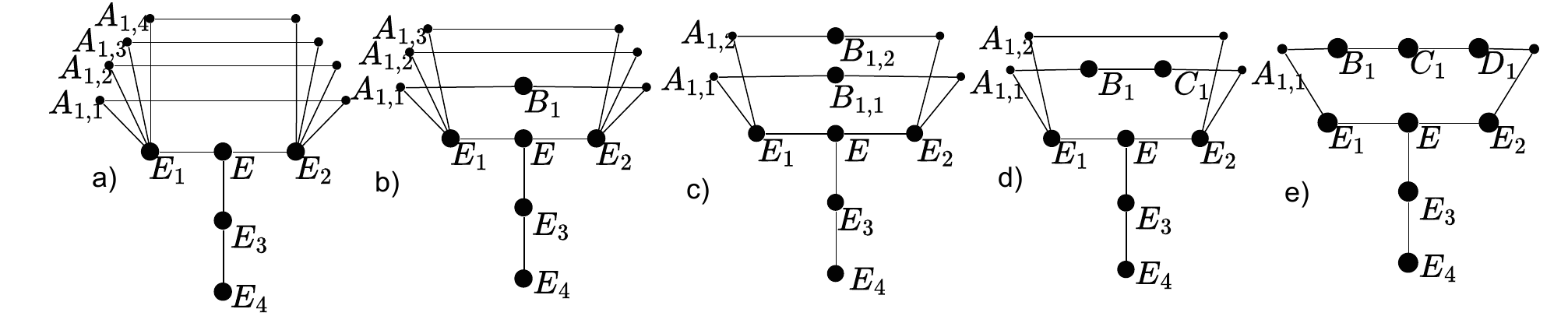}
    \caption{Dual graph: $(-K_S)^2=1$, $\mathbb{D}_5$ singularity, $\delta_P(S)=\frac{4}{3}$}
\end{figure}
\par    Then $\tau(E_1)=\frac{5}{4}$ and the  Zariski Decomposition of the divisor $-K_S-vE_1$ is:
   { 
 {\allowdisplaybreaks\begin{align*}
\hspace*{-0.7cm}&{\text{\bf a). }} & P(v)=\begin{cases}
-K_S-vE_1-\frac{v}{5} (6E+3E_2+4E_3+2E_4)\text{ if }v\in[0,1],\\
-K_S-vE_1-\frac{v}{5} (6E+3E_2+4E_3+2E_4)-(v-1)(A_{1,1}+A_{1,2}+A_{1,3}+A_{1,4})\text{ if }v\in\big[1,\frac{5}{4}\big].
\end{cases} \\
\hspace*{-0.7cm}&&N(v)=\begin{cases}
\frac{v}{5} (6E+3E_2+4E_3+2E_4)\text{ if }v\in[0,1],\\
\frac{v}{5} (6E+3E_2+4E_3+2E_4)+(v-1)(A_{1,1}+A_{1,2}+A_{1,3}+A_{1,4})\text{ if }v\in\big[1,\frac{5}{4}\big].
\end{cases}\\
\hspace*{-0.7cm}&{\text{\bf b). }} & P(v)=\begin{cases}
-K_S-vE_1-\frac{v}{5} (6E+3E_2+4E_3+2E_4)\text{ if }v\in[0,1],\\
-K_S-vE_1-\frac{v}{5} (6E+3E_2+4E_3+2E_4)-(v-1)(2A_{1,1}+B_1+A_{1,2}+A_{1,3})\text{ if }v\in\big[1,\frac{5}{4}\big].
\end{cases} \\
\hspace*{-0.7cm}&&N(v)=\begin{cases}
\frac{v}{5} (6E+3E_2+4E_3+2E_4)\text{ if }v\in[0,1],\\
\frac{v}{5} (6E+3E_2+4E_3+2E_4)+(v-1)(2A_{1,1}+B_1+A_{1,2}+A_{1,3})\text{ if }v\in\big[1,\frac{5}{4}\big].
\end{cases}\\
\hspace*{-0.7cm}&{\text{\bf c). }} & P(v)=\begin{cases}
-K_S-vE_1-\frac{v}{5} (6E+3E_2+4E_3+2E_4)\text{ if }v\in[0,1],\\
-K_S-vE_1-\frac{v}{5} (6E+3E_2+4E_3+2E_4)-(v-1)(2A_{1,1}+B_{1,1}+A_{1,2}+B_{1,2})\text{ if }v\in\big[1,\frac{5}{4}\big].
\end{cases} \\
\hspace*{-0.7cm}&&N(v)=\begin{cases}
\frac{v}{5} (6E+3E_2+4E_3+2E_4)\text{ if }v\in[0,1],\\
\frac{v}{5} (6E+3E_2+4E_3+2E_4)+(v-1)(2A_{1,1}+B_{1,1}+A_{1,2}+B_{1,2})\text{ if }v\in\big[1,\frac{5}{4}\big].
\end{cases}\\
\hspace*{-0.7cm}&{\text{\bf d). }} & P(v)=\begin{cases}
-K_S-vE_1-\frac{v}{5} (6E+3E_2+4E_3+2E_4)\text{ if }v\in[0,1],\\
-K_S-vE_1-\frac{v}{5} (6E+3E_2+4E_3+2E_4)-(v-1)(3A_{1,1}+2B_{1}+C_1+A_{1,2})\text{ if }v\in\big[1,\frac{5}{4}\big].
\end{cases} \\
\hspace*{-0.7cm}&&N(v)=\begin{cases}
\frac{v}{5} (6E+3E_2+4E_3+2E_4)\text{ if }v\in[0,1],\\
\frac{v}{5} (6E+3E_2+4E_3+2E_4)+(v-1)(3A_{1,1}+2B_{1}+C_1+A_{1,2})\text{ if }v\in\big[1,\frac{5}{4}\big].
\end{cases}\\
\hspace*{-0.7cm}&{\text{\bf e). }} & P(v)=\begin{cases}
-K_S-vE_1-\frac{v}{5} (6E+3E_2+4E_3+2E_4)\text{ if }v\in[0,1],\\
-K_S-vE_1-\frac{v}{5} (6E+3E_2+4E_3+2E_4)-(v-1)(4A_{1,1}+3B_{1}+2C_1+D_{1})\text{ if }v\in\big[1,\frac{5}{4}\big].
\end{cases} \\
\hspace*{-0.7cm}&&N(v)=\begin{cases}
\frac{v}{5} (6E+3E_2+4E_3+2E_4)\text{ if }v\in[0,1],\\
\frac{v}{5} (6E+3E_2+4E_3+2E_4)+(v-1)(4A_{1,1}+3B_{1}+2C_1+D_{1})\text{ if }v\in\big[1,\frac{5}{4}\big].
\end{cases}
\end{align*}}}
The Zariski Decomposition in part a). follows from 
 $$-K_S-vE_1\sim_{\DR} \Big(\frac{5}{4}-v\Big)E_1+\frac{1}{4}\Big(6E+3E_2+4E_3+2E_4+A_{1,1}+A_{1,2}+A_{1,3}+A_{1,4}\Big).$$
 A similar statement holds in other parts. Moreover, 
$$(P(v))^2=\begin{cases}
1-\frac{4v^2}{5}  \text{ if }v\in[0,1],\\
\frac{(5-4v)^2}{5}  \text{ if }v\in\big [1, \frac{5}{4}\big ].
\end{cases}
P(v)\cdot E_1=\begin{cases}
\frac{4v}{5}  \text{ if }v\in[0,1],\\
4(1-\frac{4v}{5})  \text{ if }v\in\big [1, \frac{5}{4}\big ].
\end{cases}$$
We have
$S_{S} (E_1)=\frac{3}{4}$. Thus, $\delta_P(S)\le \frac{4}{3}$ for $P\in E_1$. Moreover, if $P\in E_1\backslash E$  we have 
$$h(v)\le\begin{cases}
\frac{8v^2}{25} \text{ if }v\in [0,1],\\
\frac{4(5 - 4v)(7v - 5)}{25} \text{ if }v\in \big [1, \frac{5}{4}\big ].
\end{cases}
$$
Thus,
$S(W_{\bullet,\bullet}^{E_1};P)\le  \frac{19}{60}<\frac{3}{4}$.
We get $\delta_P(S)=\frac{4}{3}$ for $P\in (E_1\cup E_2)\backslash E$.
\par{\bf Step 3.} Suppose $P\in  E_3$.  Then $\tau(E_3)=2$ and the  Zariski decomposition of the divisor $-K_S-vE_3\sim C+E_1+E_2+2E+(2-v)E_3+E_4$ is:
 {\allowdisplaybreaks\begin{align*}
&& P(v)=\begin{cases}
-K_S-vE_3-\frac{v}{2} (E_4+2E+E_1+E_2)\text{ if }v\in[0,1],\\
-K_S-vE_3-\frac{v}{2} (E_4+2E+E_1+E_2)-(v-1)C\text{ if }v\in[1,2].
\end{cases} \\
&&\text{}N(v)=\begin{cases}
\frac{v}{2} (E_4+2E+E_1+E_2)\text{ if }v\in[0,1],\\
\frac{v}{2} (E_4+2E+E_1+E_2)+(v-1)C\text{ if }v\in[1,2].
\end{cases}
\end{align*}}
Moreover, 
$$(P(v))^2=\begin{cases}
1-\frac{v^2}{2}  \text{ if }v\in[0,1],\\
\frac{(2-v)^2}{2}  \text{ if }v\in[1,2].
\end{cases}
P(v)\cdot E_3=\begin{cases}
\frac{v}{2}  \text{ if }v\in[0,1],\\
1-\frac{v}{2}  \text{ if }v\in[1,2].
\end{cases}$$
Now we apply the computation from Section \ref{dP1-D4} (Step 1.) and get that $\delta_P(S)=1$ for $P\in E_3\backslash E$.
\par{\bf Step 4.} Suppose $P\in E_4$.  Then $\tau(E_4)=1$ and the  Zariski decomposition of the divisor $-K_S-vE_4\sim C+E_1+E_2+2E+2E_3+(1-v)E_4$ is given by:
$$P(v)=
-K_S-vE_4-\frac{v}{2} (2E_3+2E+E_1+E_2)
\text{ and }
N(v)=
\frac{v}{2} (2E_3+2E+E_1+E_2)\text{ if }v\in[0,1].$$
Moreover, 
$$(P(v))^2=(1-v)(1+v)\text{ and }
P(v)\cdot E_4=v\text{ if }v\in[0,1].$$
Now we apply the computation from Section \ref{dP1-D4} (Step 2.) and get that $\delta_P(S)=\frac{3}{2}$ for $P\in E_4\backslash E_3$. 
\\Thus, $\delta_{\mathcal{P}} (X)=\frac{6}{7}$.
\end{proof}
\subsubsection{$\mathbb{D}_6$ singularity on Du Val Del Pezzo surfaces of degree $1$}\label{dP1-D6}
 \begin{lemma}  Let $X$ be a singular del Pezzo surface of degree $1$ with an $\mathbb{D}_6$ singularity at point $\mathcal{P}$.  Let $\mathcal{C}$ be a~curve in the~pencil $|-K_X|$ that contains~$\mathcal{P}$.  Then $\delta_{\mathcal{P}} (X)=\frac{3}{4}$.
 \end{lemma}
 \begin{proof}
Let $S$ be the minimal resolution of singularities.  Then $S$ is a weak del Pezzo surface of degree $1$. Suppose $C$ is a strict transform of $\mathcal{C}$ on $S$ and $E$, $E_1$, $E_2$, $E_3$, $E_4$ and $E_5$ are the exceptional divisors with the intersection:
\begin{figure}[h!]
    \centering
\includegraphics[width=4cm]{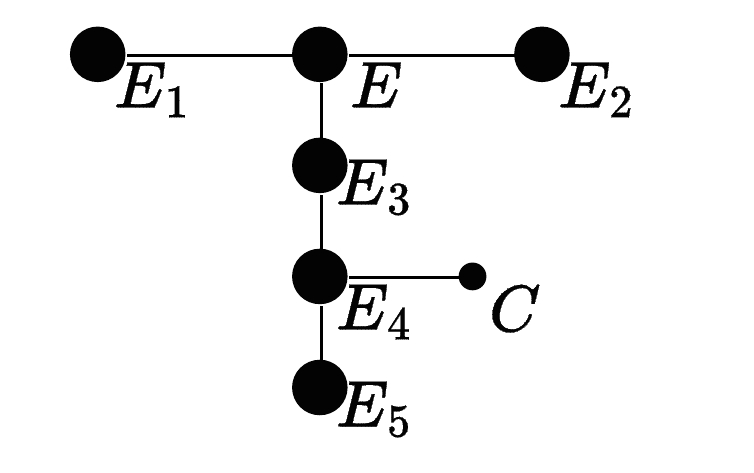}
    \caption{Dual graph: $(-K_S)^2=1$, $\mathbb{D}_6$ singularity}
\end{figure}
\par  
We have $-K_S\sim C+E_1+E_2+2E+2E_3+2E_4+E_5$. Let $P$ be a point on $S$.
\par{\bf Step 1.} Suppose $P\in E$.  Then $\tau(E)=2$ and the  Zariski decomposition of the divisor $-K_S-vE\sim C+E_1+E_2+(2-v)E+2E_3+2E_4+E_5$ is given by:
\begin{align*}
&&P(v)=
-K_S-vE-\frac{v}{4} (2E_1+2E_2+3E_3+2E_4+E_5)\text{ if }v\in[0,2].
\\&&
N(v)=
\frac{v}{4} (2E_1+2E_2+3E_3+2E_4+E_5)\text{ if }v\in[0,2].
\end{align*}
Moreover, 
$$(P(v))^2=\frac{(2-v)(2+v)}{4}
P(v)\cdot E=\frac{v}{4}\text{ and }\text{ if }v\in[0,2].$$
We have
$S_{S} (E)=\frac{4}{3}$. Thus, $\delta_P(S)\le \frac{3}{4}$ for $P\in E$. Moreover,  for  such points we have 
$h(v)\le \frac{7v^2}{32}\text{ if }v\in[0,2].$
Thus,
$S(W_{\bullet,\bullet}^{E};P)\le\frac{7}{6}<\frac{4}{3}$.
We get $\delta_P(S)=\frac{3}{4}$ for $P\in E$. 
\par{\bf Step 2.} Suppose $P\in  E_1\cup E_2$. Without loss of generality we can assume that $P\in E_1$ since the proof
is similar in other cases.  There exist $(-1)$-curves and $(-2)$-curves  which form one of the following dual graphs:
\begin{figure}[h!]
    \centering
\hspace*{-1cm}\includegraphics[width=18cm]{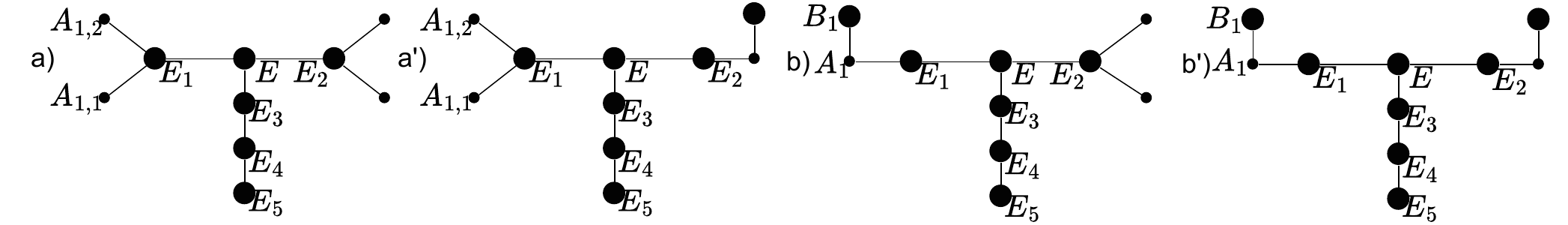}
    \caption{Dual graph: $(-K_S)^2=1$, $\mathbb{D}_6$ singularity, $\delta_P(S)=\frac{6}{5}$}
\end{figure}
\par Then $\tau(E_1)=\frac{3}{2}$ and the  Zariski Decomposition of the divisor $-K_S-vE_1$ is:
   { 
 {\allowdisplaybreaks\begin{align*}
\hspace*{-0.3cm}&{\text{\bf a, a'). }} & P(v)=\begin{cases}
-K_S-vE_1-\frac{v}{3} (4E+2E_2+3E_3+2E_4+E_5)\text{ if }v\in[0,1],\\
-K_S-vE_1-\frac{v}{3} (4E+2E_2+3E_3+2E_4+E_5)-(v-1)(A_{1,1}+A_{1,2})\text{ if }v\in\big[1,\frac{3}{2}\big].
\end{cases} \\
\hspace*{-0.3cm}&&N(v)=\begin{cases}
\frac{v}{3} (4E+2E_2+3E_3+2E_4+E_5)\text{ if }v\in[0,1],\\
\frac{v}{3} (4E+2E_2+3E_3+2E_4+E_5)+(v-1)(A_{1,1}+A_{1,2})\text{ if }v\in\big[1,\frac{3}{2}\big].
\end{cases}\\
\hspace*{-0.3cm}&{\text{\bf b, b'). }} & P(v)=\begin{cases}
-K_S-vE_1-\frac{v}{3} (4E+2E_2+3E_3+2E_4+E_5)\text{ if }v\in[0,1],\\
-K_S-vE_1-\frac{v}{3} (4E+2E_2+3E_3+2E_4+E_5)-(v-1)(2A_{1,1}+B_1)\text{ if }v\in\big[1,\frac{3}{2}\big].
\end{cases} \\
\hspace*{-0.3cm}&&N(v)=\begin{cases}
\frac{v}{3} (4E+2E_2+3E_3+2E_4+E_5)\text{ if }v\in[0,1],\\
\frac{v}{3} (4E+2E_2+3E_3+2E_4+E_5)+(v-1)(2A_{1,1}+B_1)\text{ if }v\in\big[1,\frac{3}{2}\big].
\end{cases}
\end{align*}}}
The Zariski Decomposition in part a). follows from 
 $$-K_S-vE_1\sim_{\DR} \Big(\frac{3}{2}-v\Big)E_1+\frac{1}{2}\Big(4E+2E_2+3E_3+2E_4+E_5+A_{1,1}+A_{1,2}\Big).$$
 A similar statement holds in other parts.
Moreover, 
$$(P(v))^2=\begin{cases}
1-\frac{2v^2}{3}  \text{ if }v\in[0,1],\\
\frac{(3-2v)^2}{3}  \text{ if }v\in\big [1, \frac{3}{2}\big ].
\end{cases}
P(v)\cdot E_1=\begin{cases}
\frac{2v}{3}  \text{ if }v\in[0,1],\\
2(1-\frac{2v}{3})  \text{ if }v\in\big [1, \frac{3}{2}\big ].
\end{cases}$$
We have $S_{S} (E_1)=\frac{5}{6}$.
Thus, $\delta_P(S)\le \frac{6}{5}$ for $P\in E_1$. Moreover, if $P\in E_1\backslash E$  we have 
$$h(v)\le\begin{cases}
\frac{2v^2}{9} \text{ if }v\in [0,1],\\
\frac{2(2v - 3)(4v - 3)}{9} \text{ if }v\in\big [1, \frac{3}{2}\big ].
\end{cases}
$$
Thus,
$S(W_{\bullet,\bullet}^{E_1};P)\le\frac{1}{3}<\frac{5}{6}$.
We get $\delta_P(S)=\frac{6}{5}$ for $P\in (E_1\cup E_2)\backslash E$.
\par{\bf Step 3.} Suppose $P\in  E_3$. Then $\tau(E_3)=2$ and the  Zariski decomposition of the divisor  $-K_S-vE_3\sim C+E_1+E_2+2E+(2-v)E_3+2E_4+E_5$ is:
 {\allowdisplaybreaks\begin{align*}
&&P(v)=\begin{cases}
-K_S-vE_3-\frac{v}{6} (3E_1+3E_2+6E+4E_4+2E_5)\text{ if }v\in\big[0,\frac{3}{2}\big],\\
-K_S-vE_3-\frac{v}{2} (2E+E_1+E_2)-(v-1)(2E_4+E_5)-(2v-3)C\text{ if }v\in\big[\frac{3}{2},2\big].
\end{cases}\\ 
& & N(v)=\begin{cases}
\frac{v}{6} (3E_1+3E_2+6E+4E_4+2E_5)\text{ if }v\in\big[0,\frac{3}{2}\big],\\
\frac{v}{2} (2E+E_1+E_2)+(v-1)(2E_4+E_5)+(2v-3)C\text{ if }v\in\big[\frac{3}{2},2\big].
\end{cases}  
\end{align*}}
Moreover, 
$$(P(v))^2=\begin{cases}
1-\frac{v^2}{3}  \text{ if }v\in\big [0,\frac{3}{2}\big ],\\
(2-v)^2  \text{ if }v\in\big[\frac{3}{2},2\big].
\end{cases}
P(v)\cdot E_3=\begin{cases}
\frac{v}{3}  \text{ if }v\in\big [0,\frac{3}{2}\big ],\\
2-v  \text{ if }v\in\big[\frac{3}{2},2\big].
\end{cases}$$
Now we apply the computation from Section \ref{dP1-D5} (Step 1.) and get that $\delta_P(S)=\frac{6}{7}$ for $P\in E_3\backslash E$.
\par{\bf Step 4.} Suppose $P\in  E_4$. Then $\tau(E_4)=2$ and the  Zariski decomposition of the divisor  $-K_S-vE_4\sim C+E_1+E_2+2E+2E_3+(2-v)E_4+E_5$ is:
 {\allowdisplaybreaks\begin{align*}
&& P(v)=\begin{cases}
-K_S-vE_4-\frac{v}{2} (2E_3+2E+E_1+E_2+E_5)\text{ if }v\in[0,1],\\
-K_S-vE_4-\frac{v}{2} (2E_3+2E+E_1+E_2+E_5)-(v-1)C\text{ if }v\in[1,2].
\end{cases} \\
&&\text{}N(v)=\begin{cases}
\frac{v}{2} (2E_3+2E+E_1+E_2+E_5)\text{ if }v\in[0,1],\\
\frac{v}{2} (2E_3+2E+E_1+E_2+E_5)+(v-1)C\text{ if }v\in[1,2].
\end{cases}
\end{align*}}
Moreover, 
$$(P(v))^2=\begin{cases}
1-\frac{v^2}{2}  \text{ if }v\in[0,1],\\
\frac{(2-v)^2}{2}  \text{ if }v\in[1,2].
\end{cases}
P(v)\cdot E_4=\begin{cases}
\frac{v}{2}  \text{ if }v\in[0,1],\\
1-\frac{v}{2}  \text{ if }v\in[1,2].
\end{cases}$$
Now we apply the computation from Section \ref{dP1-D4} (Step 1.) and get $\delta_P(S)=1$ for $P\in E_4\backslash E_3$.
\par{\bf Step 5.} Suppose $P\in E_5$. Then $\tau(E_5)=1$ and the  Zariski decomposition of the divisor  $-K_S-vE_5\sim C+E_1+E_2+2E+2E_3+2E_4+(1-v)E_5$ is given by:
\begin{align*}
&&P(v)=
-K_S-vE_5-\frac{v}{2} (2E_4+2E_3+2E+E_1+E_2)\text{ if }v\in[0,1].
\\&&
N(v)=
\frac{v}{2} (2E_4+2E_3+2E+E_1+E_2)\text{ if }v\in[0,1].
\end{align*}
Moreover, 
$$(P(v))^2=(1-v)(1+v)\text{ and }
P(v)\cdot E_5=v\text{ if }v\in[0,1].$$
Now we apply the computation from Section \ref{dP1-D4} (Step 2.) and get that $\delta_P(S)=\frac{3}{2}$ for $P\in E_5\backslash E_4$. 
\\Thus, $\delta_{\mathcal{P}} (X)=\frac{3}{4}$.
\end{proof}

\subsubsection{$\mathbb{D}_7$ singularity on Du Val Del Pezzo surfaces of degree $1$}\label{dP1-D7}
 \begin{lemma} Let $X$ be a singular del Pezzo surface of degree $1$ with an $\mathbb{D}_7$ singularity at point $\mathcal{P}$.  Let $\mathcal{C}$ be a~curve in the~pencil $|-K_X|$ that contains~$\mathcal{P}$.  Then $\delta_{\mathcal{P}} (X)=\frac{2}{3}$.
 \end{lemma}
 \begin{proof}
Let $S$ be the minimal resolution of singularities.  Then $S$ is a weak del Pezzo surface of degree $1$. Suppose $C$ is a strict transform of $\mathcal{C}$ on $S$ and $E$, $E_1$, $E_2$, $E_3$, $E_4$, $E_5$ and $E_6$ are the exceptional divisors with the intersection:
\begin{figure}[h!]
    \centering
 \includegraphics[width=6cm]{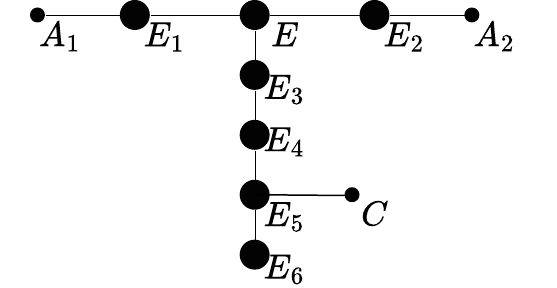}
    \caption{Dual graph: $(-K_S)^2=1$, $\mathbb{D}_7$ singularity}
\end{figure}
\par
We have $-K_S\sim C+E_1+E_2+2E+2E_3+2E_4+2E_5+E_6$. Let $P$ be a point on $S$.
 \par{\bf Step 1.} Suppose $P\in  E$.  If we contract the curve $C$  the resulting surface is isomorphic to a weak del Pezzo surface of degree $2$  with at most Du Val singularities.  Thus, there exist $(-1)$-curves  $A_1$ and $A_2$ which form the dual graph above with the rest of the curves.
Then $\tau(E)=\frac{5}{2}$ and the  Zariski decomposition of the divisor  $-K_S-vE$ is:
{  {\allowdisplaybreaks\begin{align*}
\hspace*{-0.5cm}&& P(v)=\begin{cases}
-K_S-vE-\frac{v}{2} (E_1+E_2)-\frac{v}{5} (4E_3+3E_4+2E_5+E_6)\text{ if }v\in[0,2],\\
-K_S-vE-(v-1)(E_1+E_2)-\frac{v}{5} (4E_3+3E_4+2E_5+E_6)-(v-2)(A_1+A_2)\text{ if }v\in\big[2,\frac{5}{2}\big].
\end{cases} \\
\hspace*{-0.5cm}&&\text{}N(v)=\begin{cases}
\frac{v}{2} (E_1+E_2)+\frac{v}{5} (4E_3+3E_4+2E_5+E_6)\text{ if }v\in[0,2],\\
(v-1)(E_1+E_2)+\frac{v}{5} (4E_3+3E_4+2E_5+E_6)+(v-2)(A_1+A_2)\text{ if }v\in\big[2,\frac{5}{2}\big].
\end{cases}
\end{align*}}}
The Zariski Decomposition follows from 
 $$-K_S-vE\sim_{\DR} \Big(\frac{5}{2}-v\Big)E+\frac{1}{2}\Big(3E_1+3E_2+4E_3+3E_4+2E_5+E_6+A_1+A_2\Big).$$
Moreover, 
$$(P(v))^2=\begin{cases}
1-\frac{v^2}{5}  \text{ if }v\in[0,2],\\
\frac{(5-2v)^2}{5}  \text{ if }v\in\big [2,\frac{5}{2}\big ].
\end{cases}
P(v)\cdot E=\begin{cases}
\frac{v}{5}  \text{ if }v\in[0,2],\\
2(1-\frac{2v}{5})  \text{ if }v\in\big [2,\frac{5}{2}\big ].
\end{cases}$$
We have
$S_{S} (E)=\frac{3}{2}$.
Thus, $\delta_P(S)\le \frac{2}{3}$ for $P\in E$. Moreover, if $P\in E\cap E_3$ if $P\in E\backslash E_3$   we have 
$$
h(v)\le\begin{cases}
\frac{9v^2}{50} \text{ if }v\in \big [0, 2\big ],\\
\frac{2(5 - 2v)(2v + 5)}{25} \text{ if }v\in\big [2,\frac{5}{2}\big ].
\end{cases}
\text{ or }
h(v)\le\begin{cases}
\frac{3v^2}{25} \text{ if }v\in \big [0, 2\big ],\\
\frac{6v(5 - 2v)}{25} \text{ if }v\in\big [2,\frac{5}{2}\big ].
\end{cases}
$$
Thus
$S(W_{\bullet,\bullet}^{E};P)\le \frac{4}{3}<\frac{3}{2}$
or
$S(W_{\bullet,\bullet}^{E};P)\le \frac{9}{10}<\frac{3}{2}$.
We get $\delta_P(S)=\frac{2}{3}$ for $P\in E$.
\par{\bf Step 2.} Suppose $P\in  E_1\cup E_2$. ithout loss of generality we can assume that $P\in E_1$ since the proof is similar
in other cases. Then $\tau(E_1)=\frac{3}{2}$ and the  Zariski decomposition of the divisor  $-K_S-vE_3$ is the following:
{\small  
 {\allowdisplaybreaks\begin{align*}
    \hspace*{-0.5cm}&&P(v)=\begin{cases}
-K_S-vE_1-\frac{v}{7} (10E+5E_2+8E_3+6E_4+4E_5+2E_6)\text{ if }v\in[0,1],\\
-K_S-vE_1-\frac{v}{7} (10E+5E_2+8E_3+6E_4+4E_5+2E_6)-(v-1)A_1\text{ if }v\in\big[1,\frac{7}{5}\big],\\
-K_S-vE_1-(v-1)(10E+8E_3+6E_4+4E_5+2E_6+A_1)-(5v-6)E_2-(5v-7)A_2\text{ if }v\in\big[\frac{7}{5},\frac{3}{2}\big].
\end{cases}\\
\hspace*{-0.5cm}&&N(v)=\begin{cases}
\frac{v}{7} (10E+5E_2+8E_3+6E_4+4E_5+2E_6)\text{ if }v\in[0,1],\\
\frac{v}{7} (10E+5E_2+8E_3+6E_4+4E_5+2E_6)+(v-1)A_1\text{ if }v\in\big[1,\frac{7}{5}\big],\\
(v-1)(10E+8E_3+6E_4+4E_5+2E_6+A_1)+(5v-6)E_2+(5v-7)A_2\text{ if }v\in\big[\frac{7}{5},\frac{3}{2}\big].
\end{cases}
\end{align*}}}
Then $\tau(E_1)=\frac{3}{2}$ and the  Zariski Decomposition follows from 
 $$-K_S-vE_1\sim_{\DR} \Big(\frac{3}{2}-v\Big)E_1+\frac{1}{2}\Big(2A_2+3E_2+5E+4E_3+3E_4+2E_5+E_6+A_1\Big).$$
Moreover, 
$$(P(v))^2=\begin{cases}
1-\frac{4v^2}{7}  \text{ if }v\in[0,1],\\
2 - 2v + \frac{3v^2}{7}\text{ if }v\in\big[1,\frac{7}{5}\big],\\
(3-2v)^2  \text{ if }v\in\big [\frac{7}{5}, \frac{3}{2}\big ].
\end{cases}
P(v)\cdot E_1=\begin{cases}
\frac{4v}{7}  \text{ if }v\in[0,1],\\
1-\frac{3v}{7}\text{ if }v\in\big[1,\frac{7}{5}\big],\\
2(3-2v)  \text{ if }v\in\big [\frac{7}{5}, \frac{3}{2}\big ].
\end{cases}$$
We have
$S_{S} (E_3)=\frac{9}{10}$.
Thus, $\delta_P(S)\le \frac{10}{9}$ for $P\in E_1$. Moreover, if $P\in E_1\backslash E$   we have 
 $$h(v)\le\begin{cases}
\frac{8v^2}{49}  \text{ if }v\in[0,1],\\
\frac{(7 - 3v)(11v - 7)}{98}\text{ if }v\in\big[1,\frac{7}{5}\big],\\
2(3 - 2v )(2 - v)  \text{ if }v\in\big [\frac{7}{5}, \frac{3}{2}\big ].
\end{cases}
$$
Thus,
$S(W_{\bullet,\bullet}^{E_1};P)\le \frac{3}{10}<\frac{9}{10}$.
We get $\delta_P(S)=\frac{10}{9}$ for $P\in (E_1\cup E_2)\backslash E$.
\par{\bf Step 3.} Suppose $P\in E_3$. Then $\tau(E_3)=2$ and the Zariski decomposition of the divisor $-K_S-vE_3\sim C+E_1+E_2+2E+(2-v)E_3+2E_4+2E_5+E_6$ is given by:
 {\allowdisplaybreaks\begin{align*}
&&P(v)=
-K_S-vE_3-\frac{v}{4} (2E_1+2E_2+4E+3E_4+2E_5+E_6)\text{ if }v\in[0,2].\\
&&
N(v)=
\frac{v}{4} (2E_1+2E_2+4E+3E_4+2E_5+E_6)\text{ if }v\in[0,2].
\end{align*}}
Moreover, 
$$(P(v))^2=\frac{(2-v)(2+v)}{4}
\text{ and }P(v)\cdot E_3=\frac{v}{4}\text{ if }v\in[0,2].$$
Now we apply the computation from Section \ref{dP1-D6} (Step 1.) and get that $\delta_P(S)=\frac{3}{4}$ for $P\in E_3\backslash E$. 
\\
{\bf Step 4.} Suppose $P\in  E_4$. Then $\tau(E_4)=2$ and the  Zariski decomposition of the divisor $-K_S-vE_4\sim C+E_1+E_2+2E+2E_3+(2-v)E_4+2E_5+E_6$ is:
 {\allowdisplaybreaks\begin{align*}
&&P(v)=\begin{cases}
-K_S-vE_4-\frac{v}{6} (3E_1+3E_2+6E+6E_3+4E_5+2E_6)\text{ if }v\in\big[0,\frac{3}{2}\big],\\
-K_S-vE_4-\frac{v}{2} (E_1+E_2+2E+2E_3)-(v-1)(2E_5+E_6)-(2v-3)C\text{ if }v\in\big[\frac{3}{2},2\big].
\end{cases}\\ 
&&N(v)=\begin{cases}
\frac{v}{6} (3E_1+3E_2+6E+6E_3+4E_5+2E_6)\text{ if }v\in\big[0,\frac{3}{2}\big],\\
\frac{v}{2} (E_1+E_2+2E+2E_3)+(v-1)(2E_5+E_6)+(2v-3)C\text{ if }v\in\big[\frac{3}{2},2\big].
\end{cases}  
\end{align*}}
Moreover, 
$$(P(v))^2=\begin{cases}
1-\frac{v^2}{3}  \text{ if }v\in\big [0,\frac{3}{2}\big ],\\
(2-v)^2  \text{ if }v\in\big[\frac{3}{2},2\big].
\end{cases}
P(v)\cdot E_4=\begin{cases}
\frac{v}{3}  \text{ if }v\in\big [0,\frac{3}{2}\big ],\\
2-v  \text{ if }v\in\big[\frac{3}{2},2\big].
\end{cases}$$
Now we apply the computation from Section \ref{dP1-D5} (Step 1.) and get that $\delta_P(S)=\frac{6}{7}$ for $P\in E_4\backslash E_3$.
\par{\bf Step 5.} Suppose $P\in  E_5$. Then $\tau(E_5)=2$ and the  Zariski decomposition of the divisor $-K_S-vE_5\sim C+E_1+E_2+2E+2E_3+2E_4+(2-v)E_5+E_6$ is:
 {\allowdisplaybreaks\begin{align*}
&& P(v)=\begin{cases}
-K_S-vE_5-\frac{v}{2} (2E_4+2E_3+2E+E_1+E_2+E_6)\text{ if }v\in[0,1],\\
-K_S-vE_5-\frac{v}{2} (2E_4+2E_3+2E+E_1+E_2+E_6)-(v-1)C\text{ if }v\in[1,2].
\end{cases} \\
&&\text{} N(v)=\begin{cases}
\frac{v}{2} (2E_4+2E_3+2E+E_1+E_2+E_6)\text{ if }v\in[0,1],\\
\frac{v}{2} (2E_4+2E_3+2E+E_1+E_2+E_6)+(v-1)C\text{ if }v\in[1,2].
\end{cases}
\end{align*}}
Moreover, 
$$(P(v))^2=\begin{cases}
1-\frac{v^2}{2}  \text{ if }v\in[0,1],\\
\frac{(2-v)^2}{2}  \text{ if }v\in[1,2].
\end{cases}
P(v)\cdot E_5=\begin{cases}
\frac{v}{2}  \text{ if }v\in[0,1],\\
1-\frac{v}{2}  \text{ if }v\in[1,2].
\end{cases}$$
Now we apply the computation from Section \ref{dP1-D4} (Step 1.) and get that $\delta_P(S)=1$ for $P\in E_5\backslash E_4$.
\par{\bf Step 6.} Suppose $P\in E_6$. Then $\tau(E_6)=1$ and the  Zariski decomposition of the divisor $-K_S-vE_6\sim C+E_1+E_2+2E+2E_3+2E_4+2E_5+(1-v)E_6$ is given by:
 {\allowdisplaybreaks\begin{align*}
    && P(v)=
-K_S-vE_6-\frac{v}{2} (2E_5+2E_4+2E_3+2E+E_1+E_2)\text{ if }v\in[0,1].\\
&&
N(v)=
\frac{v}{2} (2E_5+2E_4+2E_3+2E+E_1+E_2)\text{ if }v\in[0,1].
\end{align*}}
Moreover, 
$$(P(v))^2=(1-v)(1+v)\text{ and }
P(v)\cdot E_6=v\text{ if }v\in[0,1].$$
Now we apply the computation from Section \ref{dP1-D4} (Step 2.) and get that $\delta_P(S)=\frac{3}{2}$ for $P\in E_6\backslash E_5$. Thus, $\delta_{\mathcal{P}} (X)=\frac{2}{3}$.
 \end{proof}
 \subsubsection{$\mathbb{D}_8$ singularity on Du Val Del Pezzo surfaces of degree $1$}\label{dP1-D8}
 \begin{lemma} Let $X$ be a singular del Pezzo surface of degree $1$ with an $\mathbb{D}_8$ singularity at point $\mathcal{P}$.  Let $\mathcal{C}$ be a~curve in the~pencil $|-K_X|$ that contains~$\mathcal{P}$.  Then $\delta_{\mathcal{P}} (X)=\frac{3}{5}$.
 \end{lemma}
 \begin{proof}
Let $S$ be the minimal resolution of singularities.  Then $S$ is a weak del Pezzo surface of degree $1$. Suppose $C$ is a strict transform of $\mathcal{C}$ on $S$ and $E$, $E_1$, $E_2$, $E_3$, $E_4$, $E_5$, $E_6$ and $E_7$ are the exceptional divisors with the intersection:
\begin{figure}[h!]
    \centering
 \includegraphics[width=5cm]{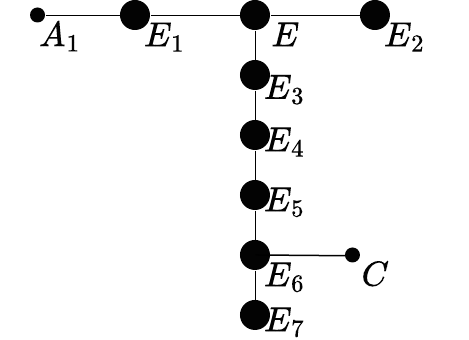}
    \caption{Dual graph: $(-K_S)^2=1$, $\mathbb{D}_8$ singularity}
\end{figure}
\par
 We have $-K_S\sim C+E_1+E_2+2E+2E_3+2E_4+2E_5+2E_6+E_7$. Let $P$ be a point on $S$.
 \par{\bf Step 1.} Suppose $P\in  E$.There exist a $(-1)$-curve $A_1$  which form the dual graph above with the rest of the curves.
Then the corresponding Zariski Decomposition of the divisor $-K_S-vE$ is:
 {\allowdisplaybreaks\begin{align*}
\hspace*{-0.5cm}&& P(v)=\begin{cases}
-K_S-vE-\frac{v}{6} (3E_1+3E_2+5E_3+4E_4+3E_5+2E_6+E_7)\text{ if }v\in[0,2],\\
-K_S-vE-(v-1)E_1-\frac{v}{6} (3E_2+5E_3+4E_4+3E_5+2E_6+E_7)-(v-2)A_1\text{ if }v\in[2,3].
\end{cases} \\
\hspace*{-0.5cm}&&N(v)=\begin{cases}
\frac{v}{6} (3E_1+3E_2+5E_3+4E_4+3E_5+2E_6+E_7)\text{ if }v\in[0,2],\\
(v-1)E_1+\frac{v}{6} (3E_2+5E_3+4E_4+3E_5+2E_6+E_7)+(v-2)A_1\text{ if }v\in[2,3].
\end{cases}
\end{align*}}
Then $\tau(E)=3$ and the  Zariski Decomposition follows from 
 $$-K_S-vE\sim_{\DR} (3-v)E+\frac{1}{2}\Big(4E_1+3E_2+5E_3+4E_4+3E_5+2E_6+E_7+2A_1\Big).$$
Moreover, 
$$(P(v))^2=\begin{cases}
1-\frac{v^2}{6}  \text{ if }v\in[0,2],\\
\frac{(3-v)^2}{3}  \text{ if }v\in[2,3].
\end{cases}
P(v)\cdot E=\begin{cases}
\frac{v}{6}  \text{ if }v\in[0,2],\\
1-\frac{v}{6}  \text{ if }v\in[2,3].
\end{cases}$$
We have
$S_{S} (E)=\frac{5}{3}$.
Thus, $\delta_P(S)\le \frac{3}{5}$ for $P\in E$. Moreover, if $P\in E\cap E_1$ if $P\in E\backslash E_1$   we have 
$$
h(v)\le\begin{cases}
\frac{7v^2}{72} \text{ if }v\in \big [0, 2\big ],\\
\frac{(3-v)(5v-3)}{18} \text{ if }v\in [2,3].
\end{cases}
\text{ or }
h(v)\le\begin{cases}
\frac{11v^2}{72} \text{ if }v\in \big [0, 2\big ],\\
\frac{(3-v)(4v+3)}{18} \text{ if }v\in [2,3].
\end{cases}
$$
Thus,
$S(W_{\bullet,\bullet}^{E};P)\le 1<\frac{5}{3}$
or
$S(W_{\bullet,\bullet}^{E};P)\le \frac{3}{2}<\frac{5}{3}$.
We get $\delta_P(S)=\frac{3}{5}$ for $P\in E$.
\par{\bf Step 2.} Suppose $P\in  E_1$. Then $\tau(E_1)=2$ and the  Zariski decomposition of the divisor $-K_S-vE_1$ is:
 {\allowdisplaybreaks\begin{align*}
\hspace*{-0.5cm}&& P(v)=\begin{cases}
-K_S-vE_1-\frac{v}{4} (6E+3E_2+5E_3+4E_4+3E_5+2E_6+E_7)\text{ if }v\in[0,1],\\
-K_S-vE_1-\frac{v}{4} (6E+3E_2+5E_3+4E_4+3E_5+2E_6+E_7)-(v-1)A_1\text{ if }v\in[1,2].
\end{cases} \\
\hspace*{-0.5cm}&& N(v)=\begin{cases}
\frac{v}{4} (6E+3E_2+5E_3+4E_4+3E_5+2E_6+E_7)\text{ if }v\in[0,1],\\
\frac{v}{4} (6E+3E_2+5E_3+4E_4+3E_5+2E_6+E_7)+(v-1)A_1\text{ if }v\in[1,2].
\end{cases}
\end{align*}}
The Zariski Decomposition follows from 
 $$-K_S-vE_1\sim_{\DR} (2-v)E_1+\frac{1}{2}\Big(6E+3E_2+5E_3+4E_4+3E_5+2E_6+E_7+2A_1\Big).$$
Moreover, 
$$(P(v))^2=\begin{cases}
1-\frac{v^2}{2}  \text{ if }v\in[0,1],\\
\frac{(2-v)^2}{2}  \text{ if }v\in[1,2].
\end{cases}
P(v)\cdot E_1=\begin{cases}
\frac{v}{2}  \text{ if }v\in[0,1],\\
1-\frac{v}{2}  \text{ if }v\in[1,2].
\end{cases}$$
Now we apply the computation from Section \ref{dP1-D4} (Step 1.) and get $\delta_P(S)=1$ for $P\in E_1\backslash E$.
\par{\bf Step 3.} Suppose $P\in  E_2$. Then $\tau(E_2)=\frac{3}{2}$ and the  Zariski decomposition of the divisor $-K_S-vE_2$ is the following:
{ 
 {\allowdisplaybreaks\begin{align*}
\hspace*{-1cm}&& P(v)=\begin{cases}
-K_S-vE_2-\frac{v}{4} (6E+3E_1+5E_3+4E_4+3E_5+2E_6+E_7)\text{ if }v\in\big[0,\frac{4}{3}\big],\\
-K_S-vE_2-(v-1)(6E+5E_3+4E_4+3E_5+2E_6+E_7)-(6v-7)E_1-(6v-8)A_1\text{ if }v\in\big[\frac{4}{3},\frac{3}{2}\big].
\end{cases} \\
\hspace*{-1cm}&&N(v)=\begin{cases}
\frac{v}{4} (6E+3E_1+5E_3+4E_4+3E_5+2E_6+E_7)\text{ if }v\in\big[0,\frac{4}{3}\big],\\
(v-1)(6E+5E_3+4E_4+3E_5+2E_6+E_7)+(6v-7)E_1+(6v-8)A_1\text{ if }v\in\big[\frac{4}{3},\frac{3}{2}\big].
\end{cases}
\end{align*}}}
The Zariski Decomposition follows from 
 $$-K_S-vE_2\sim_{\DR} \Big(\frac{3}{2}-v\Big)E_2+\frac{1}{2}\Big(6E+5E_3+4E_4+3E_5+2E_6+E_7+ 4E_1+2A_1\Big).$$
Moreover, 
$$(P(v))^2=\begin{cases}
1-\frac{v^2}{2}  \text{ if }v\in\big[0,\frac{4}{3}\big],\\
(3-2v)^2 \text{ if }v\in\big[\frac{4}{3},\frac{3}{2}\big].
\end{cases}
P(v)\cdot E_2=\begin{cases}
\frac{v}{2}  \text{ if }v\in \big[0,\frac{4}{3}\big],\\
2(3-2v) \text{ if }v\in \big[\frac{4}{3},\frac{3}{2}\big].
\end{cases}$$
We have
$S_{S} (E_2)=\frac{17}{18}$. Thus, $\delta_P(S)\le \frac{18}{17}$ for $P\in E_2$. Moreover, if $P\in E_2\backslash E_1$   we have:
$$h(v)\le\begin{cases}
\frac{v^2}{8} \text{ if }v\in \big [0,\frac{4}{3}\big ],\\
2(3-2v)^2 \text{ if }v\in \big [\frac{4}{3},\frac{3}{2}\big ].
\end{cases}
$$
Thus,
$S(W_{\bullet,\bullet}^{E_2};P)\le  \frac{2}{9} <\frac{17}{18}$.
We get $\delta_P(S)=\frac{18}{17}$ for $P\in E_2\backslash E_1$.
 \par{\bf Step 4.} Suppose $P\in  E_3$. Then $\tau(E_3)=\frac{5}{2}$ and the   Zariski decomposition of the divisor $-K_S-vE_3$ is  the following:
{  {\allowdisplaybreaks\begin{align*}
\hspace*{-1cm}&& P(v)=\begin{cases}
-K_S-vE_3-\frac{v}{2} (2E+E_1+E_2)-\frac{v}{5} (4E_3+3E_4+2E_5+E_6)\text{ if }v\in[0,2],\\
-K_S-vE_3-(v-1)(2E+E_2)-\frac{v}{5} (4E_3+3E_4+2E_5+E_6)-(2v-3)E_1-(2v-4)A_1\text{ if }v\in\big[2,\frac{5}{2}\big].
\end{cases} \\
\hspace*{-1cm}&&N(v)=\begin{cases}
\frac{v}{2} (2E+E_1+E_2)+\frac{v}{5} (4E_3+3E_4+2E_5+E_6)\text{ if }v\in[0,2],\\
(v-1)(2E+E_2)+\frac{v}{5} (4E_3+3E_4+2E_5+E_6)+(2v-3)E_1+(2v-4)A_1\text{ if }v\in\big[2,\frac{5}{2}\big].
\end{cases}
\end{align*}}}
The Zariski Decomposition follows from 
 $$-K_S-vE_3\sim_{\DR} \Big(\frac{5}{2}-v\Big)E_3+\frac{1}{2}\Big(6E+3E_2+4E_3+3E_4+2E_5+E_6+4E_1+2A_1\Big).$$
Moreover, 
$$(P(v))^2=\begin{cases}
1-\frac{v^2}{5}  \text{ if }v\in[0,2],\\
\frac{(5-2v)^2}{5}  \text{ if }v\in\big [2,\frac{5}{2}\big ].
\end{cases}
P(v)\cdot E_3=\begin{cases}
\frac{v}{5}  \text{ if }v\in[0,2],\\
2(1-\frac{2v}{5})  \text{ if }v\in\big [2,\frac{5}{2}\big ].
\end{cases}$$
Now we apply the computation from Section \ref{dP1-D7} (Step 1.) and get that $\delta_P(S)=\frac{2}{3}$ for $P\in E_3\backslash E$. 
\par{\bf Step 5.} Suppose $P\in E_4$. Then $\tau(E_4)=2$ and the  Zariski decomposition of the divisor $-K_S-vE_4\sim C+E_1+E_2+2E+2E_3+(2-v)E_4+2E_5+2E_6+E_7$ is given by:
 {\allowdisplaybreaks\begin{align*}
&&P(v)=
-K_S-vE_4-\frac{v}{4} (2E_1+2E_2+4E+4E_3+3E_4+2E_5+E_6)\text{ if }v\in[0,2].\\
&&
N(v)=
\frac{v}{4} (2E_1+2E_2+4E+4E_3+3E_4+2E_5+E_6)\text{ if }v\in[0,2].
\end{align*}}
Moreover, 
$$(P(v))^2=\frac{(2-v)(2+v)}{4}
\text{ and }P(v)\cdot E_4=\frac{v}{4}\text{ if }v\in[0,2].$$
Now we apply the computation from Section \ref{dP1-D6} (Step 1.) and get that $\delta_P(S)=\frac{3}{4}$ for $P\in E_4\backslash E_3$. 
\\
{\bf Step 6.} Suppose $P\in  E_5$. Then $\tau(E_5)=2$ and the  Zariski decomposition of the divisor $-K_S-vE_5\sim C+E_1+E_2+2E+2E_3+2E_4+(2-v)E_5+2E_6+E_7$ is:
 {\allowdisplaybreaks\begin{align*}
\hspace*{-0.5cm}&&P(v)=\begin{cases}
-K_S-vE_5-\frac{v}{6} (3E_1+3E_2+6E+6E_3+6E_4+4E_6+2E_7)\text{ if }v\in\big[0,\frac{3}{2}\big],\\
-K_S-vE_5-\frac{v}{2} (E_1+E_2+2E+2E_3+2E_4)-(v-1)(2E_6+E_7)-(2v-3)C\text{ if }v\in\big[\frac{3}{2},2\big].
\end{cases}\\ 
\hspace*{-0.5cm}&& N(v)=\begin{cases}
\frac{v}{6} (3E_1+3E_2+6E+6E_3+6E_4+4E_6+2E_7)\text{ if }v\in\big[0,\frac{3}{2}\big],\\
\frac{v}{2} (E_1+E_2+2E+2E_3+2E_4)+(v-1)(2E_6+E_7)+(2v-3)C\text{ if }v\in\big[\frac{3}{2},2\big].
\end{cases}  
\end{align*}}
Moreover, 
$$(P(v))^2=\begin{cases}
1-\frac{v^2}{3}  \text{ if }v\in\big [0,\frac{3}{2}\big ],\\
(2-v)^2  \text{ if }v\in\big[\frac{3}{2},2\big].
\end{cases}
P(v)\cdot E_5=\begin{cases}
\frac{v}{3}  \text{ if }v\in\big [0,\frac{3}{2}\big ],\\
2-v  \text{ if }v\in\big[\frac{3}{2},2\big].
\end{cases}$$
Now we apply the computation from Section \ref{dP1-D5} (Step 1.) and get that $\delta_P(S)=\frac{6}{7}$ for $P\in E_5\backslash E_4$.
\par{\bf Step 7.} Suppose $P\in  E_6$. Then $\tau(E_6)=2$ and the  Zariski decomposition of the divisor $-K_S-vE_6\sim C+E_1+E_2+2E+2E_3+2E_4+2E_5+(2-v)E_6+E_7$ is:
 {\allowdisplaybreaks\begin{align*}
&& P(v)=\begin{cases}
-K_S-vE_6-\frac{v}{2} (2E_5+2E_4+2E_3+2E+E_1+E_2+E_6)\text{ if }v\in[0,1],\\
-K_S-vE_6-\frac{v}{2} (2E_5+2E_4+2E_3+2E+E_1+E_2+E_6)-(v-1)C\text{ if }v\in[1,2].
\end{cases} \\
&&\text{} N(v)=\begin{cases}
\frac{v}{2} (2E_5+2E_4+2E_3+2E+E_1+E_2+E_6)\text{ if }v\in[0,1],\\
\frac{v}{2} (2E_5+2E_4+2E_3+2E+E_1+E_2+E_6)+(v-1)C\text{ if }v\in[1,2].
\end{cases}
\end{align*}}
Moreover, 
$$(P(v))^2=\begin{cases}
1-\frac{v^2}{2}  \text{ if }v\in[0,1],\\
\frac{(2-v)^2}{2}  \text{ if }v\in[1,2].
\end{cases}
P(v)\cdot E_6=\begin{cases}
\frac{v}{2}  \text{ if }v\in[0,1],\\
1-\frac{v}{2}  \text{ if }v\in[1,2].
\end{cases}$$
Now we apply the computation from Section \ref{dP1-D4} (Step 1.) and get that $\delta_P(S)=1$ for $P\in E_6\backslash E_5$.
\par{\bf Step 8.} Suppose $P\in E_7$. Then $\tau(E_7)=1$ and the  Zariski decomposition of the divisor $-K_S-vE_7\sim C+E_1+E_2+2E+2E_3+2E_4+2E_5+2E_6+(1-v)E_7$ is given by:
 {\allowdisplaybreaks\begin{align*}
   \hspace*{-0.5cm} && P(v)=
-K_S-vE_7-\frac{v}{2} (2E_6+2E_5+2E_4+2E_3+2E+E_1+E_2)\text{ if }v\in[0,1].\\
\hspace*{-0.5cm}&&
N(v)=
\frac{v}{2} (2E_6+2E_5+2E_4+2E_3+2E+E_1+E_2)\text{ if }v\in[0,1].
\end{align*}}
Moreover, 
$$(P(v))^2=(1-v)(1+v)\text{ and }
P(v)\cdot E_7=v\text{ if }v\in[0,1].$$
Now we apply the computation from Section \ref{dP1-D4} (Step 2.) and get that $\delta_P(S)=\frac{3}{2}$ for $P\in E_7\backslash E_6$. 
\\Thus, $\delta_{\mathcal{P}} (X)=\frac{3}{5}$.
\end{proof}
\subsubsection{$\mathbb{E}_6$ singularity on Du Val Del Pezzo surfaces of degree $1$}\label{dP1-E_6}
 \begin{lemma} Let $X$ be a singular del Pezzo surface of degree $1$ with an $\mathbb{E}_6$ singularity at point $\mathcal{P}$.  Let $\mathcal{C}$ be a~curve in the~pencil $|-K_X|$ that contains~$\mathcal{P}$.  Then $\delta_{\mathcal{P}} (X)=\frac{3}{5}$.
 \end{lemma}
 \begin{proof}
Let $S$ be the minimal resolution of singularities.  Then $S$ is a weak del Pezzo surface of degree $1$. Suppose $C$ is a strict transform of $\mathcal{C}$ on $S$ and $E$, $E_1$, $E_2$, $E_3$, $E_4$ and $E_5$ are the exceptional divisors with the intersection:
\begin{figure}[h!]
    \centering
 \includegraphics[width=6cm]{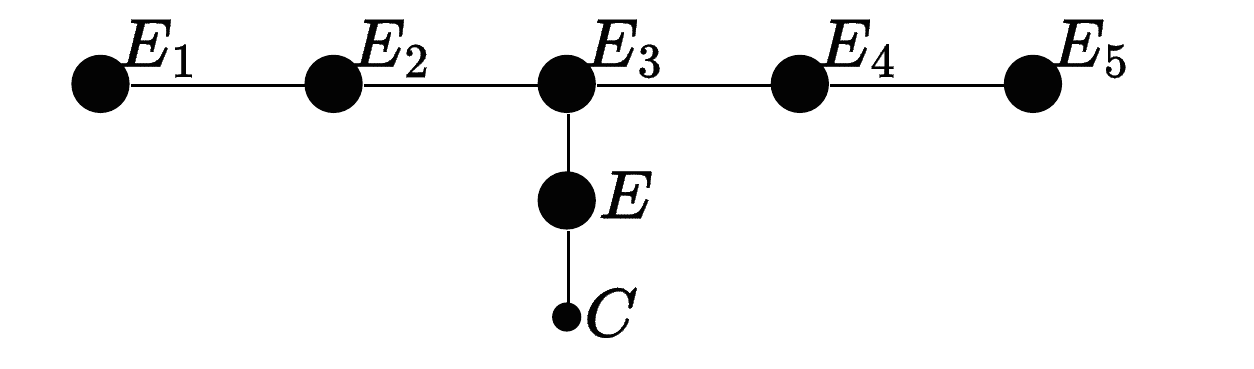}
    \caption{Dual graph: $(-K_S)^2=1$, $\mathbb{E}_6$ singularity}
\end{figure}
\par
We have $-K_S\sim C+E_1+2E_2+3E_3+2E_4+E_5+2E$. Let $P$ be a point on $S$.
  \par{\bf Step 1.} Suppose $P\in  E_3$. Then $\tau(E_3)=3$ and the   Zariski decomposition of the divisor $-K_S-vE_3\sim C+E_1+2E_2+(3-v)E_3+2E_4+E_5+2E$ is  the following:
 {\allowdisplaybreaks\begin{align*}
&& P(v)=\begin{cases}
-K_S-vE_3-\frac{v}{3} (E_1+2E_2+2E_4+E_5)-\frac{v}{2}E\text{ if }v\in[0,2],\\
-K_S-vE_3-\frac{v}{3} (E_1+2E_2+2E_4+E_5)-(v-1)E-(v-2)C\text{ if }v\in[2,3].
\end{cases} \\
&&\text{}N(v)=\begin{cases}
\frac{v}{3} (E_1+2E_2+2E_4+E_5)+\frac{v}{2}E\text{ if }v\in[0,2],\\
\frac{v}{3} (E_1+2E_2+2E_4+E_5)+(v-1)E+(v-2)C\text{ if }v\in[2,3].
\end{cases}
\end{align*}}
Moreover, 
$$(P(v))^2=\begin{cases}
1-\frac{v^2}{6}  \text{ if }v\in[0,2],\\
\frac{(3-v)^2}{3}  \text{ if }v\in[2,3].
\end{cases}
P(v)\cdot E_3=\begin{cases}
\frac{v}{6}  \text{ if }v\in[0,2],\\
1-\frac{v}{6}  \text{ if }v\in[2,3].
\end{cases}$$
Now we apply the computation from Section \ref{dP1-D8} (Step 1.) and get that $\delta_P(S)=\frac{3}{5}$ for $P\in E_3$. 
\par{\bf Step 2.} Suppose $P\in  E_2\cup E_4$. Without loss of generality we can assume that $P\in E_2$ since the proof
is similar in other cases. Then $\tau(E_2)=2$ and the  Zariski decomposition of the divisor $-K_S-vE_2\sim C+E_1+(2-v)E_2+3E_3+2E_4+E_5+2$ is:
  {\allowdisplaybreaks\begin{align*}
\hspace*{-0.5cm}&&P(v)=\begin{cases}
-K_S-vE_2-\frac{v}{2}E_1-\frac{v}{5} (3E+6E_3+4E_4+2E_5)\text{ if }v\in\big[0,\frac{5}{3}\big],\\
-K_S-vE_2-\frac{v}{2}E_1-(v-1)(3E_3+2E_4+E_5)-(3v-4)E-(3v-5)C
\text{ if }v\in\big[\frac{5}{3},2\big].
\end{cases}\\
\hspace*{-0.5cm}&&N(v)=\begin{cases}
\frac{v}{2}E_1+\frac{v}{5} (3E+6E_3+4E_4+2E_5)\text{ if }v\in\big[0,\frac{5}{3}\big],\\
\frac{v}{2}E_1+(v-1)(3E_3+2E_4+E_5)+(3v-4)E+(3v-5)C\text{ if }v\in\big[\frac{5}{3},2\big].
\end{cases}
\end{align*}}
Moreover, 
$$(P(v))^2=\begin{cases}
1-\frac{3v^2}{10}  \text{ if }v\in\big [0,\frac{5}{3}\big ],\\
\frac{3(2-v)^2}{2}  \text{ if }v\in\big[\frac{5}{3},2\big].
\end{cases}
P(v)\cdot E_2=\begin{cases}
\frac{3v}{10}  \text{ if }v\in\big [0,\frac{5}{3}\big ],\\
3(1-\frac{v}{2})  \text{ if }v\in\big[\frac{5}{3},2\big].
\end{cases}$$
We have $S_{S} (E_2)=\frac{11}{9}$. Thus, $\delta_P(S)\le \frac{9}{11}$ for $P\in E_2$. Moreover,  if $P\in E_2\backslash E_3$  we have 
$$h(v)\le\begin{cases}
\frac{39v^2}{200} \text{ if }v\in \big [0, \frac{5}{3}\big ],\\
\frac{3(v - 2)(v - 6)}{8}\text{ if }v\in\big[\frac{5}{3},2\big].
\end{cases}
$$
Thus,
$S(W_{\bullet,\bullet}^{E_2};P)\le  \frac{7}{9}< \frac{11}{9}$.
We get $\delta_P(S)=\frac{9}{11}$ for $P\in (E_2\cup E_4)\backslash E_3$.
\par{\bf Step 3.} Suppose $P\in  E_1\cup E_5$. Without loss of generality we can assume that
$P \in E_1$ since the proof is similar in other cases. There exist $(-1)$-curves and $(-2)$-curves  which form one of the following dual graphs:
\begin{figure}[h!]
    \centering
\includegraphics[width=15cm]{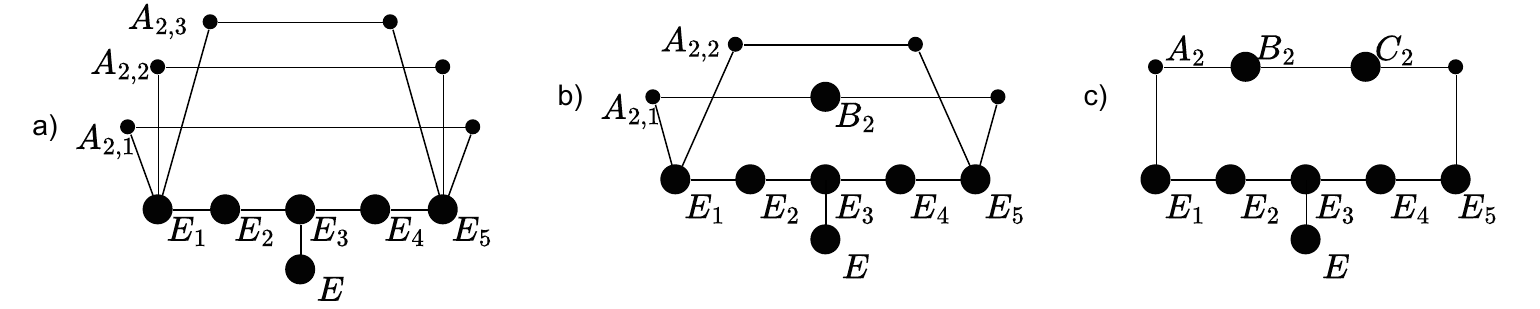}
    \caption{Dual graph: $(-K_S)^2=1$, $\mathbb{E}_6$ singularity, $\delta_P(S)=\frac{3}{5}$}
\end{figure}
\par Then the corresponding Zariski Decomposition of the divisor $-K_S-vE_1$ is:
   { 
 {\allowdisplaybreaks\begin{align*}
\hspace*{-0.7cm}&{\text{\bf a). }} & P(v)=\begin{cases}
-K_S-vE_1-\frac{v}{4} (5E_2+6E_3+4E_4+2E_5+3E)\text{ if }v\in[0,1],\\
-K_S-vE_1-\frac{v}{4} (5E_2+6E_3+4E_4+2E_5+3E)-(v-1)(A_{2,1}+A_{2,2}+A_{2,3})\text{ if }v\in\big[1,\frac{4}{3}\big].
\end{cases} \\
\hspace*{-0.7cm}&&N(v)=\begin{cases}
\frac{v}{4} (5E_2+6E_3+4E_4+2E_5+3E)\text{ if }v\in[0,1],\\
\frac{v}{4} (5E_2+6E_3+4E_4+2E_5+3E)+(v-1)(A_{2,1}+A_{2,2}+A_{2,3})\text{ if }v\in\big[1,\frac{4}{3}\big].
\end{cases}\\
\hspace*{-0.7cm}&{\text{\bf b). }} & P(v)=\begin{cases}
-K_S-vE_1-\frac{v}{4} (5E_2+6E_3+4E_4+2E_5+3E)\text{ if }v\in[0,1],\\
-K_S-vE_1-\frac{v}{4} (5E_2+6E_3+4E_4+2E_5+3E)-(v-1)(2A_{2,1}+B_{2,1}+A_{2,2})\text{ if }v\in\big[1,\frac{4}{3}\big].
\end{cases} \\
\hspace*{-0.7cm}&&N(v)=\begin{cases}
\frac{v}{4} (5E_2+6E_3+4E_4+2E_5+3E)\text{ if }v\in[0,1],\\
\frac{v}{4} (5E_2+6E_3+4E_4+2E_5+3E)+(v-1)(2A_{2,1}+B_{2,1}+A_{2,2})\text{ if }v\in\big[1,\frac{4}{3}\big].
\end{cases}\\
\hspace*{-0.7cm}&{\text{\bf c). }} & P(v)=\begin{cases}
-K_S-vE_1-\frac{v}{4} (5E_2+6E_3+4E_4+2E_5+3E)\text{ if }v\in[0,1],\\
-K_S-vE_1-\frac{v}{4} (5E_2+6E_3+4E_4+2E_5+3E)-(v-1)(3A_{2}+B_{2}+C_2)\text{ if }v\in\big[1,\frac{4}{3}\big].
\end{cases} \\
\hspace*{-0.7cm}&&N(v)=\begin{cases}
\frac{v}{4} (5E_2+6E_3+4E_4+2E_5+3E)\text{ if }v\in[0,1],\\
\frac{v}{4} (5E_2+6E_3+4E_4+2E_5+3E)+(v-1)(3A_{2}+B_{2}+C_2)\text{ if }v\in\big[1,\frac{4}{3}\big].
\end{cases}
\end{align*}}}
Then $\tau(E_1)=\frac{4}{3}$ and the  Zariski Decomposition in part a). follows from 
 $$-K_S-vE_1\sim_{\DR} \Big(\frac{4}{3}-v\Big)E_1+\frac{1}{3}\Big(5E_2+6E_3+4E_4+2E_5+3E+A_{2,1}+A_{2,2}+A_{2,3}\Big).$$
 A similar statement holds in other parts.
Moreover, 
$$(P(v))^2=\begin{cases}
1-\frac{3v^2}{4}  \text{ if }v\in[0,1],\\
\frac{(4-3v)^2}{4}  \text{ if }v\in\big [1, \frac{4}{3}\big ].
\end{cases}
P(v)\cdot E_1=\begin{cases}
\frac{3v}{4}  \text{ if }v\in[0,1],\\
3(1-\frac{3v}{4})  \text{ if }v\in\big [1, \frac{4}{3}\big ].
\end{cases}$$
We apply the computation from Section \ref{dP1-A5} (Step 2.) and get $\delta_P(S)=\frac{3}{5}$ if $P\in (E_1\cup E_5)\backslash (E_2\cup E_4)$.
\par{\bf Step 4.} Suppose $P\in  E$. Then $\tau(E)=2$ and the  Zariski decomposition of the divisor $-K_S-vE\sim C+E_1+2E_2+3E_3+2E_4+E_5+(2-v)E$ is:
 {\allowdisplaybreaks\begin{align*}
&& P(v)=\begin{cases}
-K_S-vE-\frac{v}{2} (E_1+2E_2+3E_3+2E_4+E_5)\text{ if }v\in[0,1],\\
-K_S-vE-\frac{v}{2} (E_1+2E_2+3E_3+2E_4+E_5)-(v-1)C\text{ if }v\in[1,2].
\end{cases} \\
&&\text{} N(v)=\begin{cases}
\frac{v}{2} (E_1+2E_2+3E_3+2E_4+E_5)\text{ if }v\in[0,1],\\
\frac{v}{2} (E_1+2E_2+3E_3+2E_4+E_5)+(v-1)C\text{ if }v\in[1,2].
\end{cases}
\end{align*}}
Moreover, 
$$(P(v))^2=\begin{cases}
1-\frac{v^2}{2}  \text{ if }v\in[0,1],\\
\frac{(2-v)^2}{2}  \text{ if }v\in[1,2].
\end{cases}
P(v)\cdot E=\begin{cases}
\frac{v}{2}  \text{ if }v\in[0,1],\\
1-\frac{v}{2}  \text{ if }v\in[1,2].
\end{cases}$$
Now we apply the computation from Section \ref{dP1-D4} (Step 1.) and get that $\delta_P(S)=1$ for $P\in E\backslash E_3$.
\\Thus, $\delta_{\mathcal{P}} (X)=\frac{3}{5}$.
 \end{proof}

\subsubsection{$\mathbb{E}_7$ singularity on Du Val Del Pezzo surfaces of degree $1$}\label{dP1-E7}
 \begin{lemma} Let $X$ be a singular del Pezzo surface of degree $1$ with an $\mathbb{E}_7$ singularity at point $\mathcal{P}$.  Let $\mathcal{C}$ be a~curve in the~pencil $|-K_X|$ that contains~$\mathcal{P}$.  Then $\delta_{\mathcal{P}} (X)=\frac{3}{7}$.
 \end{lemma}
 \begin{proof}
Let $S$ be the minimal resolution of singularities.  Then $S$ is a weak del Pezzo surface of degree $1$. Suppose $C$ is a strict transform of $\mathcal{C}$ on $S$ and $E$, $E_1$, $E_2$, $E_3$, $E_4$, $E_5$ and $E_6$ are the exceptional divisors with the intersection:
\begin{figure}[h!]
    \centering
 \includegraphics[width=9cm]{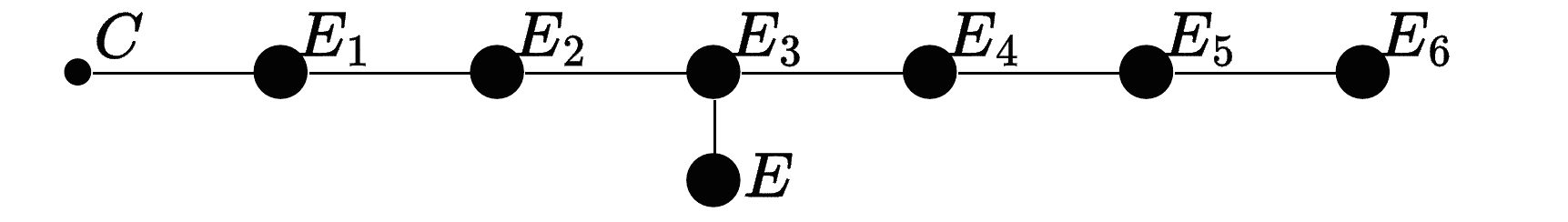}
    \caption{Dual graph: $(-K_S)^2=1$, $\mathbb{E}_7$ singularity}
\end{figure}
\par
We have $-K_S\sim C+2E_1+3E_2+4E_3+3E_4+2E_5+E_6+2E$. Let $P$ be a point on $S$.
\par{\bf Step 1.} Suppose $P\in  E_3$. Then $\tau(E_3)=4$ and the   Zariski decomposition of the divisor $-K_S-vE_3\sim C+2E_1+3E_2+(4-v)E_3+3E_4+2E_5+E_6+2E$ is  the following:
 {\allowdisplaybreaks\begin{align*}
\hspace*{-0.5cm}&& P(v)=\begin{cases}
-K_S-vE_3-\frac{v}{4} (2E+3E_4+2E_5+E_6)-\frac{v}{3} (E_1+2E_2)\text{ if }v\in[0,3],\\
-K_S-vE_3-\frac{v}{4} (2E+3E_4+2E_5+E_6)-(v-1)E_1-(v-2)E_2-(v-3)C\text{ if }v\in[3,4].
\end{cases} \\
\hspace*{-0.5cm}&&N(v)=\begin{cases}
\frac{v}{4} (2E+3E_4+2E_5+E_6)-\frac{v}{3} (E_1+2E_2)\text{ if }v\in[0,3],\\
\frac{v}{4} (2E+3E_4+2E_5+E_6)+(v-1)E_1+(v-2)E_2+(v-3)C\text{ if }v\in[3,4].
\end{cases}
\end{align*}}
Moreover, 
$$(P(v))^2=\begin{cases}
1-\frac{v^2}{12}  \text{ if }v\in[0,3],\\
\frac{(4-v)^2}{4}  \text{ if }v\in[3,4].
\end{cases}
P(v)\cdot E_3=\begin{cases}
\frac{v}{12}  \text{ if }v\in[0,3],\\
1-\frac{v}{4}  \text{ if }v\in[3,4].
\end{cases}$$
We have
$S_{S} (E_3)=\frac{7}{3}$. Thus, $\delta_P(S)\le \frac{3}{7}$ for $P\in E_3$. Moreover, if $P\in E_3\cap (E\cup E_4)$ if $P\in E_3\backslash (E\cup E_4)$   we have 
$$
h(v)\le\begin{cases}
\frac{19v^2}{228} \text{ if }v\in \big [0, 3\big ],\\
\frac{(4 - v )(5v + 4)}{32} \text{ if }v\in [3,4]
\end{cases}
\text{ or }
h(v)\le\begin{cases}
\frac{17v^2}{228} \text{ if }v\in \big [0, 3\big ],\\
\frac{(4 -v)(7v - 4)}{32} \text{ if }v\in [3,4]
\end{cases}
$$
Thus,
$S(W_{\bullet,\bullet}^{E_3};P)\le \frac{11}{6} <\frac{7}{3}$ or
$S(W_{\bullet,\bullet}^{E_3};P)\le \frac{5}{3} <\frac{7}{3}$.
We get $\delta_P(S)=\frac{3}{7}$ for $P\in E_3$.
 \par{\bf Step 2.} Suppose $P\in  E_2$. Then $\tau(E_2)=3$ and the   Zariski decomposition of the divisor $-K_S-vE_2\sim C+2E_1+(3-v)E_2+4E_3+3E_4+2E_5+E_6+2E$ is  the following:
 {\allowdisplaybreaks\begin{align*}
\hspace*{-0.5cm}&& P(v)=\begin{cases}
-K_S-vE_2-\frac{v}{3} (2E+4E_3+3E_4+2E_5+E_6)-\frac{v}{2}E_1\text{ if }v\in[0,2],\\
-K_S-vE_2-\frac{v}{3} (2E+4E_3+3E_4+2E_5+E_6)-(v-1)E_1-(v-2)C\text{ if }v\in[2,3].
\end{cases} \\
\hspace*{-0.5cm}&&\text{}N(v)=\begin{cases}
\frac{v}{3} (2E+4E_3+3E_4+2E_5+E_6)+\frac{v}{2}E_1\text{ if }v\in[0,2],\\
\frac{v}{3} (2E+4E_3+3E_4+2E_5+E_6)+(v-1)E_1+(v-2)C\text{ if }v\in[2,3].
\end{cases}
\end{align*}}
Moreover, 
$$(P(v))^2=\begin{cases}
1-\frac{v^2}{6}  \text{ if }v\in[0,2],\\
\frac{(3-v)^2}{3}  \text{ if }v\in[2,3].
\end{cases}
P(v)\cdot E_2=\begin{cases}
\frac{v}{6}  \text{ if }v\in[0,2],\\
1-\frac{v}{6}  \text{ if }v\in[2,3].
\end{cases}$$
Now we apply the computation from Section \ref{dP1-D8} (Step 1.) and get that $\delta_P(S)=\frac{3}{5}$ for $P\in E_2\backslash E_3$.
\par{\bf Step 3.} Suppose $P\in  E_1$. Then $\tau(E_1)=2$ and the  Zariski decomposition of the divisor $-K_S-vE_1\sim C+(2-v)E_1+3E_2+4E_3+3E_4+2E_5+E_6+2E$ is:
 {\allowdisplaybreaks\begin{align*}
&& P(v)=\begin{cases}
-K_S-vE_1-\frac{v}{2} (2E+3E_2+4E_3+3E_4+2E_5+E_6)\text{ if }v\in[0,1],\\
-K_S-vE_1-\frac{v}{2} (2E+3E_2+4E_3+3E_4+2E_5+E_6)-(v-1)C\text{ if }v\in[1,2].
\end{cases} \\
&&\text{} N(v)=\begin{cases}
\frac{v}{2} (2E+3E_2+4E_3+3E_4+2E_5+E_6)\text{ if }v\in[0,1],\\
\frac{v}{2} (2E+3E_2+4E_3+3E_4+2E_5+E_6)+(v-1)C\text{ if }v\in[1,2].
\end{cases}
\end{align*}}
Moreover, 
$$(P(v))^2=\begin{cases}
1-\frac{v^2}{2}  \text{ if }v\in[0,1],\\
\frac{(2-v)^2}{2}  \text{ if }v\in[1,2].
\end{cases}
P(v)\cdot E_1=\begin{cases}
\frac{v}{2}  \text{ if }v\in[0,1],\\
1-\frac{v}{2}  \text{ if }v\in[1,2].
\end{cases}$$
Now we apply the computation from Section \ref{dP1-D4} (Step 1.) and get that $\delta_P(S)=1$ for $P\in E_1\backslash E_2$.
\par{\bf Step 4.} Suppose $P\in  E$.  Then $\tau(E)=2$ and the  Zariski decomposition of the divisor $-K_S-vE\sim C+2E_1+3E_2+4E_3+3E_4+2E_5+E_6+(2-v)E$ is:
{   {\allowdisplaybreaks\begin{align*}
\hspace*{-1cm}&& P(v)=\begin{cases}
-K_S-vE-\frac{v}{7} (4E_1+8E_2+12E_3+9E_4+6E_5+3E_6) \text{ if }v\in\big[0,\frac{7}{4}\big],\\
-K_S-vE-(4v - 7)C-(4v-6)E_1-(4v - 5)E_2-(v-1)(4E_3+3E_4+2E_5+E_6)
\text{ if }v\in\big[\frac{7}{4},2\big].
\end{cases} \\
\hspace*{-1cm}&& N(v)=\begin{cases}
\frac{v}{7} (4E_1+8E_2+12E_3+9E_4+6E_5+3E_6)\text{ if }v\in\big[0,\frac{7}{4}\big],\\
(4v - 7)C+(4v-6)E_1+(4v - 5)E_2+(v-1)(4E_3+3E_4+2E_5+E_6)\text{ if }v\in\big[\frac{7}{4},2\big].
\end{cases}
\end{align*}}}
Moreover, 
$$(P(v))^2=\begin{cases}
1-\frac{2v^2}{7}  \text{ if }v\in\big [0,\frac{7}{4}\big ],\\
2(2-v)^2  \text{ if }v\in\big[\frac{7}{4},2\big].
\end{cases}
P(v)\cdot E=\begin{cases}
\frac{2v}{7}  \text{ if }v\in\big [0,\frac{7}{4}\big ],\\
2(2-v)  \text{ if }v\in\big[\frac{7}{4},2\big].
\end{cases}$$
We have
$S_{S} (E)=\frac{5}{4}$. Thus, $\delta_P(S)\le \frac{4}{5}$ for $P\in E$. Moreover,  if $P\in E\backslash E_3$   we have 
$$h(v)\le\begin{cases}
\frac{2v^2}{49} \text{ if }v\in \big [0, \frac{7}{4}\big ],\\
2(v - 2)^2\text{ if }v\in\big[\frac{7}{4},2\big].
\end{cases}
$$
Thus
$S(W_{\bullet,\bullet}^{E_2};P)\le  \frac{1}{6}< \frac{5}{4}$.
We get $\delta_P(S)=\frac{4}{5}$ for $P\in E\backslash E_3$.
\par{\bf Step 5.} Suppose $P\in  E_4$. Then $\tau(E_4)=3$ and the  Zariski decomposition of the divisor $-K_S-vE_4\sim C+2E_1+3E_2+4E_3+(3-v)E_4+2E_5+E_6+2E$ is:
{\small   {\allowdisplaybreaks\begin{align*}
\hspace*{-0.7cm}&&P(v)=\begin{cases}
-K_S-vE_4-\frac{v}{5} (2E_1+4E_2+6E_3+3E)-\frac{v}{3} (2E_5+E_6)\text{ if }v\in\big[0,\frac{5}{2}\big],\\
-K_S-vE_4-(2v - 5)C-(2v - 4)E_1-(2v - 3)E_2-(2v - 2)E_3-(v-1)E-\frac{v}{3} (2E_5+E_6)
\text{ if }v\in\big [\frac{5}{2}, 3\big ].
\end{cases} \\
\hspace*{-0.7cm}&& N(v)=\begin{cases}
\frac{v}{5} (2E_1+4E_2+6E_3+3E)-\frac{v}{3} (2E_5+E_6)\text{ if }v\in\big[0,\frac{5}{2}\big],\\
(2v - 5)C+(2v - 4)E_1+(2v - 3)E_2+(2v - 2)E_3+(v-1)E+\frac{v}{3} (2E_5+E_6)\text{ if }v\in\big [\frac{5}{2}, 3\big ].
\end{cases}
\end{align*}}}
Moreover, 
$$(P(v))^2=\begin{cases}
1-\frac{2v^2}{15}  \text{ if }v\in\big [0,\frac{5}{2}\big ],\\
\frac{2(3-v)^2}{3}  \text{ if }v\in\big [\frac{5}{2}, 3\big ].
\end{cases}
P(v)\cdot E_4=\begin{cases}
\frac{2v}{15}  \text{ if }v\in\big [0,\frac{5}{2}\big ],\\
2(1-\frac{v}{3})  \text{ if }v\in\big [\frac{5}{2}, 3\big ].
\end{cases}$$
We have
$S_{S} (E_4)=\frac{11}{6}$. Thus, $\delta_P(S)\le \frac{6}{11}$ for $P\in E_4$. Moreover,  if $P\in E_4\backslash E_3$   we have 
$$h(v)\le\begin{cases}
\frac{22v^2}{225} \text{ if }v\in \big [0, \frac{5}{2}\big ],\\
\frac{2(3-v)(v + 3)}{9}\text{ if }v\in\big [\frac{5}{2}, 3\big ].
\end{cases}
$$
Thus,
$S(W_{\bullet,\bullet}^{E_4};P)\le \frac{4}{3}< \frac{11}{6}$.
We get $\delta_P(S)=\frac{6}{11}$ for $P\in  E_4\backslash E_3$.
\par{\bf Step 6.} Suppose $P\in E_5$. Then $\tau(E_5)=2$ and the  Zariski decomposition of the divisor $-K_S-vE_5\sim C+2E_1+3E_2+4E_3+3E_4+(2-v)E_5+E_6+2E$ is given by:
 {\allowdisplaybreaks\begin{align*}
&&P(v)=
-K_S-vE_5-\frac{v}{4} (2E_1+4E_2+6E_3+5E_4+3E+2E_6)\text{ if }v\in[0,2].\\
&&
N(v)=
\frac{v}{4} (2E_1+4E_2+6E_3+5E_4+3E+2E_6)\text{ if }v\in[0,2].
\end{align*}}
Moreover, 
$$(P(v))^2=\frac{(2-v)(2+v)}{4}
\text{ and }P(v)\cdot E_5=\frac{v}{4}\text{ if }v\in[0,2].$$
Now we apply the computation from Section \ref{dP1-D6} (Step 1.) and get that $\delta_P(S)=\frac{3}{4}$ for $P\in E_5\backslash E_4$.
\par{\bf Step 7.} Suppose $P\in  E_6$.  There exist $(-1)$-curves and $(-2)$-curves  which form one of the following dual graphs:
\begin{figure}[h!]
    \centering
\includegraphics[width=15cm]{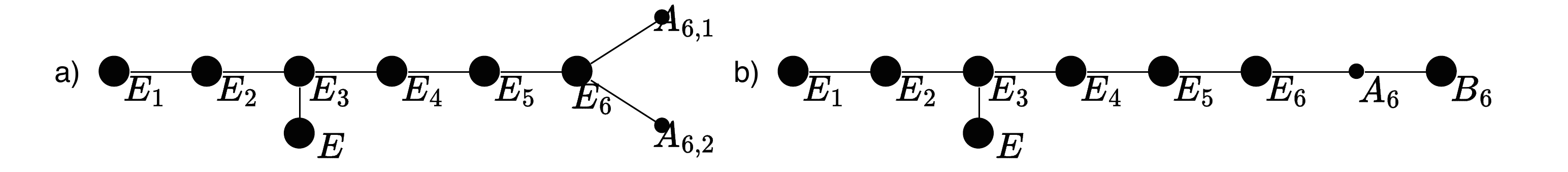}
    \caption{Dual graph: $(-K_S)^2=1$, $\mathbb{E}_7$ singularity, $\delta_P(S)=\frac{6}{5}$}
\end{figure}
\par
Then $\tau(E_6)=\frac{3}{2}$ and the  Zariski Decomposition of the divisor $-K_S-vE_6$ is:
   {
 {\allowdisplaybreaks\begin{align*}
\hspace*{-0.7cm}&{\text{\bf a). }} & P(v)=\begin{cases}
-K_S-vE_6-\frac{v}{3} (2E_1+4E_2+6E_3+5E_4+4E_5+3E)\text{ if }v\in[0,1],\\
-K_S-vE_6-\frac{v}{3} (2E_1+4E_2+6E_3+5E_4+4E_5+3E)-(v-1)(A_{6,1}+A_{6,2})\text{ if }v\in\big[1,\frac{3}{2}\big].
\end{cases} \\
\hspace*{-0.7cm}&&N(v)=\begin{cases}
\frac{v}{3} (2E_1+4E_2+6E_3+5E_4+4E_5+3E)\text{ if }v\in[0,1],\\
\frac{v}{3} (2E_1+4E_2+6E_3+5E_4+4E_5+3E)+(v-1)(A_{1,1}+A_{1,2})\text{ if }v\in\big[1,\frac{3}{2}\big].
\end{cases}\\
\hspace*{-0.7cm}&{\text{\bf b). }} & P(v)=\begin{cases}
-K_S-vE_6-\frac{v}{3} (2E_1+4E_2+6E_3+5E_4+4E_5+3E)\text{ if }v\in[0,1],\\
-K_S-vE_6-\frac{v}{3} (2E_1+4E_2+6E_3+5E_4+4E_5+3E)-(v-1)(2A_{6}+B_6)\text{ if }v\in\big[1,\frac{3}{2}\big].
\end{cases} \\
\hspace*{-0.7cm}&&N(v)=\begin{cases}
\frac{v}{3} (2E_1+4E_2+6E_3+5E_4+4E_5+3E)\text{ if }v\in[0,1],\\
\frac{v}{3} (2E_1+4E_2+6E_3+5E_4+4E_5+3E)+(v-1)(2A_{6}+B_6)\text{ if }v\in\big[1,\frac{3}{2}\big].
\end{cases}
\end{align*}}}
The Zariski Decomposition in part a). follows from 
 $$-K_S-vE_6\sim_{\DR} \Big(\frac{3}{2}-v\Big)E_6+\frac{1}{2}\Big(2E_1+4E_2+6E_3+5E_4+4E_5+3E+A_{6,1}+A_{6,2}\Big).$$
 A similar statement holds in other parts.
Moreover, 
$$(P(v))^2=\begin{cases}
1-\frac{2v^2}{3}  \text{ if }v\in[0,1],\\
\frac{(3-2v)^2}{3}  \text{ if }v\in\big [1, \frac{3}{2}\big ].
\end{cases}
P(v)\cdot E_1=\begin{cases}
\frac{2v}{3}  \text{ if }v\in[0,1],\\
2(1-\frac{2v}{3})  \text{ if }v\in\big [1, \frac{3}{2}\big ].
\end{cases}$$
Now we apply the computation from Section \ref{dP1-D6} (Step 2.) and get that $\delta_P(S)=\frac{6}{5}$ for $P\in E_6\backslash E_5$.
\\Thus, $\delta_{\mathcal{P}} (X)=\frac{3}{7}$.
 \end{proof}
\subsubsection{$\mathbb{E}_8$ singularity on Du Val Del Pezzo surfaces of degree $1$}\label{dP1-E_8}
\begin{lemma} Let $X$ be a singular del Pezzo surface of degree $1$ with an $\mathbb{E}_8$ singularity at point $\mathcal{P}$.  Let $\mathcal{C}$ be a~curve in the~pencil $|-K_X|$ that contains~$\mathcal{P}$. Then $\delta_{\mathcal{P}} (X)=\frac{3}{11}$.
\end{lemma}
 \begin{proof}
Let $S$ be the minimal resolution of singularities.  Then $S$ is a weak del Pezzo surface of degree $1$. Suppose $C$ is a strict transform of $\mathcal{C}$ on $S$ and $E$, $E_1$, $E_2$, $E_3$, $E_4$, $E_5$, $E_6$ and $E_7$ are the exceptional divisors with the intersection:
\begin{figure}[h!]
    \centering
\includegraphics[width=8cm]{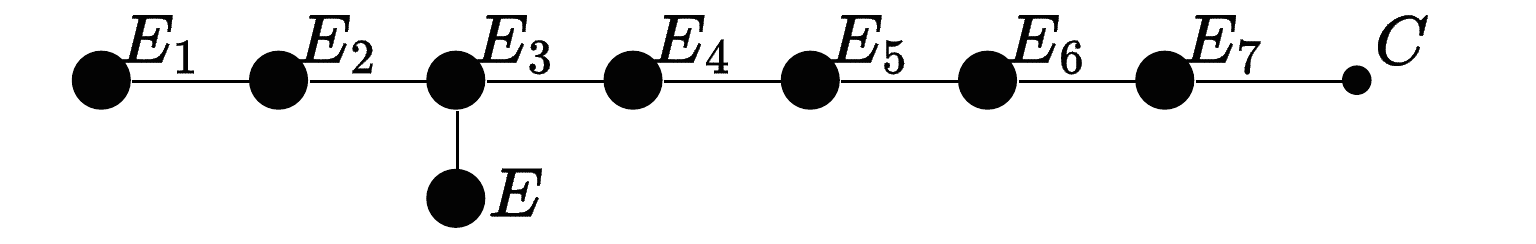}
    \caption{Dual graph: $(-K_S)^2=1$, $\mathbb{E}_8$ singularity}
\end{figure}
\par
We have $-K_S\sim C+2E_1+4E_2+6E_3+5E_4+4E_5+3E_6+2E_7+3E$.\\
 {\bf Step 1.} Suppose $P\in  E_3$. Then $\tau(E_3)=6$ and the   Zariski decomposition of the divisor $-K_S-vE_3\sim C+2E_1+4E_2+(6-v)E_3+5E_4+4E_5+3E_6+2E_7+3E$ is  the following:
{\small   {\allowdisplaybreaks\begin{align*}
\hspace*{-0.8cm}&& P(v)=\begin{cases}
-K_S-vE_3-\frac{v}{2}E-\frac{v}{3}(E_1+2E_2)-\frac{v}{5}(4E_4+3E_5+2E_6+E_7)\text{ if }v\in\big[0,5\big],\\
-K_S-vE_3-\frac{v}{2}E-\frac{v}{3}(E_1+2E_2)-(v-1)E_4-(v-2)E_5-(v-3)E_6-(v-4)E_7-(v-5)C
\text{ if }v\in[5,6].
\end{cases} \\
\hspace*{-0.8cm}&&\text{}N(v)=\begin{cases}
\frac{v}{2}E+\frac{v}{3}(E_1+2E_2)+\frac{v}{5}(4E_4+3E_5+2E_6+E_7)\text{ if }v\in\big[0,5\big],\\
\frac{v}{2}E+\frac{v}{3}(E_1+2E_2)+(v-1)E_4+(v-2)E_5+(v-3)E_6+(v-4)E_7+(v-5)C\text{ if }v\in[5,6].
\end{cases}
\end{align*}}}
Moreover, 
$$(P(v))^2=\begin{cases}
1-\frac{v^2}{30}  \text{ if }v\in[0,5],\\
\frac{(6-v)^2}{6}  \text{ if }v\in[5,6].
\end{cases}
P(v)\cdot E_3=\begin{cases}
\frac{v}{30}  \text{ if }v\in[0,5],\\
1-\frac{v}{6}  \text{ if }v\in[5,6].
\end{cases}$$
We have $S_{S} (E_3)=\frac{11}{3}$.
Thus, $\delta_P(S)\le \frac{3}{11}$ for $P\in E_3$. Moreover, if $P\in E_3\cap (E\cup E_2)$ if $P\in E_3\backslash (E\cup E_2)$   we have 
$$
h(v)\le\begin{cases}
\frac{41v^2}{1800} \text{ if }v\in [0,5],\\
\frac{(6 - v )(7v + 6)}{72} \text{ if }v\in[5,6].
\end{cases}
\text{ or }
h(v)\le\begin{cases}
\frac{49v^2}{1800} \text{ if }v\in [0,5],\\
\frac{(6 - v )(11v - 6)}{72} \text{ if }v\in[5,6].
\end{cases}
$$
Thus,
$S(W_{\bullet,\bullet}^{E_3};P)\le \frac{5}{2} <\frac{11}{3}$ or
$S(W_{\bullet,\bullet}^{E_3};P)\le 3 <\frac{11}{3}$.
We get $\delta_P(S)=\frac{3}{11}$ for $P\in E_3$.
\par{\bf Step 2.} Suppose $P\in  E_2$. Then $\tau(E_2)=4$ and the  Zariski decomposition of the divisor $-K_S-vE_2\sim C+2E_1+(4-v)E_2+6E_3+5E_4+4E_5+3E_6+2E_7+3E$ is:
{\footnotesize  {\allowdisplaybreaks\begin{align*}
\hspace*{-1.7cm}&&P(v)=\begin{cases}
-K_S-vE_2-\frac{v}{2}E_1-\frac{v}{7}(5E+10E_3+8E_4+6E_5+4E_6+2E_7)\text{ if }v\in\big[0,\frac{7}{2}\big],\\
-K_S-vE_2-\frac{v}{2}E_1-(v-1)E-(2v - 2)E_3-(2v - 3)E_4-(2v - 4)E_5-(2v - 5)E_6-(2v - 6)E_7-(2v - 7)C
\text{ if }v\in\big[\frac{7}{2},4\big].
\end{cases} 
\\
\hspace*{-1.7cm}&& N(v)=\begin{cases}
\frac{v}{2}E_1+\frac{v}{7}(5E+10E_3+8E_4+6E_5+4E_6+2E_7)\text{ if }v\in\big[0,\frac{7}{2}\big],\\
\frac{v}{2}E_1+(v-1)E+(2v - 2)E_3+(2v - 3)E_4+(2v - 4)E_5+(2v - 5)E_6+(2v - 6)E_7+(2v - 7)C\text{ if }v\in\big[\frac{7}{2},4\big].
\end{cases}
\end{align*}}}
Moreover, 
$$(P(v))^2=\begin{cases}
1-\frac{v^2}{14}  \text{ if }v\in\big [0,\frac{7}{2}\big ],\\
\frac{(4-v)^2}{2}  \text{ if }v\in\big [\frac{7}{2}, 4\big ].
\end{cases}
P(v)\cdot E_2=\begin{cases}
\frac{v}{14}  \text{ if }v\in\big [0,\frac{7}{2}\big ],\\
2-\frac{v}{2}  \text{ if }v\in\big [\frac{7}{2}, 4\big ].
\end{cases}$$
We have
$S_{S} (E_2)=\frac{5}{2}$.
Thus, $\delta_P(S)\le \frac{2}{5}$ for $P\in E_2$. Moreover,  if $P\in E_2\backslash E_3$   we have 
$$h(v)\le\begin{cases}
\frac{15v^2}{392} \text{ if }v\in \big [0, \frac{7}{2}\big ],\\
\frac{(4-v)(4+v)}{8}\text{ if }v\in\big [\frac{7}{2}, 4\big ].
\end{cases}
$$
Thus,
$S(W_{\bullet,\bullet}^{E_2};P)\le \frac{4}{3}< \frac{5}{2}$.
We get $\delta_P(S)=\frac{2}{5}$ for $P\in  E_2\backslash E_3$.
\par{\bf Step 3.} Suppose $P\in E_1$. Then $\tau(E_1)=2$ and the  Zariski decomposition of the divisor $-K_S-vE_1\sim C+(2-v)E_1+4E_2+6E_3+5E_4+4E_5+3E_6+2E_7+3E$ is given by:
 {\allowdisplaybreaks\begin{align*}
&&P(v)=
-K_S-vE_1-\frac{v}{4} (5E+7E_2+10E_3+8E_4+6E_5+4E_6+2E_7)\text{ if }v\in[0,2].\\
&&
N(v)=
\frac{v}{4} (5E+7E_2+10E_3+8E_4+6E_5+4E_6+2E_7)\text{ if }v\in[0,2].
\end{align*}}
Moreover, 
$$(P(v))^2=\frac{(2-v)(2+v)}{4}
\text{ and }P(v)\cdot E_1=\frac{v}{4}\text{ if }v\in[0,2].$$
Now we apply the computation from Section \ref{dP1-D6} (Step 1.) and get that $\delta_P(S)=\frac{3}{4}$ for $P\in E_1\backslash E_2$.
\par{\bf Step 4.} Suppose $P\in  E$. Then $\tau(E)=3$ and the  Zariski decomposition of the divisor $-K_S-vE\sim C+2E_1+4E_2+6E_3+5E_4+4E_5+3E_6+2E_7+(3-v)E$ is:
{\footnotesize  {\allowdisplaybreaks\begin{align*}
\hspace*{-1.5cm}&&P(v)=\begin{cases}
-K_S-vE-\frac{v}{8}(5E_1+10E_2+15E_3+12E_4+9E_5+6E_6+3E_7)\text{ if }v\in\big[0,\frac{8}{3}\big],\\
-K_S-vE-(v-1)(E_1+2E_2+3E_3)-(3v-4)E_4-(3v-5)E_5-(3v-6)E_6-(3v-7)E_7-(3v-8)C
\text{ if }v\in\big[\frac{8}{3},3\big].
\end{cases}\\
\hspace*{-1.5cm}&& N(v)=\begin{cases}
\frac{v}{8}(5E_1+10E_2+15E_3+12E_4+9E_5+6E_6+3E_7)\text{ if }v\in\big[0,\frac{8}{3}\big],\\
(v-1)(E_1+2E_2+3E_3)+(3v-4)E_4+(3v-5)E_5+(3v-6)E_6+(3v-7)E_7+(3v-8)C\text{ if }v\in\big[\frac{8}{3},3\big].
\end{cases}
\end{align*}}}
Moreover, 
$$(P(v))^2=\begin{cases}
1-\frac{v^2}{8}  \text{ if }v\in\big [0,\frac{8}{3}\big ],\\
(3-v)^2  \text{ if }v\in\big[\frac{8}{3},3\big].
\end{cases}
P(v)\cdot E=\begin{cases}
\frac{v}{8}  \text{ if }v\in\big [0,\frac{8}{3}\big ],\\
3-v  \text{ if }v\in\big[\frac{8}{3},3\big].
\end{cases}$$
We have
$S_{S} (E)=\frac{17}{9}$. Thus, $\delta_P(S)\le \frac{9}{17}$ for $P\in E$. Moreover,  if $P\in E\backslash E_3$   we have 
$$h(v)\le\begin{cases}
\frac{v^2}{128} \text{ if }v\in \big [0, \frac{8}{3}\big ],\\
\frac{(3-v)^2}{2}\text{ if }v\in \text{ if }v\in\big[\frac{8}{3},3\big].
\end{cases}
$$
Thus,
$S(W_{\bullet,\bullet}^{E};P)\le  \frac{1}{9}< \frac{17}{9}$.
We get $\delta_P(S)=\frac{9}{17}$ for $P\in  E\backslash E_3$.
 \par{\bf Step 5.} Suppose $P\in  E_4$. Then $\tau(E_4)=5$ and the   Zariski decomposition of the divisor $-K_S-vE_4\sim C+2E_1+4E_2+6E_3+(5-v)E_4+4E_5+3E_6+2E_7+3E$ is  the following:
{ \small  {\allowdisplaybreaks\begin{align*}
\hspace*{-0.5cm}&& P(v)=\begin{cases}
-K_S-vE_4-\frac{v}{5}(2E_1+4E_2+6E_3+3E)-\frac{v}{4}(3E_5+2E_6+E_7)\text{ if }v\in[0,4],\\
-K_S-vE_4-\frac{v}{5}(2E_1+4E_2+6E_3+3E)-(v-1)E_5-(v-2)E_6-(v-3)E_7-(v-4)C
\text{ if }v\in[4,5].
\end{cases} \\
\hspace*{-0.5cm}&&\text{}N(v)=\begin{cases}
\frac{v}{5}(2E_1+4E_2+6E_3+3E)+\frac{v}{4}(3E_5+2E_6+E_7)\text{ if }v\in[0,4],\\
\frac{v}{5}(2E_1+4E_2+6E_3+3E)+(v-1)E_5+(v-2)E_6+(v-3)E_7+(v-4)C\text{ if }v\in[4,5].
\end{cases}
\end{align*}}}
Moreover, 
$$(P(v))^2=\begin{cases}
1-\frac{v^2}{20}  \text{ if }v\in[0,4],\\
\frac{(5-v)^2}{5}  \text{ if }v\in[4,5].
\end{cases}
P(v)\cdot E_4=\begin{cases}
\frac{v}{20}  \text{ if }v\in[0,4],\\
1-\frac{v}{5}  \text{ if }v\in[4,5].
\end{cases}$$
We have
$S_{S} (E_4)= 3$.
Thus, $\delta_P(S)\le \frac{1}{3}$ for $P\in E_4$. Moreover, if $P\in E_4\backslash E_3$   we have 
$$
h(v)\le\begin{cases}
\frac{31v^2}{800} \text{ if }v\in [0,4],\\
\frac{(5 - v )(9v - 5)}{50} \text{ if }v\in[4,5].
\end{cases}
$$
Thus,
$S(W_{\bullet,\bullet}^{E_4};P)\le \frac{7}{3} <3$.
We get $\delta_P(S)=3$ for $P\in E_4\backslash E_3$.
\par{\bf Step 6.} Suppose $P\in  E_5$. Then $\tau(E_5)=4$ and the   Zariski decomposition of the divisor $-K_S-vE_5\sim C+2E_1+4E_2+6E_3+5E_4+(4-v)E_5+3E_6+2E_7+3E$ is  the following:
{ {\allowdisplaybreaks\begin{align*}
\hspace*{-0.5cm}&& P(v)=\begin{cases}
-K_S-vE_5-\frac{v}{4} (2E_1+4E_2+6E_3+5E_4+3E)-\frac{v}{3} (2E_6+E_7)\text{ if }v\in[0,3],\\
-K_S-vE_5-\frac{v}{4} (2E_1+4E_2+6E_3+5E_4+3E)-(v-1)E_6-(v-2)E_7-(v-3)C\text{ if }v\in[3,4].
\end{cases} \\
\hspace*{-0.8cm}&&N(v)=\begin{cases}
\frac{v}{4} (2E_1+4E_2+6E_3+5E_4+3E)-\frac{v}{3} (2E_6+E_7)\text{ if }v\in[0,3],\\
\frac{v}{4} (2E_1+4E_2+6E_3+5E_4+3E)+(v-1)E_6+(v-2)E_7+(v-3)C\text{ if }v\in[3,4].
\end{cases}
\end{align*}}}
Moreover, 
$$(P(v))^2=\begin{cases}
1-\frac{v^2}{12}  \text{ if }v\in[0,3],\\
\frac{(4-v)^2}{4}  \text{ if }v\in[3,4].
\end{cases}
P(v)\cdot E_5=\begin{cases}
\frac{v}{12}  \text{ if }v\in[0,3],\\
1-\frac{v}{4}  \text{ if }v\in[3,4].
\end{cases}$$
Now we apply the computation from Section \ref{dP1-E7} (Step 1.) and get that $\delta_P(S)=\frac{3}{7}$ for $P\in E_5\backslash E_4$.
 \par{\bf Step 7.} Suppose $P\in  E_6$. Then $\tau(E_6)=3$ and the   Zariski decomposition of the divisor $-K_S-vE_6\sim C+2E_1+4E_2+6E_3+5E_4+4E_5+(3-v)E_6+2E_7+3E$ is  the following:
 {\allowdisplaybreaks\begin{align*}
\hspace*{-0.2cm}&& P(v)=\begin{cases}
-K_S-vE_6-\frac{v}{3} (2E_1+4E_2+6E_3+5E_4+4E_5+3E)-\frac{v}{2}E_7\text{ if }v\in[0,2],\\
-K_S-vE_6-\frac{v}{3} (2E_1+4E_2+6E_3+5E_4+4E_5+3E)-(v-1)E_7-(v-2)C\text{ if }v\in[2,3].
\end{cases} \\
\hspace*{-0.2cm}&&\text{}N(v)=\begin{cases}
\frac{v}{3} (2E_1+4E_2+6E_3+5E_4+4E_5+3E)+\frac{v}{2}E_7\text{ if }v\in[0,2],\\
\frac{v}{3} (2E_1+4E_2+6E_3+5E_4+4E_5+3E)+(v-1)E_7+(v-2)C\text{ if }v\in[2,3].
\end{cases}
\end{align*}}
Moreover, 
$$(P(v))^2=\begin{cases}
1-\frac{v^2}{6}  \text{ if }v\in[0,2],\\
\frac{(3-v)^2}{3}  \text{ if }v\in[2,3].
\end{cases}
P(v)\cdot E_6=\begin{cases}
\frac{v}{6}  \text{ if }v\in[0,2],\\
1-\frac{v}{6}  \text{ if }v\in[2,3].
\end{cases}$$
Now we apply the computation from Section \ref{dP1-D8} (Step 1.) and get that $\delta_P(S)=\frac{3}{5}$ for $P\in E_6\backslash E_5$.
\par{\bf Step 8.} Suppose $P\in  E_7$.  Then $\tau(E_7)=2$ and the  Zariski decomposition of the divisor $-K_S-vE_7\sim C+2E_1+4E_2+6E_3+5E_4+4E_5+3E_6+(2-v)E_7+3E$ is:
 {\allowdisplaybreaks\begin{align*}
&& P(v)=\begin{cases}
-K_S-vE_7-\frac{v}{2} (2E_1+4E_2+6E_3+5E_4+4E_5+3E_6+3E)\text{ if }v\in[0,1],\\
-K_S-vE_7-\frac{v}{2} (2E_1+4E_2+6E_3+5E_4+4E_5+3E_6+3E)-(v-1)C\text{ if }v\in[1,2].
\end{cases} \\
&& N(v)=\begin{cases}
\frac{v}{2} (2E_1+4E_2+6E_3+5E_4+4E_5+3E_6+3E)\text{ if }v\in[0,1],\\
\frac{v}{2} (2E_1+4E_2+6E_3+5E_4+4E_5+3E_6+3E)+(v-1)C\text{ if }v\in[1,2].
\end{cases}
\end{align*}}
Moreover, 
$$(P(v))^2=\begin{cases}
1-\frac{v^2}{2}  \text{ if }v\in[0,1],\\
\frac{(2-v)^2}{2}  \text{ if }v\in[1,2].
\end{cases}
P(v)\cdot E_7=\begin{cases}
\frac{v}{2}  \text{ if }v\in[0,1],\\
1-\frac{v}{2}  \text{ if }v\in[1,2].
\end{cases}$$
Now we apply the computation from Section \ref{dP1-D4} (Step 1.) and get that $\delta_P(S)=1$ for $P\in E_7\backslash E_6$.
\\Thus, $\delta_{\mathcal{P}} (X)=\frac{3}{11}$.
 \printbibliography
\end{proof}

\end{document}